\documentclass{article}
\usepackage{graphicx} % Required for inserting images

\title{Explicit evaluation of the $q$-Stokes matrices for certain  confluent hypergeometric $q$-difference equations}
\author{}

\usepackage{amsmath}
\usepackage{amsfonts}
\usepackage{amssymb}
\usepackage{amsthm}
\usepackage{hyperref}
\usepackage{tikz-cd}
\usepackage{mathtools}

\usepackage{yhmath}

\usepackage{scalerel,stackengine}
\stackMath
\newcommand\reallywidehat[1]{%
\savestack{\tmpbox}{\stretchto{%
  \scaleto{%
    \scalerel*[\widthof{\ensuremath{#1}}]{\kern-.6pt\bigwedge\kern-.6pt}%
    {\rule[-\textheight/2]{1ex}{\textheight}}%WIDTH-LIMITED BIG WEDGE
  }{\textheight}% 
}{1.5ex}}%
\stackon[1pt]{#1}{\tmpbox}%
}
\usepackage{scalerel}

\usepackage{scalerel}

\usepackage{scalerel}

\newcommand\reallywidehatt[1]{\arraycolsep=0pt\relax%
\begin{array}{c}
\stretchto{
  \scaleto{
    \scalerel*[\widthof{\ensuremath{#1}}]{\kern-.5pt\bigwedge\kern-.5pt}
    {\rule[-\textheight/2]{1ex}{\textheight}} %WIDTH-LIMITED BIG WEDGE
  }{\textheight} % 
}{0.7ex}\\           % THIS SQUEEZES THE WEDGE TO 0.5ex HEIGHT
#1\\                 % THIS STACKS THE WEDGE ATOP THE ARGUMENT
\rule{-1ex}{0ex}
\end{array}
}

\usepackage{graphicx}
\graphicspath{ {./images/} }

\usepackage{xcolor}

\definecolor{darkgreen}{rgb}{0.0, 0.5, 0.0}

\usepackage{hyperref}
\usepackage{amsmath}
\usepackage{amsthm}
\usepackage{amssymb}
\usepackage[top=2.8cm,bottom=2.0cm,left=2.5cm,right=1.8cm]{geometry}

\usepackage{amsthm} % 如果要使用proof语句就必须要引入这个包

\newtheorem{thm}{Theorem}[section]
\newtheorem{lemma}[thm]{Lemma}
\newtheorem{prop}[thm]{Proposition}
\newtheorem{defn}[thm]{Definition}
\newtheorem{rmk}[thm]{Remark}

\newtheorem{cor}[thm]{Corollary}

\author{Jinghong Lin, Yiming Ma and Xiaomeng Xu}
\date{}

\newcommand{\Addresses}{{
  \bigskip
  \footnotesize
\noindent \textsc{School of Mathematical Sciences, Peking University, Beijing 100871, China}\par\nopagebreak
  \textit{E-mail address}: \texttt{linjinghong@pku.edu.cn}
}\\
\\
\footnotesize
\noindent \textsc{School of Mathematical Sciences, Peking University, Beijing 100871, China}\par\nopagebreak
  \textit{E-mail address}: \texttt{ymma.pku@icloud.com}
\\
\\
\footnotesize
\noindent \textsc{
School of Mathematical Sciences \& Beijing International Center
for Mathematical Research, Peking University, Beijing 100871, China}\par\nopagebreak
  \textit{E-mail address}: \texttt{xxu@bicmr.pku.edu.cn}
}

\begin{document}

\maketitle

\abstract{We prove a connection formula for the  basic hypergeomtric function ${}_n\varphi_{n-1}\left(
			a_1,...,a_{n-1},0;
			b_1,...,b_{n-1}
		; q, z\right)$ by  using  the $q$-Borel resummation. As an application, we compute  $q$-Stokes matrices of a special  confluent hypergeometric $q$-difference  system with an irregular singularity. We show that by letting $q\rightarrow 1$, the $q$-Stokes matrices recover the known expressions of the Stokes matrices of the corresponding confluent hypergeometric differential system.}
% ...........................................
% INCLUDE the necessary ingredients here.
% .................................s..........
\section{Introduction}
Set $q\in\mathbb{R}$ with $0<q<1$.    Take the Lie algebra ${\mathfrak {gl}_n}$ over the field of complex numbers with a positive integer $n$. Consider the confluent hypergeometric $q$-difference  system
\begin{equation}\label{u+A/z}
	D_qF_q(z):=\frac{\sigma_qF_q(z)-F_q(z)}{qz-z}=\left(u+\frac{A}{z}\right)F_q(z),
\end{equation}
where $\sigma_q F(z)=F(q z)$, $u=\mathrm{diag}\left(u_1,...,u_{n-1},u_n\right)$ and $A=(a_{ij})\in {\mathfrak {gl}_n}$.  The infinity $z=\infty$ is a  Fuchsian singularity  of    \eqref{u+A/z} if the diagonal matrix $u$ is invertible. One can use the method of isomonodromy deformation to study the connection problem of the equation \eqref{u+A/z} (see e.g., \cite{jimbo1996q, MR2652472, MR2642627, ohyama2021q, MR4245859}). It is parallel to the corresponding differential system case studied in \cite{MR688949, MR533348, TX, Xu}. In this method,    the monodromy data of the system \eqref{u+A/z} with $u=\mathrm{diag}(0,...,0,1)$  are needed.  

In this paper, we focus on the irregular      case of the form
\begin{equation}\label{qdes}
    D_qF_q(z)=\left(E_{n n}+\frac{A}{z}\right)F_q(z),
\end{equation}
with  $E_{n n}=\mathrm{diag}(0,...,0,1)$.    
Denote by $A^{\left(k\right)}$ the upper left $k\times k$ submatrix of $A=\left(a_{i j}\right)\in \mathfrak{gl}_n$, and by $\lambda^{\left(k\right)}_1,...,\lambda^{\left(k\right)}_{k}$ the $k$ eigenvalues of $A^{\left(k\right)}$.   Set  the complex numbers
    \begin{align*}
	&\mu_i:=\mathrm{log}_q\left(1-(1-q)\lambda_i^{(n)}\right),\quad 1\leq i\leq n,\\
&\nu_i:=\mathrm{log}_q\left(1-(1-q)\lambda_i^{(n-1)}\right),\quad 1\leq i\leq n-1,
	\end{align*}
where we take the principal branch of the logarithmic function.

Under the following assumption
\begin{equation}
    \nu_i- \nu_j\notin \mathbb{Z}, \ \text{for all} \ 1\leq i \neq j \leq n-1, \label{eq:vivj}
\end{equation}
 the system \eqref{qdes}  has   a unique formal fundamental solution $\hat{F}_q$ around $z=\infty$. The theory of $q$-Borel resummation states that for  $\lambda\in \mathbb{C}^{*}\backslash -\frac{1-q}{q^{\nu_j}}q^{\mathbb{Z}}$,  $1\leq j \leq n-1$,    there exists a meromorphic fundamental solution $F^{(\infty)}_q(z,\lambda)$ which is  asymptotic to  $\hat{F}_q$  of  $q$-Gevrey order  $1$ along the direction $\lambda$.  For different directions, these solutions are in general different
(which reflects the $q$-Stokes phenomenon), and the transition between them can be measured  by  a  $q$-Stokes
matrix (see e.g.,  \cite{marotte2000multisommabilite, ramis1992growth, ramis2013local, MR1952546, zhang1999developpements, zhang2002sommation, zhang2005remarks} for the detailed study of $q$-Borel resummation, $q$-Gevrey order and so on).   The main result of this paper is the following theorem, which gives the
$q$-Stokes matrix   of the system  \eqref{qdes}  explicitly. See Section \ref{sec:qdes} for a proof.

\begin{thm}\label{intro: qstokes}
   If the following condition holds for the matrix $A$:
\begin{equation*}
\begin{aligned}
    &   \mu_i-\mu_j\notin \mathbb{Z}, \  &&\text{for all}  \  1\leq i \neq j \leq  n,   \\
    & \nu_i-\nu_j\notin  \mathbb{Z},  \  &&\text{for all}  \  1\leq i\neq  j \leq  n-1,   \\
    &\lambda^{\left(n-1\right)}_j   \neq   \lambda^{\left(n-2\right)}_l,  \  &&\text{for all}  \  1\leq j \leq  n-1, 1\leq l \leq n-2,
\end{aligned}
\end{equation*}
and the following condition holds for   $z,\lambda,\mu$:
\begin{equation*}
     \lambda,\mu  \in \mathbb{C}^*,  \  \mu\notin \lambda q^{\mathbb{Z}},   \ \lambda,\mu  \notin -\frac{1-q}{q^{\nu_l}}q^{\mathbb{Z}}, \  \text{for} \   1\leq l \leq n-1,       \  \text{and} \  z\in \mathbb{C}^*, \  z\notin - \frac{1}{\mu}q^{\mathbb{Z}}, \  z\notin  - \frac{1}{\lambda}q^{\mathbb{Z}},
\end{equation*}
    then the $q$-Stokes matrix $S_q\left(z,\lambda,\mu;E_{nn}, A\right)$ of the system \eqref{qdes} takes the block matrix form
\begin{equation*}
S_q\left(z,\lambda,\mu;E_{nn}, A\right)=\left(\begin{array}{cc}
    \mathrm{Id}_{n-1} & 0  \\
    b_q & 1
  \end{array}\right),
  \end{equation*}
where $b_{q}=\left(\left(b_q\right)_1,...,\left(b_{q}\right)_{n-1}\right)$ with
\begin{align*}
\left(b_{q}\right)_k=&\sum_{j=1}^{n-1} \text{ \scriptsize $ \frac{(-1)^{n+k}\Delta_{1,...,n-1}^{1,..., n-2,n}\left(A-\lambda_j^{\left(n-1\right)}\cdot \mathrm{Id}_n\right) \Delta^{1,...,\hat{k},...,n-1}_{1,...,n-2}   \left(A-\lambda^{\left(n-1\right)}_j\cdot \mathrm{Id}_n\right)}{{\prod_{l=1,l\ne j}^{n-1}\left(\lambda^{\left(n-1\right)}_l-\lambda^{\left(n-1\right)}_j\right)\prod_{l=1}^{n-2}\left(\lambda^{\left(n-2\right)}_l-\lambda^{\left(n-1\right)}_j\right)} } $ }  \\
    &  \sum_{s=1}^n
    \frac{  \prod_{l=1}^{n-1} ( q^{1+\mu_{s}-\nu_{l}}  ;q)_\infty}{\prod_{l=1,l\neq   s}^{n}   (q^{1+\mu_{s}-\mu_{l}}   ;q)_\infty}   \frac{\prod_{l=1,l\neq s}^{n}\theta_q((1-q)q^{1-\mu_l}/\lambda)}{\prod_{l=1}^{n-1}\theta_q((1-q)q^{1-\nu_l}/\lambda)}   
 \frac{\theta_q\left(  q^{ \mu_s-\nu_n  }  \lambda z\right)}{\theta_q(\lambda  z )}     \\
  &\times   \frac{ \prod_{ l=1,l\neq  s  }^n  ( q^{1+\mu_l-{\nu}_j} ;q )_{\infty} \prod_{l=1,l\neq  j}^{n-1}  ( q^{ \nu_l-\mu_s  }  ;q)_{\infty} }{ \prod_{l=1,l\neq  j}^{n-1} ( q^{ 1+\nu_l-\nu_j  } ;q)_{\infty}\prod_{l=1,l\neq  s}^n ( q^{\mu_{ l }-\mu_s} ;q)_{\infty} }        \frac{ \theta_q\left( q^{ \nu_j-\mu_s-1  }    {  \mu z }  \right)  }{ \theta_q\left(  \mu z   \right)   }  \frac{ \theta_q \left(  {{ q^{ \mu_s       } \mu   }}/{(1-q)} \right)  }{ \theta_q\left(   {{  q^{    \nu_j  -1  } \mu  }}/({1-q})  \right)  }    \frac{z^{\nu_j}}{z^{\nu_n+1}} .
\end{align*} 
Here  the definitions of the $q$-Pochhammer symbol $(a;q)_{\infty}$  and the $q$-theta function $\theta_q(z)$ are given in Section \ref{notation}   and    $\Delta_{1,...,n-1}^{1,..., n-2,n}\left(A-\lambda_j^{\left(n-1\right)}\cdot \mathrm{Id}_n\right)$  is  {a}  minor of the matrix $A-\lambda^{\left(n-1\right)}_j\cdot \mathrm{Id}_n$  introduced in Section \ref{sec:diag}.
\end{thm}

Under the nonresonant condition
\begin{equation}\label{qnreso}
    \mu_i-\mu_j\notin \mathbb{Z}_{>0}, \ \text{for all} \ 1\leq i,j \leq n,
\end{equation}
the system \eqref{qdes}  has  a convergent fundamental solution $F^{(0)}_q(z)$  around $z=0$.  The transition between  $F^{(0)}_q(z)$  and  $F^{(\infty)}_q(z,\lambda)$   is measured by some  connection  matrix given   explicitly in Theorem  \ref{qstokes matrix}.  

The computation of the $q$-Stokes matrix and the  connection  matrix
  mainly relies on the following   connection formula of  the   confluent hypergeometric $q$-difference   equation of the form
\begin{equation}\label{intro:conflu hyper eq}
      z\prod_{l=1}^{n-1}(1-a_l \sigma_q)y(z)=(1-\sigma_q)\prod_{l=1}^{n-1}\left(1-\frac{b_l}{q}\sigma_q\right)y(z),
 \end{equation}
 which   expresses  
the solutions  of  \eqref{intro:conflu hyper eq}  around the origin      as a sum of the solutions  of  \eqref{intro:conflu hyper eq}  around infinity         with     pseudo-constant  coefficients (see Definition \ref{def:pseudo}). Here   $\sigma_q$ is the $q$-shift operator defined by $\sigma_q y(z)=y(q z)$.
See Section \ref{seclab conne formu eq} for a proof of the following theorem.

 \begin{thm}\label{intro:our connection formula}
For $a_1,a_2,...,a_{n-1},b_1,b_2,...,b_{n-1}$ such that         
\begin{equation*}
    a_i\neq 0, \ a_i/a_j\notin q^{\mathbb{Z}}, \ \text{for}\ 1\leq i\neq j \leq n-1, \  \text{and} \ b_l\notin \{0\}\cup q^{-\mathbb{N}}, \  \text{for}  \ 1\leq l \leq n-1,
\end{equation*}
and  for  $\lambda,z$  such that     
\begin{equation*}
    \lambda\in \mathbb{C}^*,\ \lambda\notin -a_i q^{\mathbb{Z}}, \  \text{for} \   1\leq i \leq n-1, \  \text{and} \  z\in \mathbb{C}^*, \  z\notin q^{\mathbb{Z}}, \  z\notin  -\frac{\prod_{l=1}^{n-1} b_l}{ \lambda \prod_{l=1}^{n-1} a_l}q^{\mathbb{Z}},
\end{equation*}
we have
\begin{align*}
    &{}_n\varphi_{n-1}\left(
			a_1,...,a_{n-1},0;
			b_1,...,b_{n-1}
		; q, z\right)\\
  =&\sum_{i=1}^{n-1}
		\frac{(a_1,...,\widehat{a_{{i}} },...,a_{n-1},b_1/a_i,...,b_{n-1}/a_i;q)_\infty}{(b_1,...,b_{n-1},a_1/a_i,...,\widehat{a_{{i}}/a_i},...,a_{n-1}/a_i;q)_\infty}\frac{\theta_q(-a_i z)}{ \theta_q(-z) } z^{\log_qa_i}  \cdot {f}^{(\infty)}_q(z,\lambda)_{i}\\
  &+\frac{(a_1,...,a_{n-1};q)_\infty}{(b_1,...,b_{n-1};q)_\infty}
		\frac{\prod_{s=1}^{n-1}\theta_q\big( {\prod_{l=1}^{n-1}a_l} \big / \big(  {\lambda \prod_{l=1,l\neq s}^{n-1}b_l}  \big)    \big)}{\prod_{s=1}^{n-1}\theta_q\left({a_s\prod_{l=1}^{n-1}a_l  }\big /\big({\lambda \prod_{l=1}^{n-1}b_l  } \big)  \right)}
		\frac{\theta_q\left( {\lambda z\prod_{l=1}^{n-1}b_l   }\big /{\prod_{l=1}^{n-1}a_l}\right)}{\theta_q\left( \lambda z \right)}
\frac{\prod_{l=1}^{n-1}z^{\log_qb_l}   }{ \prod_{l=1}^{n-1}z^{\log_qa_l}  }	{f}^{(\infty)}_q(z)_{n}	,
\end{align*}
 where the  basic hypergeometric series ${}_n\varphi_{n-1}\left(
			a_1,...,a_{n-1},0;
			b_1,...,b_{n-1}
		; q, z\right)$ (see Definition \ref{def:basic f})   is  a  solution  of \eqref{intro:conflu hyper eq} around the origin,   and  $
 {f}_q^{(\infty)}(z,\lambda)_{1} ,
...  ,
 {f}_q^{(\infty)}(z,\lambda)_{n-1}  ,
{f}_q^{(\infty)}(z)_{n}$ is a fundamental system  of meromorphic  solutions of \eqref{intro:conflu hyper eq}      around infinity   (see Proposition \ref{fqlambda}).  Here    $\widehat{a_{i}}$   and   $\widehat{a_{{i}}/a_i}$   mean the corresponding $i$-th terms  $a_i$  and  $a_{{i}}/a_i$  are skipped. 
\end{thm}

At the end of this article, we show that Theorem \ref{intro: qstokes} and Theorem \ref{intro:our connection formula}  recover the   corresponding   Stokes matrices  and   connection formula   in the differential case  as $q\rightarrow 1$.  See Section \ref{qlimit}.

The connection problem of the basic hypergeometric series ${}_n\varphi_m(a_1, a_2, ..., a_n ;b_1, ..., b_m;q,z)$  (see Definition \ref{def:basic f})  can be traced back to Heine (see e.g., \cite{Heine1846, Heine1847, MR204726}), and after that
\begin{itemize}
    \item  Thomae (see \cite{thomae1869beitrage}) and  Watson (see \cite{watson1910continuation})     gave the connection formula  of ${}_2\varphi_1(a_1,a_2;b_1;q,z)$ if   $a_1,a_2,b_1\neq 0$.
    \item  The connection formula of ${}_n\varphi_{n-1}(a_1,...,a_n;b_1,...,b_{n-1};q,z)$ was shown by Thomae   (see  \cite{thomae1870series})  when $a_i,b_j\neq 0$, for $1\leq i \leq n$,  $1\leq j \leq n-1$,  and  equivalent versions can be found in \cite{gasper2004basic, MR201688}.
\end{itemize}
However, if one of the conditions: $m=n-1$  or  $a_i,b_j\neq 0$  for $1\leq i \leq n$,  $1\leq j \leq m$,  does not hold,  the connection problem involves  divergent series  and  the  $q$-Borel resummation is required.  In these cases,
\begin{itemize}
    \item     Zhang obtained the connection formula of  ${}_2f_0(a_1,a_2;-;\lambda;q,z)$ (see Definition \ref{def:nfn-2} for the notation)      in \cite{zhang2002sommation},    ${}_2\varphi_1(0,0;b_1;q,z)$,  ${}_0\varphi_1(-;b_1;q,z)$ in \cite{zhang2003fonctions},  ${}_1\varphi_1(a_1;b_1;q,z)$,       ${}_2\varphi_1(a_1,0;b_1;q,z)$  in       \cite{ zhang2005remarks}    and so on.
        \item Morita obtained the   connection formula of  ${}_1\varphi_1(0;b_1;q,z)$ in \cite{morita2011connection},  ${}_2\varphi_1(a_1,a_2;0;q,z)$ in \cite{morita2013connection},   ${}_2\varphi_0(0,0;-;q,z)  $, ${}_0\varphi_1(-;0;q,z)$ in \cite{morita2014stokes}, ${}_3\varphi_1(a_1,a_2,a_3;b_1;q,z)$ in \cite{morita2014stokes1} and so on.
        \item Ohyama obtained the connection formula of   ${}_1\varphi_1(0;b_1;q,z)$ in  \cite{ohyama2016q},    ${}_r\varphi_{r-1}(0,...,0;b_1,...,b_{r-1};q,z)$  in   \cite{ohyama2017connection},  ${}_n\varphi_m(a_1, a_2, ..., a_n ;b_1, ..., b_m;q,z)$  with   $a_l\neq 0$,  $1\leq l \leq n$, for $n\geq m+1$   in \cite{ohyama2021q} (with Zhang) and so on.
        \item Adachi obtained the connection formula of ${ }_n f_{m}(a_1,...,a_n ; b_1,...,b_m ; \lambda ; q, z)$  (see Definition \ref{def:nfn-2})   for $n\geq m+2$  in \cite{adachi2019q} and so on.

        \item 
This paper derives the connection formula of ${}_n\varphi_{n-1}(a_1,...,a_{n-1},0;b_1,...,b_{n-1};q,z)$ with $a_l,b_l\neq 0$, for $1\leq l \leq n-1$. For $n=2$, it recovers the connection formula of $_2\varphi_1(a_1,0;b_1;q,z)$ in Zhang \cite{ zhang2005remarks}. 
\end{itemize}
There may be other works missed in the above list. 

After 1990s, Ramis,  Sauloy  and Zhang   are constructing modern theory
on $q$-asymptotics (see \cite{ramis2013local})  and the application to global analysis on $q$-Painlev\'e equations  leads to   the Riemann-Hilbert correspondence in modern sense(see e.g., \cite{MR4245859, ohyama2021q}).

\vspace{2mm}
The organization of the paper is as follows. Section \ref{sec:equation case} recalls the basic knowledge of basic hypergeometric series,  the theory of $q$-Borel resummation and  derive the connection formula of the equation  \eqref{intro:conflu hyper eq}.  Section \ref{main section}  computes explicitly the $q$-Stokes matrix of the system \eqref{qdes}. Section \ref{qlimit} studies   the behavior of   the connection formula  and the $q$-Stokes matrix   as $q\rightarrow 1$.

\subsection*{Acknowledgments}
\noindent
We would like to thank Qian Tang for the discussion and valuable feedback on the paper. The authors are supported by  the  National Key Research and Development Program of China (No. 2020YFE0204200),  by   the National Key Research and Development Program of China (No. 2021YFA1002000) and by the National Natural Science Foundation of China (No. 12171006).

\section{Connection Formula of the Equation \eqref{intro:conflu hyper eq} }\label{sec:equation case}
This section derives explicitly the connection formula  of the equation  \eqref{intro:conflu hyper eq}. In particular, Section \ref{notation} gives the basic notations of $q$-calculus. Section \ref{hyper series} 
introduces the basic hypergeometric series. Section \ref{sec:q-borel sum} overviews the $q$-Borel resummation. Section \ref{seclab conne formu eq} then {gives} the explicit connection   formula  of the equation   \eqref{intro:conflu hyper eq} via the confluence method. 
\subsection{Basic Notations}\label{notation}
        We recall some basic notations of $q$-calculus with $0<q<1$ (referring to \cite{gasper2004basic} for more details): 
        \begin{itemize}
            \item The $q$-Pochhammer symbol is 
            \begin{equation*}
          (a ; q)_n:=  \left\{
          \begin{array}{lr}
             1,   & \text{if} \ \ n=0,       \\
\prod_{l=0}^{n-1} (1-aq^{l}),         & \text{if} \ \ n \geq 1.
             \end{array}
\right.
\end{equation*}
Moreover, $(a ; q)_{\infty}:=\prod_{l=0}^{\infty}\left(1-a q^l\right)$ and 
\begin{equation*}
    \left(a_1, a_2, \ldots, a_m ; q\right)_{\infty}:=\left(a_1 ; q\right)_{\infty}\left(a_2 ; q\right)_{\infty} \cdots\left(a_m ; q\right)_{\infty}.
\end{equation*}
\item  The $q$-theta function is
\begin{equation*}
    \theta_q(z):=\sum_{n \in \mathbb{Z}} q^{\frac{n(n-1)}{2}} z^n, \quad  z \in \mathbb{C}^*.
\end{equation*}
It satisfies
\begin{equation}
\label{theta(q z)}
\theta_q(q^n z)=z^{-n}q^{-\frac{n(n-1)}{2}}\theta_q(z)=\theta_q\left(\frac{1}{q^{n-1}z}\right),\quad  n\in \mathbb{Z},
\end{equation}
and the Jacobi's triple product identity
\begin{equation}
\label{Jacobi's triple product identity}
\theta_q(z)=(q,-z,-q / z ; q)_{\infty}.
\end{equation} 
\item  The $q$-exponential function is  
\begin{equation*}
    e_q(z):=\frac{1}{\left((1-q)z;q\right)_{\infty}},
\end{equation*}
with the Taylor expansion
\begin{equation}\label{eq:taylor}
    e_q(z)=\sum_{k=0}^{\infty} \frac{(1-q)^k}{(q;q)_k}z^k, \quad  0 \leq  | z | <1.
\end{equation}
\item The $q$-Gamma function $\Gamma_q(z)$  is
\begin{equation*}
    \Gamma_q(z):=\frac{(q ; q)_{\infty}}{\left(q^z ; q\right)_{\infty}}(1-q)^{1-z}, \quad z\in \mathbb{C}\backslash \{0,-1,-2,...\},
\end{equation*}
{where $q^z$ and $(1-q)^{1-z}$ take the principal values}.
        \end{itemize}

\subsection{Basic Hypergeometric Series}\label{hyper series}
\begin{defn}\label{def:basic f}
    A basic hypergeometric series associated to   $\boldsymbol{a}=(a_1,a_2,...,a_n)$ and  $\boldsymbol{b}=(b_1,b_2,...,b_m)$ is
\begin{equation*}
    {}_n\varphi_m\left(\begin{array}{c}
a_1, a_2, ..., a_n \\
b_1, ..., b_m
\end{array} ; q, z\right)
:=
\sum_{k=0}^{\infty} \frac{ \prod_{l=1}^n    \left(a_l ; q\right)_k}{(q ; q)_k\prod_{l=1}^m \left(b_l ; q\right)_k }\left[(-1)^k q^{\frac{k(k-1)}{2}}\right]^{1+m-n} z^k,
\end{equation*}
for  $b_1,...,b_m\notin q^{-\mathbb{N}  }$.
{ We also denote it by ${}_n\varphi_m(a_1, a_2, ..., a_n ;b_1, ..., b_m;q,z)$, or simply ${}_n\varphi_m(\boldsymbol{a};\boldsymbol{b};q,z)$.}
\end{defn}
 The  basic hypergeometric series   is a $q$-analog of the generalized hypergeometric series defined as follows. 
\begin{defn}
 A   generalized hypergeometric series  for    $\boldsymbol{\alpha}=(\alpha_1,...,\alpha_n)$ and  $\boldsymbol{\beta}=(\beta_1,...,\beta_m)$   is  defined as  
{\begin{equation*}
    {}_nF_m\left(\begin{array}{c}
\alpha_1, \alpha_2, ..., \alpha_n \\
\beta_1, ..., \beta_m
\end{array} ;  z\right)
:=
\sum_{k=0}^{\infty}\frac{\prod_{l=1}^n\left(\alpha_l\right)_k}{\prod_{l=1}^m\left(\beta_l\right)_k }\frac{z^k}{k!},
\end{equation*}}
where $\left(\alpha\right)_0=1$ and $\left(\alpha\right)_k=\alpha \left(\alpha+1\right) \cdots \left(\alpha+k-1\right)$, for $k\geq 1$  and  $\beta_l\notin {\mathbb{Z}_{\leq 0}}$,  for $1\leq l \leq m$. We also denote it by ${}_nF_m(\alpha_1, \alpha_2, ..., \alpha_n ;\beta_1, ..., \beta_m;z)$, or simply ${}_nF_m(\boldsymbol{\alpha};\boldsymbol{\beta};z)$. 
 \end{defn}

 \begin{prop}[see e.g., {\cite[Chapter 16, Chapter 17]{OLBC}}]
    The radius of convergence  { of}  ${}_n\varphi_m\left(
\boldsymbol{a};
 \boldsymbol{b}
 ; q, z\right)$ and ${}_{n}F_{m}\left(\boldsymbol{\alpha};\boldsymbol{\beta};z\right)$  is $0$, $1$ and $\infty$, {for} $n\ge m+2$, $n=m+1$ and $n\leq m$ respectively.  
%  Moreover,    ${}_n\varphi_m\left(
% \boldsymbol{a};
%  \boldsymbol{b}
%  ; q, z\right)$  satisfies the hypergeometric $q$-difference equation \eqref{intro: qhyper eq}.
\end{prop}
The equation  \eqref{intro:conflu hyper eq}  is  a   special case of the  hypergeometric $q$-difference equation: 
\begin{equation}\label{intro: qhyper eq}
      z\prod_{l=1}^{n}(1-a_l \sigma_q)y(z)=(1-\sigma_q)\prod_{l=1}^{n-1}\left(1-\frac{b_l}{q}\sigma_q\right)y(z),
 \end{equation}
 and we have 
\begin{prop}[\cite{birkhoff1913generalized}]\label{pre:connection prob}
   {If}      the  following    condition holds for  $a_1,...,a_n,b_1,...,b_{n-1}$:
\begin{equation*}
    a_i\neq 0,\  a_i/a_j\notin q^{\mathbb{Z}\backslash \{ 0 \} },\  \text{for} \   1\leq i, j \leq n, \ \text{and} \ b_l\neq   0, \ \text{for} \ 1\leq l\leq n-1,
\end{equation*}
     {then} the equation \eqref{intro: qhyper eq} has a fundamental system of    meromorphic  solutions around infinity:
 \begin{equation}\label{qhyper fsol}
    {}_n\varphi_{n-1}\left(\begin{array}{c}
			a_i, a_i q/{b_1},...,a_i q/{b_{n-1}  }\\
			a_i q/{{a}_{1}},...,\widehat{a_i q/{{a}_{i}} },...,a_i q/{{a}_{n}}
\end{array}		; q, \frac{q\prod_{l=1}^{n-1} b_l}{z\prod_{l=1}^{n}a_l }\right)z^{-\log_qa_i}, \  \text{for} \    i=1,...,n. 
 \end{equation}
Here     $\widehat{a_i q/{{a}_{i}} }$ means the $i$-th item  ${a_i q/{{a}_{i}} }$ is skipped.
\end{prop}
When           $a_n=0$,     ${}_n\varphi_{n-1}\left(\begin{array}{c}
a_1, a_2, ..., a_{n-1},0 \\
b_1, ..., b_{n-1}
\end{array} ; q, z\right)$ is   known as  a   confluent   type basic hypergeometric series, and satisfies the    confluent hypergeometric $q$-difference equation  \eqref{intro:conflu hyper eq}. 
 The equation  \eqref{intro:conflu hyper eq}  also  has a solution of the form 
 \begin{equation*} 
     \left(1+ c_1 z^{-1}+c_2 z^{-2}+...\right)\frac{z^{\sum_{l=1}^{n-1}\log_q\left(a_l/b_l\right)    }   }{(z;q)_\infty},   
 \end{equation*}
where  $c_k$  can be uniquely determined  recursively 
and  $1+\sum_{k=1}^{\infty}c_k z^{-k}$ is convergent  on some  neighbourhood of  infinity (see e.g.,  \cite{adams1928linear}).      Then we get
\begin{prop}\label{n=m+1an=0}
    Suppose the following condition holds  for  $a_1,...,a_{n-1},b_1,...,b_{n-1}$:
    \begin{equation*}
    a_i\neq 0,\  a_i/a_j\notin q^{\mathbb{Z}\backslash \{ 0 \} },\  \text{for} \   1\leq i , j \leq n-1, \ \text{and} \ b_l\neq   0, \ \text{for} \ 1\leq l \leq n-1,
\end{equation*}
then the equation \eqref{intro:conflu hyper eq} has a fundamental system of   formal  solutions around infinity:
\begin{align}
    \hat{f}_q(z)=&\left(
 \hat{f}_q(z)_{1} ,
...  ,
 \hat{f}_q(z)_{n-1}  ,
{f}_q(z)_{n}
\right)     \nonumber   \\
:=&
    \left(
 \hat{h}_q(z)_{1}\cdot z^{-\log_q a_1},
...  ,
 \hat{h}_q(z)_{n-1}\cdot   z^{-\log_q a_{n-1}  },
{h}_q(z)_{n} \cdot   \frac{z^{\sum_{l=1}^{n-1}\log_q\left(a_l/b_l\right)}}{(z;q)_\infty}
\right),    \label{eq:fqinfn}
\end{align}
with  
\begin{align}
   & \hat{h}_q(z)_{i}=  {}_n\varphi_{n-2}\left(\begin{array}{c}
			a_i, a_i q/b_1,...,a_i q/b_{n-1}\\
			a_i q/a_1,...,\widehat{a_i q/a_{{i}} },...,a_i q/a_{n-1}
		\end{array};  q, {\frac{\prod_{l=1}^{n-1} b_l}{ {a_i}z\prod_{l=1}^{n-1}a_l }}\right)  , \  \text{for} \ i=1,...,n-1,  \nonumber   \\
  &{h}_q(z)_{n}=1+ c_1 z^{-1}+c_2 z^{-2}+...,   \label{eq:qformalsol}
\end{align} 
with  $c_k$  uniquely determined  in a recursive way. 
The equation \eqref{intro:conflu hyper eq} also has a fundamental system of meromorphic solutions around the origin:
\begin{equation*}
    {f}^{(0)}_q(z)=\left(
 {f}^{(0)}_q(z)_{1} ,
...  ,
 {f}^{(0)}_q(z)_{n-1}  ,
{f}^{(0)}_q(z)_{n}
\right),
\end{equation*}
with  
\begin{align*}
&{f}^{(0)}_q(z)_{i}:={}_n\varphi_{n-1}\left( \begin{array}{c}
   qa_1/b_i,...,qa_{n-1}/b_i,0   \\
q^2/b_i,qb_1/b_i,...,\widehat{qb_i/b_i},...,qb_{n-1}/b_i  
\end{array}; q,z \right)z^{1-\log_qb_i},\  \text{for}\  i=1,...,n-1,\\
&{f}^{(0)}_q(z)_{n}:={}_n\varphi_{n-1}\left(\begin{array}{c}
a_1, a_2, ..., a_{n-1},0 \\
b_1, ..., b_{n-1}
\end{array} ; q, z\right).
\end{align*}
\end{prop}

\subsection{$q$-Borel Resummation}\label{sec:q-borel sum}
In this subsection, we give a brief review to the theory of $q$-Borel resummation (see e.g.,   \cite{adachi2019q, dreyfus2016q, marotte2000multisommabilite, ramis2013local, zhang1999developpements}). The definitions and properties can be found in \cite{adachi2019q, dreyfus2016q, zhang1999developpements}.
Denote that  
\begin{itemize}
    \item $\mathbb{C}[\![z^{-1}]\!]$  (or $\mathbb{C}[\![z]\!]$) is the ring of formal power series   of  the variable $z^{-1}$ (or $z$);
    \item  $\mathcal{M}(\mathbb{C}^{*},0)$ is the field of functions which  are  meromorphic on some   punctured   neighbourhood  of $0$ in $\mathbb{C}^*$;
    \item $S(a,b)=\{z\in   \widetilde{\mathbb{C}^*}~|~ a<{\mathrm {arg}}(z)<b\}$ is an open sector  with opening
angle $b-a$ on the     Riemann surface $\widetilde{\mathbb{C}^*}$ of the natural logarithm, for  two real numbers $a<b$;
    % \item  $\delta_q$ is the operator $\delta_q :=z\circ D_q$.
    \item   $[\lambda;q]$ is the discrete $q$-spiral $\{\lambda q^n~|~n\in \mathbb{Z}\}$.
\end{itemize}

 \begin{defn}
     Given  $\hat{f}=\sum_{n=0}^{\infty}a_n z^{-n}\in \mathbb{C}[\![z^{-1}]\!]$, {its} $q$-Borel transformation $\hat{\mathcal{B}}_{q;1}\hat{f}$ is defined as:
\begin{equation*}
(\hat{\mathcal{B}}_{q ; 1} \hat{f})(\xi) \coloneqq  \sum_{n = 0}^{\infty} a_{n} q^{\frac{n(n-1)}{2}} \xi^{n}\in \mathbb{C}[\![\xi]\!].
\end{equation*}
 \end{defn}

Let us introduce the space $\mathbb{H}^{[\lambda;q]}_{q;1}$ of functions with $q$-exponential growth of order (at most) $1$  at infinity along $\lambda\in \mathbb{C}^*$.
\begin{defn}
      Let $\lambda\in \mathbb{C}^*$ and $f(\xi)\in \mathcal{M}(\mathbb{C}^{*},0)$. We say  that  $f(\xi)\in \mathbb{H}^{[\lambda;q]}_{q;1}$,  if $f(\xi)$   admits an analytic continuation defined on $S(\mathrm{arg}(\lambda)-\varepsilon,\mathrm{arg}(\lambda)+\varepsilon)$   for some $\varepsilon>0$,     with $q$-exponential growth of order  $1$  at infinity,   which  means that  there exist  constants $L,M>0$, such that for any  $\xi\in S(\mathrm{arg}(\lambda)-\varepsilon,\mathrm{arg}(\lambda)+\varepsilon)$,  
     \begin{equation*}
     % \label{q exp growth}
         | f(\xi) | < M \theta_{ q  }(L |  \xi |).
     \end{equation*}

 % Generally speaking, for a function $f(\xi)$ defined on some  sector $S(a,b)$, if $f(\xi)$ satisfies \eqref{q exp growth} on $S(a,b)$, then we say that $f(\xi)$ has the $q$-exponential growth at infinity on $S(a,b)$.   
\end{defn}

\begin{defn}\label{q-laplace}
If a function $f(\xi)\in \mathbb{H}^{[\lambda;q]}_{q;1}$ for some $\lambda\in \mathbb{C}^*$, then   its $q$-Laplace transformation is defined by
    \begin{equation*}
(\mathcal{L}_{q ; 1}^{[\lambda;q]} f)(z) \coloneqq \sum_{n =-\infty}^{\infty} \frac{f\left(\lambda q^{n}\right)}{\theta_{q}\left({\lambda q^{n}}{z}\right)}.
\end{equation*}
\end{defn}
It is clear that the $q$-Laplace transformation is independent of the chosen representative $\lambda$ and we have
\begin{prop}
       In a   sufficiently small punctured    neighbourhood  of infinity, $\mathcal{L}_{q ; 1}^{[\lambda;q]} f$ has poles of order at most $1$, that are contained in the $q$-spiral $[-\lambda^{-1};q]$.
\end{prop}

Now let us introduce    $[\lambda;q]$-summable functions.
\begin{defn}\label{qsum 1/z}
 For  $\hat{f}\in \mathbb{C}[\![z^{-1}]\!]$, if   $\hat{\mathcal{B}}_{q ; 1} \hat{f}\in \mathbb{H}^{[\lambda;q]}_{q;1}$  for   some $\lambda\in \mathbb{C}^*$, then we say $\hat{f}$ is $[\lambda;q]$-summable and call  $\mathcal{L}_{q ; 1}^{[\lambda;q]} \circ \hat{\mathcal{B}}_{q ; 1} \hat{f}$  its $[\lambda;q]$-sum.         
\end{defn}
The $q$-summability of convergent power series   is  given as follows.
\begin{prop}\label{qsum:convergent}
    If the disc of convergence of   $\hat{f}\in \mathbb{C}[\![z^{-1}]\!]$  is  $\left\{z \in \mathbb{C}^* ~|~R<|z|\right\}$,  then  $\hat{f}$ is $[\lambda;q]$-summable for any $\lambda \in \mathbb{C}^* $  and  
    \begin{equation*}
        \mathcal{L}_{q ; 1}^{[\lambda;q]} \circ \hat{\mathcal{B}}_{q ; 1} \hat{f}=\hat{f},  \  \text{for} \  z\in \left\{z \in \mathbb{C}^* ~|~R<|z|\right\}\backslash [-\lambda^{-1};q].
    \end{equation*}
\end{prop}
The $q$-summability of  the basic hypergeometric series ${ }_n \varphi_{n-2}\left(\boldsymbol{a} ; \boldsymbol{b} ; q,  z^{-1}\right)$ is given as follows.
\begin{prop}[\cite{adachi2019q, dreyfus2015confluence}]\label{prop: qsummability}
       If     $\boldsymbol{a}=(a_1,...,a_n)$,   $\boldsymbol{b}=(b_1,...,b_{n-2})$ satisfy
   \begin{equation}\label{anbn-2cond}
  a_i\neq 0,\      a_i/a_j\notin q^{\mathbb{Z}}, \ \text{for} \  1\leq i \ne j \leq n, \ \text{and}\    b_l\notin \{ 0 \} \cup  q^{-\mathbb{N}}, \  \text{for}\ 1\leq l\leq n-2,  
   \end{equation}
   the  formal series      ${ }_n \varphi_{n-2}\left(\boldsymbol{a} ; \boldsymbol{b} ; q,  z^{-1}\right)$ is $[\lambda ; q]$-summable for any $\lambda \in \mathbb{C}^* \backslash\left[-1 ; q\right]$.  Moreover, for   $c\in \mathbb{C}^*$, the  formal series  ${ }_n \varphi_{n-2}\left(\boldsymbol{a} ; \boldsymbol{b} ; q, c z^{-1}\right)$  is $[\lambda ; q]$-summable for any $\lambda \in \mathbb{C}^* \backslash\left[-1/c ; q\right]$.  
\end{prop}
The $q$-Borel resummation can be applied to a $[\lambda ; q]$-summable formal power series solution of a linear $q$-difference equation (or a linear  
$q$-difference system)
\begin{equation}\label{linear q-difference equation}
         b_m(z) y\left(q^m z\right)+b_{m-1}(z)y\left(q^{m-1} z\right)+...+b_0(z)y\left(z\right)=0,
     \end{equation}
to obtain a meromorphic solution. That is,
 \begin{prop}\label{resummation satisfies equation} 
 For $\lambda \in \mathbb{C}^*$, it follows that
 \begin{itemize}
     \item If  coefficients  $b_i(z)$  are  rational functions,  $\hat{f}\in \mathbb{C}[\![z^{-1}]\!]$ is  $[\lambda;q]$-summable and  satisfies the  equation   \eqref{linear q-difference equation},  then  $\mathcal{L}^{[\lambda;q]}_{q;1}\circ \hat{\mathcal{B}}_{q;1}\hat{f}$  is   a  meromorphic solution of  \eqref{linear q-difference equation}.
     \item    If  coefficients  $b_i(z)$  are  rational  matrix-valued  functions,   $\hat{F}=\big(\hat{f}_{ij}  \big)$  is  a    formal  matrix-valued   solution of  the  system  \eqref{linear q-difference equation}, with  $\hat{f}_{ij}\in \mathbb{C}[\![z^{-1}]\!]$  being $[\lambda;q]$-summable,  then  $F=\left(\mathcal{L}^{[\lambda;q]}_{q;1}\circ \hat{\mathcal{B}}_{q;1}\hat{f}_{ij}  \right)$  is   a  meromorphic solution of  \eqref{linear q-difference equation}.
 \end{itemize}
 \end{prop}

\begin{defn}
    Let $\lambda \in \mathbb{C}^{*}$ and     $\hat{f}=\sum_{n=0}^\infty a_n z^{-n} \in \mathbb{C}[\![z^{-1}]\!]$.  Suppose  $f$ is a function that is analytic on some neighbourhood of infinity  except  along   the $q$-spiral $[-\lambda^{-1};q]$.  We say that  $f$  is  asymptotic to $\hat{f}$  with  $q$-Gevrey order  $1$ along the direction $\lambda$,    if  for any     $\varepsilon,R>0$,  there exist  constants $C,K>0$,    such that for   any integer  $N\geq 1$  and 
 \begin{equation*}
z\in\left\{z \in \mathbb{C}^* ~|~R<|z|\right\} \backslash \cup_{m \in \mathbb{Z}}\left\{z \in \mathbb{C}^* ~|~0<|z^{-1}+\lambda q^{m}|<\varepsilon|\lambda q^{m} |\right\},
\end{equation*}
we have
\begin{equation*}
    \left|f(z)-\sum_{n=0}^{N-1}a_n z^{-n}\right| \leq C K^{N}q^{-\frac{N(N-1)}{2}}|z|^{-N}.
\end{equation*}
\end{defn}
\begin{prop}
  If   $\hat{f}\in \mathbb{C}[\![z^{-1}]\!]$ is $[\lambda;q]$-summable,  then  $\mathcal{L}_{q ; 1}^{[\lambda;q]} \circ \hat{\mathcal{B}}_{q ; 1} \hat{f}$   is asymptotic to $\hat{f}$     with  $q$-Gevrey order  $1$ along the direction $\lambda$.
\end{prop}
Similarly, we can define the $q$-Borel resummation for $\hat{g}\in \mathbb{C}[\![z]\!]$.

\begin{defn}
     Given  $\hat{g}=\sum_{n=0}^{\infty}a_n z^{n}\in \mathbb{C}[\![z]\!]$, the $q$-Borel transformation $\hat{\mathcal{B}}_{q;1}\hat{g}$ is defined as:
\begin{equation*}
(\hat{\mathcal{B}}_{q ; 1} \hat{g})(\xi) \coloneqq  \sum_{n = 0}^{\infty} a_{n} q^{\frac{n(n-1)}{2}} \xi^{n}\in \mathbb{C}[\![\xi]\!].
\end{equation*}
   If   $\hat{\mathcal{B}}_{q ; 1} \hat{g}\in \mathbb{H}^{[\lambda;q]}_{q;1}$  for   some $\lambda\in \mathbb{C}^*$, then its  $q$-Laplace transformation is given by
\begin{equation*}
    \mathcal{L}_{q ; 1}^{[\lambda;q]} \circ \hat{\mathcal{B}}_{q ; 1} \hat{g}:=\sum_{n=-\infty}^{\infty}\frac{\hat{\mathcal{B}}_{q ; 1} \hat{g}\left(\lambda q^{n}\right)}{\theta_{q}\left(\frac{\lambda q^{n}}{z}\right)}.
\end{equation*}
In this case,   we call $\hat{g}$ is $[\lambda;q]$-summable and $\mathcal{L}_{q ; 1}^{[\lambda;q]} \circ \hat{\mathcal{B}}_{q ; 1} \hat{g}$ is the $[\lambda;q]$-sum of $\hat{g}$.
 \end{defn}
 The $q$-summability  of  ${ }_n \varphi_{n-2}\left(\boldsymbol{a} ;  \boldsymbol{b} ; q, z\right)$ is     similar to the case of ${ }_n \varphi_{n-2}\left(\boldsymbol{a} ; \boldsymbol{b} ; q,  z^{-1}\right)$ (see Proposition \ref{prop: qsummability}).
\begin{prop}[\cite{adachi2019q, dreyfus2015confluence}]\label{nfm}
   If     $\boldsymbol{a}=(a_1,...,a_n)$,   $\boldsymbol{b}=(b_1,...,b_{n-2})$ satisfy  \eqref{anbn-2cond},   
   the  formal series      ${ }_n \varphi_{n-2}\left(\boldsymbol{a} ;  \boldsymbol{b} ; q, z\right)$ is $[\lambda ; q]$-summable for any $\lambda \in \mathbb{C}^* \backslash\left[-1 ; q\right]$.
   \end{prop}
   We use the following notation to denote the $[\lambda;q]$-sum of ${ }_n \varphi_{n-2}\left(\boldsymbol{a} ;  \boldsymbol{b} ; q, z\right)$ for  simplicity.
   \begin{defn}\label{def:nfn-2}
For $\boldsymbol{a}=(a_1,...,a_n)$,   $\boldsymbol{b}=(b_1,...,b_{n-2})$ satisfying \eqref{anbn-2cond} and   $\lambda \in \mathbb{C}^* \backslash\left[-1 ; q\right]$,  the  $[\lambda;q]$-sum of ${ }_n \varphi_{n-2}\left(\boldsymbol{a} ;  \boldsymbol{b} ; q, z\right)$    is denoted by
     \begin{equation*}
         { }_n f_{n-2}\left(\boldsymbol{a} ; \boldsymbol{b} ; \lambda ; q, z\right):=\left( \mathcal{L}_{q ; 1}^{[\lambda ; q]} \circ \hat{\mathcal{B}}_{q ; 1} { }_n \varphi_{n-2}\left(\boldsymbol{a} ; \boldsymbol{b} ; q, z\right)  \right)(z).
     \end{equation*}
   \end{defn}
   We can use the function ${ }_n f_{n-2}\left(\boldsymbol{a} ; \boldsymbol{b} ; \lambda ; q, z\right)$  to rewrite the $[\lambda;q]$-sum of ${ }_n \varphi_{n-2}\left(\boldsymbol{a} ; \boldsymbol{b} ; q,  z^{-1}\right)$.   That is,
 \begin{prop}\label{fczclambda}
For    $\boldsymbol{a}=(a_1,...,a_n)$,   $\boldsymbol{b}=(b_1,...,b_{n-2})$ satisfying \eqref{anbn-2cond},     we have
\begin{align}
    &\mathcal{L}_{q ; 1}^{[\lambda ; q]} \circ \hat{\mathcal{B}}_{q ; 1} { }_n \varphi_{n-2}\left(\boldsymbol{a} ; \boldsymbol{b} ; q,  z^{-1}\right)={}_nf_{n-2}\left(
 			\boldsymbol{a};
 			\boldsymbol{b}
 		; {\lambda}; q, z^{-1} \right),  \label{qsummability1}    \\
  &\mathcal{L}_{q ; 1}^{[\lambda ; q]} \circ \hat{\mathcal{B}}_{q ; 1} { }_n \varphi_{n-2}\left(\boldsymbol{a} ; \boldsymbol{b} ; q, c z^{-1}\right)={}_nf_{n-2}\left(
 			\boldsymbol{a};
 			\boldsymbol{b}
 		; c{\lambda}; q,c z^{-1} \right).   \label{qsummability2}
\end{align}
Here in \eqref{qsummability1}  we require  $\lambda \in \mathbb{C}^* \backslash\left[-1 ; q\right]$, and  in  \eqref{qsummability2}   we require   $c\in \mathbb{C}^*$, $\lambda \in \mathbb{C}^* \backslash\left[-1/c ; q\right]$. 
     \end{prop}

\subsection{Explicit Expression of   Connection Formula  of the   Equation \eqref{intro:conflu hyper eq}}\label{seclab conne formu eq}
In this subsection, we derive the explicit  connection formula of the   equation \eqref{intro:conflu hyper eq} by studying the confluence of the known connection formulas of certain  basic hypergeometric series. 
% Other connection formulas for different confluent   type basic hypergeometric series  was already obtained via different methods, like the Barnes' contour integral,  connection preserving deformation (see .   
% Here we do the computation via the most direct way, i.e., via the known asymptotics of confluent hypergeometric functions.

From  Proposition \ref{prop: qsummability}  and  Proposition \ref{resummation satisfies equation},   we  obtain a fundamental system of meromorphic solutions of  the  equation \eqref{intro:conflu hyper eq} via  the  $q$-Borel resummation,  which can be written down explicitly by  Proposition \ref{fczclambda} as follows.
\begin{prop}\label{fqlambda}
There is a fundamental system  of meromorphic  solutions of \eqref{intro:conflu hyper eq}:
\begin{align*}
     {f}_q^{(\infty)}(z,\lambda)=&\left(
 {f}_q^{(\infty)}(z,\lambda)_{1} ,
...  ,
 {f}_q^{(\infty)}(z,\lambda)_{n-1}  ,
{f}_q^{(\infty)}(z)_{n}
\right)\\
:=&\left(  \frac{    \mathcal{L}^{[\lambda;q]}_{q;1}\circ \hat{\mathcal{B}}_{q;1}
 \hat{h}_q(z)_{1} }{z^{\log_q a_1}},
...  ,
 \frac{  \mathcal{L}^{[\lambda;q]}_{q;1}\circ \hat{\mathcal{B}}_{q;1} \hat{h}_q(z)_{n-1} }{z^{\log_q a_{n-1}}} ,
{h}_q(z)_{n}\frac{z^{\sum_{l=1}^{n-1}\log_q\left(a_l/b_l\right)}}{(z;q)_\infty}
\right),
\end{align*}
   where the $[\lambda;q]$-sum     $\mathcal{L}^{[\lambda;q]}_{q;1}\circ \hat{\mathcal{B}}_{q;1}
 \hat{h}_q(z)_{i}$,  for  $1\leq i\leq n-1$  and $\lambda\in \mathbb{C}^*\big \backslash\left[-a_i \prod_{l=1}^{n-1} a_l/ \prod_{l=1}^{n-1} b_l; q\right]$,  can be  rewritten explicitly  as 
    \begin{equation*}
    \mathcal{L}^{[\lambda;q]}_{q;1}\circ \hat{\mathcal{B}}_{q;1}
 \hat{h}_q(z)_{i}=   {}_n f_{n-2} \left( 
\begin{array}{c}
a_i, a_i q / b_1, ..., a_i q / b_{n-1} \\ 
a_i q / a_1, ...,\widehat{ a_i q / a_i}, ..., a_i q / a_{n-1} 
\end{array} ; \frac{\lambda \prod_{l=1}^{n-1} b_l}{a_i \prod_{l=1}^{n-1} a_l}; q, \frac{\prod_{l=1}^{n-1} b_l}{a_i z \prod_{l=1}^{n-1} a_l} 
\right)   .
\end{equation*}
\end{prop}

\begin{defn}\label{def:pseudo}
    We call a function $f(z)$  pseudo-constant  if $f(q z)=f(z)$.
\end{defn}
   
The connection problem  of  \eqref{intro:conflu hyper eq} is to express the solution    ${}_n\varphi_{n-1}\left(
a_1, a_2, ..., a_{n-1},0 ;
b_1, ..., b_{n-1}
 ; q, z\right)$  around the origin as a sum of the solution    ${f}_q^{(\infty)}(z,\lambda)$ around infinity (see Proposition \ref{fqlambda})     with     pseudo-constant  coefficients.     
  The method to solve the connection problem  is {to study} the confluence  of  basic hypergeometric series  along a $q$-spiral. More explicitly, we consider the following two types of limit processes for basic hypergeometric series:
\begin{align}
       & \lim_{\substack{m \in \mathbb{N} \\ m \rightarrow +\infty}   }  {}_n\varphi_{n-1}\left(
a_1, a_2, ..., a_{n};
b_1, ...,b_{n-2}, -1/( \lambda q^m )  
 ; q, -z/(\lambda q^m )      \right),  \label{conflu1}   \\
 & \lim_{\substack{m \in \mathbb{N} \\ m \rightarrow +\infty}}  {}_n\varphi_{n-1}\left(
a_1, a_2, ...,a_{n-1}, -\lambda q^m;
b_1, ..., b_{n-1}
 ; q, z\right)     \label{conflu2}.
 \end{align}
     To do this, we need  the following lemmas.
  
  The   connection problem of \eqref{intro: qhyper eq} is solved      by  the following lemma (see e.g., \cite[Chapter 4.5]{gasper2004basic}  and \cite{thomae1870series}). 
\begin{lemma}\label{connection}
When the  following   condition   holds  for  $\boldsymbol{a}=(a_1,...,a_n)$ and  $\boldsymbol{b}=(b_1,...,b_{n-1})$: 
\begin{equation*}
    a_i\neq 0,\ a_i/a_j\notin q^{\mathbb{Z}}, \ \text{for} \ 1\leq i\neq j \leq n,\  \text{and}\ b_l\notin \{0\}\cup q^{-\mathbb{N}},\ \text{for}\ 1\leq l\leq n-1,
\end{equation*}
the solutions    ${}_n \varphi_{n-1}(\boldsymbol{a} ; \boldsymbol{b} ; q, z)$  around the origin  and  \eqref{qhyper fsol} around infinity of the equation  \eqref{intro: qhyper eq}   are  related  by
\begin{align*}
		{}_n\varphi_{n-1}\left(\boldsymbol{a};\boldsymbol{b}; q, z\right)
  =&
		\sum_{i=1}^{n}
		\frac{(a_1,...,\widehat{a_i},...,a_n,b_1/a_i,...,b_{n-1}/a_i; q)_\infty}{(b_1,...,b_{n-1},a_1/a_i,...,\widehat{a_i/a_i},...,a_n/a_i;q)_\infty}
		\frac{\theta_q(-a_i z)}{\theta_q(-z)}\\
&	\times	{}_n\varphi_{n-1}\left(\begin{array}{c}
    a_i, a_i q/b_1,...,a_iq/b_{n-1}   \\
     a_i q/a_1,...,\widehat{a_iq/a_i},...,a_iq/a_n 
\end{array}	
		; q, \frac{q\prod_{l=1}^{n-1} b_l}{z\prod_{l=1}^{n}a_l }\right).
	\end{align*}
 Here  $\widehat{a_i}$,   $\widehat{a_i/a_i}$, $\widehat{a_iq/a_i}$  mean the terms $a_i, a_i/a_i, a_iq/a_i$ are skipped {respectively}.
 \end{lemma} 
   For $1\leq i\leq n-1$,    the function  ${f}_q^{(\infty)}(z,\lambda)_{i}$  (see Proposition \ref{fqlambda})  is related to the solutions ${f}^{(0)}_q(z)$ (see Proposition \ref{n=m+1an=0}) of the equation \eqref{intro:conflu hyper eq}     via the following lemma (by some parameter {rescaling}).
 \begin{lemma}\label{Adachi}
	(\cite{adachi2019q}).
{If} the  condition  \eqref{anbn-2cond} holds for $\boldsymbol{a}=(a_1,...,a_{n})$  and  $\boldsymbol{b}=(b_1,...,b_{n-2})$, 
     {then}   for $\lambda\in \mathbb{C}^*\backslash [-1 ; q]$,   $z\in  \mathbb{C}^*\backslash [-\lambda;q]$  and $\left| \frac{q\prod_{l=1}^{n-2} b_l}{z\prod_{l=1}^{n}a_l }  \right|<1$,   we have   
       \begin{align*}
  { }_n f_{n-2}(\boldsymbol{a} ; \boldsymbol{b} ; \lambda ; q, z)
  =&\sum_{k=1}^n \frac{\left(a_1,...,\widehat{a_{{k}} },..., a_n, b_1 / a_k, ..., b_{n-2} / a_k ; q\right)_{\infty}}{\left(b_1,... , b_{n-2}, a_1 / a_k,...,\widehat{a_{{k}} /a_k}, ..., a_n / a_k ; q\right)_{\infty}} \frac{\theta_{q}\left(q a_k z / \lambda\right)}{\theta_{q}(q z / \lambda)} \frac{\theta_q\left( a_k \lambda\right)}{\theta_q\left( \lambda\right)} \\
  & \times{ }_n \varphi_{n-1}\left(\begin{array}{c}
  a_k, a_k q / b_1, ..., a_k q / b_{n-2}, {0} \\
  a_k q/a_1,..., \widehat{  a_k q / a_{k}  }, ..., a_k q / a_n
  \end{array} ;q, \frac{q\prod_{l=1}^{n-2} b_l}{ z \prod_{l=1}^n a_l }\right) .
  \end{align*}
\end{lemma}
 Another lemma we need is about the $q$-Pochhammer symbol.
\begin{lemma}
	(\cite{zhang2005remarks}) 
	\label{limit-theta}
Let $a_i \in \mathbb{C}^*$,  $b_i\in \mathbb{C}^*\backslash [-1 ; q]$, for $1\leq i \leq n$.  If $\prod_{i=1}^n {a}_i=\prod_{i=1}^n b_i$, then we have
\begin{equation*}
\lim _{\substack{m \in \mathbb{N} \\ n \rightarrow +\infty}} \frac{   \left(-a_1 q^{-m},...,-a_n q^{-m};q\right)_{\infty}}{\left(-b_1 q^{-m},...,-b_n q^{-m} ; q\right)_{\infty}}=\frac{ \prod_{i=1}^n   \theta_q(a_i) }{   \prod_{i=1}^n  \theta_q(b_i) }.
\end{equation*}
\end{lemma}

The   confluence {process} \eqref{conflu1} is given in the following.
  
\begin{prop}\label{confluent}
	 If  the  condition  \eqref{anbn-2cond}  holds for $a_1,...,a_n,b_1,...,b_{n-2}$,
    then for any $\lambda \in \mathbb{C}^* \backslash\left[-1 ; q\right]$, and $z\in \mathbb{C}^* \backslash\left[-\lambda ; q\right]$, one has
\begin{equation}
	\label{Confluence}
\lim _{\substack{m \in \mathbb{N} \\ m \rightarrow +\infty}} 
{}_n \varphi_{n-1}\left(\begin{array}{c}
a_1,...,a_n \\
b_1,...,b_{n-2},  -1/( \lambda  q^m  )
\end{array} ; q,     -\frac{z}{\lambda q^m} \right)
={}_n f_{n-2}\left(\begin{array}{c}
a_1,...,a_n \\
b_1,...,b_{n-2}
\end{array} ; \lambda; q, z\right).
\end{equation}
\end{prop}

\begin{proof}
	From Lemma \ref{connection}, we have
 \begin{align}\label{nphin-1}
   &  {}_n \varphi_{n-1}\left(\begin{array}{c}
 a_1,...,a_n\\
b_1,...,b_{n-2},  -1/( \lambda q^m ) 
\end{array} ; q, -\frac{z}{\lambda   q^m }\right)   \nonumber \\  
=&\sum_{i=1}^{n}
	\frac{(a_1,...,\widehat{a_i},...,a_n,b_1/a_i,...,b_{n-2}/a_i; q)_\infty}{(b_1,...,b_{n-2},a_1/a_i,...,\widehat{a_i/a_i},...,a_n/a_i ;q)_\infty}
		\frac{\big(   -1/(a_i \lambda  q^m  )  ;q\big)_\infty}{\big(  -1/( \lambda q^m  )   ;q\big)_\infty}
		\frac{\theta_q\big(a_i z/ ( \lambda q^m  )   \big)}{\theta_q\big(z/( \lambda q^m  )\big)}  \nonumber \\
 & \times  
 {}_n\varphi_{n-1}\left(\begin{array}{c}
			a_i, a_i q/b_1,...,a_i q/b_{n-2}, -a_i \lambda q^{m}\\
			a_i q/a_1,...,\widehat{a_i q/a_i},...,a_i q/a_n
		\end{array}; q, \frac{q\prod_{l=1}^{n-2} b_l}{z\prod_{l=1}^{n}a_l }\right).
 \end{align}
In the right  side of \eqref{nphin-1},        from \eqref{theta(q z)}, \eqref{Jacobi's triple product identity}  and  Lemma \ref{limit-theta}   we have
  \begin{equation}\label{lim theta}
      \lim_{\substack{m \in \mathbb{N} \\ m \rightarrow +\infty}}
		\frac{\big(  -1/( a_i\lambda  q^m   )   ;q\big)_\infty}{\big(-1/(\lambda q^m);q\big)_\infty}
		\frac{\theta_q\big(a_i z/(\lambda  q^m   )\big)}{\theta_q\big(z/( \lambda  q^m  )   \big)}
  % =\frac{\theta_q(\frac{1}{\lambda a_i})}{\theta_q(\frac{1}{\lambda})}
		% \frac{\theta_q(\frac{1}{\lambda}a_i z)}{\theta_q(\frac{1}{\lambda}z)}
  =\frac{\theta_q(a_i\lambda )}{\theta_q(\lambda)}
		\frac{\theta_q({q a_i z}/{\lambda})}{\theta_q({qz}/{\lambda})},
  \end{equation}
  and
  \begin{align}
     & \lim_{\substack{m \in \mathbb{N} \\ m \rightarrow +\infty}}
		{}_n\varphi_{n-1}\left(\begin{array}{c}
			a_i, a_i q/b_1,...,a_i q/b_{n-2}, -a_i \lambda q^{m}\\
			a_i q/{a_1},...,\widehat{a_i q/{a}_{i} },...,a_i q/a_n
		\end{array}; q, \frac{q\prod_{l=1}^{n-2} b_l}{z\prod_{l=1}^{n}{a_l} }\right)  \nonumber  \\
  =&{}_n\varphi_{n-1}\left(\begin{array}{c}
			a_i, a_i q/b_1,...,a_i q/b_{n-2}, 0\\
			a_i q/{a_1},...,\widehat{a_i q/{a}_{i} },...,a_i q/a_n
		\end{array}; q, \frac{q\prod_{l=1}^{n-2} {b_l}}{z\prod_{l=1}^n{a_l}}\right).  \label{lim an}
  \end{align}
Combining \eqref{lim theta} and \eqref{lim an}, we obtain
  \begin{align*}
  &    \lim_{\substack{m \in \mathbb{N} \\ m \rightarrow +\infty}} {}_n \varphi_{n-1}\left(\begin{array}{c}
a_1,...,a_n \\
b_1,...,b_{n-2},-1/( \lambda q^m  )  
\end{array} ; q, -\frac{z}{\lambda   q^{m}}\right)    \\
=&\sum_{i=1}^{n}
	\frac{(a_1,...,\widehat{a_i},...,a_n,{b}_1/a_i,...,{b}_{n-2}/a_i; q)_\infty}{(b_1,...,b_{n-2},a_1/a_i,...,\widehat{a_i/a_i},...,a_n/a_i ;q)_\infty} \frac{\theta_q(a_i\lambda )}{\theta_q(\lambda)}
		\frac{\theta_q({q a_i z}/{\lambda})}{\theta_q({qz}/{\lambda})}    \\
 & \times {}_n\varphi_{n-1}\left(\begin{array}{c}
			a_i, a_i q/b_1,...,a_i q/b_{n-2}, 0\\
		a_i q/{a_1},...,\widehat{a_i q/{a}_{i} },...,a_i q/a_n
		\end{array}; q, \frac{q\prod_{l=1}^{n-2} {b_l}}{z\prod_{l=1}^n{a_l}}\right),
  \end{align*}
which coincides with  Lemma \ref{Adachi}.  So   we get the conclusion.
	\end{proof}
 Combining confluence process of  \eqref{conflu2},  we have
\begin{thm}\label{our connection formula}
If the following condition holds for           $a_1,a_2,...,a_{n-1},b_1,b_2,...,b_{n-1}$:
\begin{equation*}
    a_i\neq 0, \ a_i/a_j\notin q^{\mathbb{Z}}, \ \text{for}\ 1\leq i\neq j \leq n-1, \  \text{and} \ b_l\notin \{0\}\cup q^{-\mathbb{N}}, \  \text{for}  \ 1\leq l \leq n-1,
\end{equation*}
and  the following condition holds  for  $\lambda,z$:     
\begin{equation*}
    \lambda\in \mathbb{C}^*,\ \lambda\notin -a_i q^{\mathbb{Z}}, \  \text{for} \   1\leq i \leq n-1, \  \text{and} \  z\in \mathbb{C}^*, \  z\notin q^{\mathbb{Z}}, \  z\notin  -\frac{\prod_{l=1}^{n-1} b_l}{ \lambda \prod_{l=1}^{n-1} a_l}q^{\mathbb{Z}},
\end{equation*}
it follows that
	\begin{align*}
		&{}_n\varphi_{n-1}\left(\begin{array}{c}
			a_1,...,a_{n-1},0\\
			b_1,...,b_{n-1}\end{array}
		; q, z\right)\\
		=&
		\sum_{i=1}^{n-1}
		\frac{(a_1,...,\widehat{a_{{i}} },...,a_{n-1},b_1/a_i,...,b_{n-1}/a_i;q)_\infty}{(b_1,...,b_{n-1},a_1/a_i,...,\widehat{a_{{i}}/a_i},...,a_{n-1}/a_i;q)_\infty}\frac{\theta_q(-a_i z)}{ \theta_q(-z) }
		\\
     &\times
		{}_nf_{n-2}\left(\begin{array}{c}
			a_i, a_i q/b_1,...,a_i q/b_{n-1}\\
			a_i q/a_1,...,\widehat{a_i q/a_{{i}} },...,a_i q/a_{n-1}
		\end{array}; \frac{\lambda}{a_i}; q, {\frac{\prod_{l=1}^{n-1} b_l}{ {a_i}z\prod_{l=1}^{n-1}a_l }}\right)
		\\&+
		\frac{(a_1,...,a_{n-1};q)_\infty}{(b_1,...,b_{n-1};q)_\infty}
		\frac{\prod_{l=1}^{n-1}\theta_q(b_l/\lambda)}{\prod_{l=1}^{n-1}\theta_q(a_l/\lambda)}
		\frac{\theta_q(\lambda z)}{\theta_q\left( {\lambda z\prod_{l=1}^{n-1} a_l}\big/{\prod_{l=1}^{n-1} b_l}\right)}
		\frac{{h}_q(z)_{n}}{(z; q)_\infty}
		.
	\end{align*}
  Here  ${}_n\varphi_{n-1}\left(\begin{array}{c}
   a_1,...,a_{n-1},0   \\
      b_1,...,b_{n-1}
\end{array}
		; q, z\right)$    is  the solution  of  the  equation   \eqref{intro:conflu hyper eq} around the origin (see Proposition \ref{n=m+1an=0}),  the function  ${}_nf_{n-2}
  %       \left(\begin{array}{c}
		% 	a_i, a_i q/b_1,...,a_i q/b_{n-1}\\
		% 	a_i q/a_1,...,\widehat{a_i q/a_{{i}} },...,a_i q/a_{n-1}
		% \end{array}; \frac{\lambda}{a_i}; q, {\frac{\prod_{l=1}^{n-1} b_l}{ {a_i}z\prod_{l=1}^{n-1}a_l }}\right)
        $ is  given in Defintion  \ref{def:nfn-2},    and     ${h}_q(z)_{n}$  (see  \eqref{eq:qformalsol})  is  the   power series  part of the convergent solution  ${f}_q(z)_{n}$  of  the  equation \eqref{intro:conflu hyper eq} (see  Proposition  \eqref{n=m+1an=0}).     The symbols     $\widehat{a_{{i}}/a_i }$,   $\widehat{a_i q/a_{{i}} }$ mean the  terms  $a_i/a_i$,  $a_i q/a_{{i}}$ are skipped  respectively.       
\end{thm}

 \begin{proof}
 Let  $\lambda$ be a non-zero number such that $-\lambda/a_i\notin q^{\mathbb{Z}}$, $1\leq i \leq n-1$. By Lemma \ref{connection}, it follows that 
         \begin{align}
&{ }_{n} \varphi_{n-1}\left(\begin{array}{c}
a_1,...,a_{n-1}, -\lambda q^m \\
b_1,...,b_{n-1}\end{array}
; q, z\right)  \nonumber   \\ 
=&
\sum_{i=1}^{n-1} \frac{\left(a_1,...,\widehat{a_{{i}} },...,a_{n-1}, b_1 / a_{i},...,b_{n-1}/a_i;q\right)_{\infty}}{\left(b_1,...,b_{n-1}, a_1/a_i,...,\widehat{a_{{i}} / a_{i} },...,a_{n-1}/a_i;q\right)_{\infty}} \frac{\left(-\lambda q^m\right)_{\infty}}{\left(-\lambda q^m / a_{i}\right)_{\infty}} \frac{\theta_q\left(-a_{i} z\right)}{\theta_q(-z)}  \nonumber   \\
&\times { }_{n} \varphi_{n-1}\left(\begin{array}{c}
a_{i}, a_{i} q / b_1,...,a_i q/b_{n-1} \\
a_i q/a_1,...,\widehat{a_{i} q / a_{{i}} },...,a_i q/a_{n-1}, -a_{i} q^{1-m} / \lambda
\end{array} ; q, -\frac{q^{1-m} \prod_{l=1}^{n-1}{b_l}}{z\lambda  \prod_{l=1}^{n-1} {a_l}  }\right)  \nonumber   \\
& +\frac{(a_1,...,a_{n-1};q)_{\infty}}{(b_1,...,b_{n-1};q)_{\infty}} \frac{\left(-b_1 q^{-m}/\lambda,...,-b_{n-1} q^{-m}/\lambda ;q\right)_{\infty}}{\left(-a_1 q^{-m}/\lambda,...,-a_{n-1}q^{-m}/\lambda;q \right)_{\infty}} \frac{\theta_q(\lambda q^m z)}{\theta_q(-z)}     \nonumber    \\
&\times {}_{n} \varphi_{n-1}\left(\begin{array}{c}
-\lambda q^m, -\lambda q^{1+m}  / b_1,...,-\lambda q^{1+m}  / b_{n-1}\\
-\lambda q^{1+m} / a_1,...,-\lambda q^{1+m}/a_{n-1}
\end{array} ; q, -\frac{q^{1-m} \prod_{l=1}^{n-1} {b_l}}{z \lambda \prod_{l=1}^{n-1} {a_l}   }\right),   \label{connection in main proof}
\end{align}
By Proposition \ref{confluent}, for $1\leq i \leq n-1$, we have
\begin{align}
& \lim_{\substack{m \in \mathbb{N} \\ m \rightarrow +\infty}}  { }_{n} \varphi_{n-1}\left(\begin{array}{c}
a_{i}, a_{i} q / b_1,...,a_i q/b_{n-1} \\
a_i q/a_1,...,\widehat{a_{i} q / a_{{i}} },...,a_i q/a_{n-1}, -a_{i} q^{1-m} / \lambda
\end{array} ; q, -\frac{q^{1-m} \prod_{l=1}^{n-1}{b_l}}{z\lambda  \prod_{l=1}^{n-1} {a_l}  }\right) \nonumber   \\ 
=& { }_{n} f_{n-2}\left(\begin{array}{c}
a_{i}, a_{i} q / b_1,...,a_i q/b_{n-1} \\
a_i q/a_1,...,\widehat{a_{i} q / a_{{i}} },...,a_i q/a_{n-1}
\end{array} ;\frac{\lambda}{a_i}, q, \frac{\prod_{l=1}^{n-1} {b_l}}{z\cdot a_i\prod_{l=1}^{n-1} {a_l} }\right).\label{RLHn}
\end{align}
By Lemma \ref{limit-theta}  we have
\begin{align}
& \lim_{\substack{m \in \mathbb{N} \\ m \rightarrow +\infty}}    \frac{\left(-b_1 q^{-m}/\lambda,...,-b_{n-1} q^{-m}/\lambda ;q\right)_{\infty}}{\left(-a_1 q^{-m}/\lambda,...,-a_{n-1}q^{-m}/\lambda;q \right)_{\infty}} \frac{\theta_q(\lambda q^m z)}{\theta_q\left( {z\lambda q^m \prod_{l=1}^{n-1}a_l}\big/{\prod_{l=1}^{n-1}b_l}  \right)}   \nonumber  \\
 =&\frac{\prod_{l=1}^{n-1}\theta_q\left(b_l /\lambda\right)}{\prod_{l=1}^{n-1}\theta_q\left(a_l /\lambda\right)} \frac{\theta_q\left( {\lambda   z}   \right)}{ \theta_q \left( { z\lambda \prod_{l=1}^{n-1}a_l}  \big/  {\prod_{l=1}^{n-1}b_l} \right)}. \label{limthetan}   
 \end{align}
and  combining \eqref{theta(q z)}     and Lemma   \ref{connection},  we  get  (here for convenience, we denote $b_n=q$)
 \begin{align}
  &\lim_{\substack{m \in \mathbb{N} \\ m \rightarrow +\infty}} \frac{  \theta_q\left( {z\lambda q^m \prod_{l=1}^{n-1}a_l}\big/{\prod_{l=1}^{n-1}b_l}  \right)  }{  (q;q)_{\infty}  (q/z;q)_{\infty} }      {}_{n} \varphi_{n-1}\left(\begin{array}{c}
-\lambda q^m, -\lambda q^{1+m}  / b_1,...,-\lambda q^{1+m}  / b_{n-1} \nonumber  \\
-\lambda q^{1+m} / a_1,...,-\lambda q^{1+m}/a_{n-1}
\end{array} ; q, -\frac{q^{1-m} \prod_{l=1}^{n-1} {b_l}}{z \lambda \prod_{l=1}^{n-1} {a_l}   }\right)\\
    =& \lim_{\substack{m \in \mathbb{N} \\ m \rightarrow +\infty}}    \sum_{i=1}^{n}  \frac{ \prod_{l=1,l\neq i}^{n}  \left( -\lambda q^{1+m}/b_l;q\right)_{\infty}\prod_{l=1}^{n-1}\left( b_i/a_l;q  \right)_{\infty}   }{ \prod_{l=1}^{n-1}  \left( -\lambda q^{1+m} / a_l;q\right)_{\infty} \prod_{l=1,l\neq i}^{n}\left(b_i/b_l;q  \right)_{\infty}     }\frac{\theta_q \left(-{z\prod_{l=1}^{n-1}a_l}\big/{\prod_{l=1,l\neq i}^{n}b_l}    \right)    }{   (q;q)_{\infty}  (q/z;q)_{\infty}  }  \nonumber  \\  
 &  \times  {}_n \varphi_{n-1}   \left(\begin{array}{c}
  			-\lambda q^{1+m}/b_i,  q{a}_1/b_i,...,q{a}_{n-1}/b_i\\
  			qb_1/b_i,...,\widehat{qb_i/b_i },...,qb_n/b_i\end{array}
  		; q, z\right) \nonumber  \\
    =&\sum_{i=1}^{n}  \frac{   \prod_{l=1}^{n-1}\left( b_i/a_l;q  \right)_{\infty}   }{    \prod_{l=1,l\neq i}^{n}\left(b_i/b_l;q  \right)_{\infty}     }\frac{\theta_q \left(-z\prod_{l=1}^{n-1}a_l\big/\prod_{l=1,l\neq i}^{n}b_l    \right)    }{ (q;q)_{\infty}  (q/z;q)_{\infty} }{}_n \varphi_{n-1}   \left(\begin{array}{c}
  			  q{a}_1/b_i,...,q{a}_{n-1}/b_i,0\\
  			qb_1/b_i,...,\widehat{qb_i/b_i },...,qb_n/b_i
  		\end{array}; q, z\right),
    \label{insert}
\end{align}
which is independent of $\lambda$. So  \eqref{insert}  implies that  (here for convenience, we denote the left side of \eqref{insert}  as  $\psi(z)$):
\begin{itemize}
    \item  Together with \eqref{Jacobi's triple product identity},  $\psi(z)$  can be expanded as a power series of $z^{-1}$,   with the constant term being 1 and all coefficients  independent of $\lambda$.
    \item  Together with  Proposition \ref{n=m+1an=0},  we  have  
    \begin{equation}
     \psi(z)   \frac{z^{\sum_{l=1}^{n-1}\log_q(a_l/b_l)}}{(z;q)_\infty}=\sum_{i=1}^{n}  \frac{   \prod_{l=1}^{n-1}\left( b_i/a_l;q  \right)_{\infty} \theta_q \big(-z\prod_{l=1}^{n-1}a_l\big/\prod_{l=1,l\neq i}^{n}b_l    \big)  z^{\sum_{l=1}^{n-1}\log_q(a_l/b_l)}  }{    \prod_{l=1,l\neq i}^{n}\left(b_i/b_l;q  \right)_{\infty}   \theta_q(-z)  z^{1-\log_q b_i}  }    {f}^{(0)}_q(z)_{i}
        ,\label{psifzero}
    \end{equation}
    which is a sum  of  solutions  ${f}^{(0)}_q(z)_{i}$,  $1\leq i \leq n$, of  the  equation \eqref{intro:conflu hyper eq}  with  pseudo-constant coefficients, and as a result is  also a solution of the  equation \eqref{intro:conflu hyper eq}.
\end{itemize}
Comparing with   \eqref{eq:fqinfn}  and    \eqref{eq:qformalsol},  we  get  
\begin{equation}
    \psi(z)={h}_q(z)_{n}.\label{pf:psi}
\end{equation}
As a result,  combining \eqref{connection in main proof}, \eqref{RLHn},  \eqref{limthetan}, \eqref{insert} and \eqref{pf:psi},  we get the conclusion in this theorem.
 \end{proof}
 % \begin{rmk}
 %     From  Theorem  \ref{our connection formula}  the connection formula between ${}_n\varphi_{n-1}\left( a_1,...,a_{n-1},0;b_1,...,b_{n-1};q,z \right)$  and  ${f}_q^{(\infty)}(z,\lambda)$ (given in Proposition \ref{fqlambda})   of the  equation  \eqref{intro:conflu hyper eq}  can be obtained.
 % \end{rmk}
We can prove Theorem \ref{intro:our connection formula} now, which is the connection formula for the equation \eqref{intro:conflu hyper eq}.
\begin{proof}[Proof of Theorem  \ref{intro:our connection formula}]
 {It follows}    from Proposition \ref{fqlambda} and Theorem \ref{our connection formula}. 
\end{proof}
 
 From the formulas \eqref{psifzero} and \eqref{pf:psi} in the proof of Theorem \ref{our connection formula}, we get
  \begin{cor}\label{corfn}
      Under the condition of Theorem \ref{our connection formula}, it follows that (note that we denote $b_n=q$ here)
      \begin{equation*}
       {f}_q^{(\infty)}(z)_{n}   =\sum_{i=1}^{n}  \frac{   \prod_{l=1}^{n-1}\left( b_i/a_l;q  \right)_{\infty}   }{    \prod_{l=1,l\neq i}^{n}\left(b_i/b_l;q  \right)_{\infty}     }\frac{\theta_q \left(-z\prod_{l=1}^{n-1}a_l\big/\prod_{l=1,l\neq i}^{n}b_l    \right)    }{ \theta_q(-z) } \frac{ z^{\sum_{l=1}^{n-1}\log_q(a_l/b_l)} }  {z^{1-\log_q b_i}   }   {f}^{(0)}_q(z)_{i}
        .
      \end{equation*}
  \end{cor}
\begin{rmk}\label{rmkres:nvaphin}
    A similar method can be used to derive the connection formula of ${}_n\varphi_n(a_1,...,a_n;b_1,...,b_n;z)$, which is included in \cite{ohyama2021q}).
\end{rmk}

\section{Connection Matrix, $q$-Stokes Matrix of the System \eqref{qdes}}\label{main section}
This section computes explicitly the  connection matrix and $q$-Stokes matrix of the  system \eqref{qdes}. In particular, Section \ref{sec:sol zero inf} introduces the connection matrix and $q$-Stokes matrix of the systems \eqref{qdes}. Section \ref{sec:diag}  
gives a gauge transformation that diagonalizes the upper left block of the system \eqref{qdes}. Sections \ref{sec:diagsol}--\ref{sec:diagqdes} derive the explicit expressions of the connection matrix and $q$-Stokes matrix of the system after the gauge transformation. Section \ref{sec:qdes} then obtains the connection matrix and   $q$-Stokes matrix of the system \eqref{qdes}. 
\subsection{Fundamental Solution,  Connection Matrix and $q$-Stokes Matrix}\label{sec:sol zero inf}
 The fundamental solution of \eqref{u+A/z} is studied in \cite{birkhoff1913generalized} (especially when $u$ is invertible).  If $u$ is not invertible, there is     a slightly different formal  fundamental solution   (see e.g.,  \cite{dreyfus2015confluence}). Particularly,  the system \eqref{qdes} has a  formal fundamental solution around infinity as follows
 \begin{equation*}
    \hat{F}_q(z;E_{n n},A)=\hat{H}_q(z;E_{nn},A)  z^{\delta^{(n-1)}_{q}(A)}   e_q\left( E_{n n} q^{ -\delta_{q}^{(n-1)}(A) } z \right),
\end{equation*}
where $\hat{H}_q(z;E_{nn},A)=\mathrm{Id}_n+\sum_{m=1}^{\infty} H^{(\infty)}_{m}  z^{-m}$ and 
\begin{equation}\label{deltan-1}
    \delta^{(n-1)}_{q}\left(A\right)_{ij}=\left\{
          \begin{array}{lr}
             \log_q(1-(1-q)A^{(n-1)})_{i j},   & \text{if} \ \ 1\le i, j\le n-1,  \\
     \sum_{l=1}^{n} \mu_l-\sum_{l=1}^{n-1}\nu_l        , & \text{if} \ \ i=j=n,\\
           0, & \text{otherwise},
             \end{array}
\right.
\end{equation}
with  $A^{\left(n-1\right)}$ being  the upper left $(n-1)\times (n-1)$ submatrix of $A=\left(a_{i j}\right)$    and \[e_q\big( E_{n n} q^{ -\delta_{q}^{(n-1)}(A) } z \big)=\mathrm{diag}\Big(0,...,0,e_q\Big(q^{-\delta^{(n-1)}_{q}\left(A\right)_{n n}}z  \Big)\Big).\]  
\begin{prop}[see e.g.,  \cite{adams1928linear}]\label{convergent column}
    The entries in the last column of $\hat{F}_q(z;E_{n n},A)$  are convergent, while the  other entries of $\hat{F}_q(z;E_{n n},A)$  are   in general divergent.
\end{prop}
 But still, the entries of (the first $n-1$ columns of) the formal fundamental solution $\hat{F}_q(z;E_{nn},A)$ (or equivalently $\hat{H}_q(z;E_{n n},A)$) of the equation \eqref{qdes} are $[\lambda;q]$-summable for any $\lambda\in \mathbb{C}^{*}\backslash \cup_{1\leq j \leq n-1}[-{(1-q)}/{q^{\nu_j}};q]$ (see \cite{adachi2019q, dreyfus2015confluence}). That is,
\begin{prop}[\cite{adachi2019q, dreyfus2015confluence}]
For any $\lambda\in \mathbb{C}^{*}\backslash \cup_{1\leq j \leq n-1} [-{(1-q)}/{q^{\nu_j}};q]$,    there is a  fundamental solution $F_q^{(\infty)}(z,\lambda;E_{nn},A)$ of the equation \eqref{qdes} obtained from  the  $q$-Borel resummation as follows
  \begin{equation*}
      F_q^{(\infty)}(z,\lambda;E_{nn},A):={H}^{(\infty)}_q(z,\lambda;E_{n n},A)z^{\delta^{(n-1)}_{q}(A)}   e_q\left( E_{n n} q^{ -\delta_{q}^{(n-1)}(A) } z \right), 
  \end{equation*}
  with the $n\times n$ matrix ${H}^{(\infty)}_q(z,\lambda;E_{n n},A)$ given by
  \begin{equation*}
   {H}^{(\infty)}_q(z,\lambda;E_{n n},A)_{ij}:=   \mathcal{L}_{q ; 1}^{[\lambda;q]} \circ \hat{\mathcal{B}}_{q ; 1} (\hat{H}_q(z;E_{n n},A)_{ij}),\  \text{for} \ i,j=1,...,n.
  \end{equation*}
\end{prop}
Moreover, there is a convergent  fundamental solution around the origin of the form (see \cite{birkhoff1913generalized})
\begin{equation*}
    F_q^{(0)}(z;E_{n n},A)={H}^{(0)}_q(z;E_{nn},A)z^{\log_q(1-(1-q) A_n ) }, \  \text{with} \  {H}^{(0)}_q(z;E_{nn},A)= H^{(0)}_{0}+\sum_{m=1}^{\infty}H^{(0)}_{m}z^m  ,
\end{equation*}
where  $A_n:=\left(\lambda^{(n)}_1,...,\lambda^{(n)}_n  \right)$,   $\log_q(1-(1-q) A_n )=\operatorname{diag}(\mu_1,...,\mu_{n})$, and $H_0^{(0)}$  satisfies $A=H^{(0)}_0 A_n \left(H_0^{(0)}\right)^{-1}$. 

\begin{defn}\label{UqSq}
For any $\lambda\in \mathbb{C}^{*}\backslash \cup_{1\leq j \leq n-1} [-{(1-q)}/{q^{\nu_j}};q]$,       
the connection matrix $U_q(z,\lambda;E_{nn},A)$ of \eqref{qdes}   is  defined   by
\begin{equation*}
   F_q^{(\infty)}(z,\lambda;E_{nn},A) U_q(z,\lambda;E_{nn},A)=F_q^{(0)}(z;E_{nn},A).
\end{equation*}
For another $\mu \in \mathbb{C}^{*}$    satisfying    $\mu\notin   [\lambda;q]  $  and     $\mu\notin [-{(1-q)}/{q^{\nu_j}};q]$,  $1\leq j \leq n-1$,       the  $q$-Stokes matrix   $S_q(z,\lambda,\mu;E_{nn},A)$ of \eqref{qdes}  is defined   by
\begin{equation*}
    F_q^{(\infty)}(z,\mu;E_{nn},A)=F_q^{(\infty)}(z,\lambda;E_{nn},A) S_q(z,\lambda,\mu;E_{nn},A).
\end{equation*}
\end{defn}

\subsection{Diagonalization of the Upper Left Submatrix of $A$}\label{sec:diag}
 We first introduce a diagonalized system under   the   conditions:
\begin{align}\label{ineq1}
\lambda^{\left(n-1\right)}_i & \ne \lambda^{\left(n-1\right)}_j, \text{ for } \ 1\leq i\neq j \leq n-1,\\ \label{ineq2}
\lambda^{\left(n-1\right)}_i & \ne \lambda^{\left(n-2\right)}_j, \text{ for } \ 1\leq i \leq n-1, \  1\leq j \leq n-2 .
\end{align}

 Under the conditions \eqref{ineq1} and \eqref{ineq2},    let us take the $\left(n-1\right)\times \left(n-1\right)$ matrix $P_{n-1}$ given by
\begin{equation*}
    (P_{n-1})_{ij}:=\frac{(-1)^{i+j}\Delta^{1,...,n-2}_{1,...,\hat{i},...,n-1}\left(A-\lambda^{\left(n-1\right)}_j\cdot \mathrm{Id}_n\right)}{\sqrt{ \prod_{l=1,l\ne j}^{n-1}\left(\lambda^{\left(n-1\right)}_l-\lambda^{\left(n-1\right)}_j\right)\prod_{l=1}^{n-2}\left(\lambda^{\left(n-2\right)}_l-\lambda^{\left(n-1\right)}_j\right)} }.
\end{equation*}
Here $\Delta_{1,...,\hat{i},..., n-1}^{1,...,n-2}\left(A-\lambda^{\left(n-1\right)}_j\cdot \mathrm{Id}_n\right)$  is  the  minor of the matrix $A-\lambda^{\left(n-1\right)}_j\cdot \mathrm{Id}_n$, formed by the first $n-2$ rows and the first $n-1$ columns, with the   $i$-th column  omitted.  Then the inverse matrix of $P_{n-1}$ is given explicitly as follows and we can use $P_{n-1}$ to diagonalize the upper left  submatrix  $A^{(n-1)}$ of $A$.
\begin{prop}[see e.g., \cite{lin2024explicit}]\label{An-1Pn-1}
     If $A$ satisfies \eqref{ineq1} and \eqref{ineq2}, then $P_{n-1}$ is invertible and the inverse matrix $P_{n-1}^{-1}$ is given by
    \begin{equation*}
        \left(P_{n-1}^{-1}\right)_{i j}=\frac{(-1)^{i+j}\Delta^{1,...,\hat{j},...,n-1}_{1,...,n-2}   \left(A-\lambda^{\left(n-1\right)}_i\cdot \mathrm{Id}_n\right)}{\sqrt{ \prod_{l=1,l\ne i}^{n-1}\left(\lambda^{\left(n-1\right)}_l-\lambda^{\left(n-1\right)}_i\right)\prod_{l=1}^{n-2}\left(\lambda^{\left(n-2\right)}_l-\lambda^{\left(n-1\right)}_i\right)} },  \  \text{for}  \  1\leq i,j \leq    n-1     .
    \end{equation*}
    Moreover, we have
    \begin{equation*}
       P_{n-1}^{-1} A^{\left(n-1\right)} {P}_{n-1}= \mathrm{diag}\left(\lambda^{\left(n-1\right)}_1,\ldots,\lambda^{\left(n-1\right)}_{n-1}\right), 
    \end{equation*}
    and we introduce
    \begin{eqnarray*}
        A_{n-1}:=\operatorname{diag}\left({P}_{n-1}^{-1},1\right)\cdot A \cdot \operatorname{diag}\left({P}_{n-1},1\right) = \left(\begin{array}{cccc}
\lambda^{\left(n-1\right)}_1 & & 0 & a^{\left(n-1\right)}_1 \\
 & \ddots & & \vdots \\
0 & & \lambda^{\left(n-1\right)}_{n-1} & a^{\left(n-1\right)}_{n-1} \\
b^{\left(n-1\right)}_1 & \cdots & b^{\left(n-1\right)}_{n-1} & a_{n n} 
\end{array} \right).
    \end{eqnarray*}
Here   $a^{\left(n-1\right)}_i$ and $b^{\left(n-1\right)}_i$  are
\begin{align*}
% \label{aki}
    a^{\left(n-1\right)}_i&= \sum_{v=1}^{n-1}\left({P}_{n-1}^{-1}\right)_{i v} \cdot a_{v n}=\frac{(-1)^{i+n-1}\Delta^{1,...,n-1}_{1,...,n-2,n}\left(A-\lambda^{\left(n-1\right)}_i\cdot \mathrm{Id}_n\right)}{\sqrt{\prod_{l=1,l\ne i}^{n-1}\left(\lambda^{\left(n-1\right)}_l-\lambda^{\left(n-1\right)}_i\right)\prod_{l=1}^{n-2}\left(\lambda^{\left(n-2\right)}_l-\lambda^{\left(n-1\right)}_i\right)}},\\ 
    % \label{bki}
    b^{\left(n-1\right)}_i&= \sum_{v=1}^{n-1} a_{n v}\cdot \left({P}_{n-1}\right)_{v i}=\frac{(-1)^{i+n-1}\Delta_{1,...,n-1}^{1,..., n-2,n}\left(A-\lambda_i^{\left(n-1\right)}\cdot \mathrm{Id}_n\right)}{\sqrt{\prod_{l=1,l\ne i}^{n-1}\left(\lambda^{\left(n-1\right)}_l-\lambda^{\left(n-1\right)}_i\right)\prod_{l=1}^{n-2}\left(\lambda^{\left(n-2\right)}_l-\lambda^{\left(n-1\right)}_i\right)}},
\end{align*}
and  we  have
\begin{equation}\label{eq:aibi}
 a^{\left(n-1\right)}_i   b^{\left(n-1\right)}_i =-\frac{\prod_{l=1}^{n}\left(\lambda^{\left(n\right)}_l-\lambda^{\left(n-1\right)}_i\right)}{\prod_{l=1,l\neq i}^{n-1}\left(\lambda^{\left(n-1\right)}_l-\lambda^{\left(n-1\right)}_i\right)}.
\end{equation}
\end{prop}

Now let us use the conjugation of the $n\times n$ matrix $\operatorname{diag}\left({P}_{n-1},1\right)$ to diagonalize the upper left part of the equation \eqref{qdes}  and   obtain the equation 
\begin{equation}
    D_q F_q\left(z\right)=\left( E_{n n}+\frac{A_{n-1}}{z}\right) F_q\left(z\right).\label{diagqdes}
\end{equation}

\subsection{Fundamental Solutions of the System   \eqref{diagqdes}}\label{sec:diagsol}
In this subsection, we compute the  fundamental solution around the  origin and infinity  of \eqref{diagqdes} explicitly.    
\begin{prop}\label{ffs qsym}
Under   the conditions   \eqref{eq:vivj}  (with the condition \eqref{ineq1} included)  and       \eqref{ineq2}, 
the  formal  fundamental solution  $\hat{F}_q(z;E_{nn},A_{n-1})$ of  \eqref{diagqdes} is given by 
\begin{equation*}
    \hat{F}_q(z;E_{nn},A_{n-1})=\hat{H}_q(z;E_{nn},A_{n-1})\cdot   z^{\delta^{(n-1)}_q(A_{n-1})}   e_q\left( E_{n n} q^{ -\delta^{(n-1)}_q(A_{n-1}) } z \right),
\end{equation*}
where $\delta^{(n-1)}_q(A_{n-1})=\mathrm{diag}(\nu_1,...,\nu_{n-1},\nu_n)$ (recall   \eqref{deltan-1} for the operator $\delta_q^{(n-1)}$), with denoting
\begin{equation}\label{eq:vn}
   \nu_n:=\sum_{l=1}^n\mu_l-\sum_{l=1}^{n-1}\nu_l, 
\end{equation}  
for convenience. Here the entries of $\hat{H}_q(z;E_{nn},A_{n-1})$ are 
\begin{equation*}
\hat{H}_q(z;E_{nn},A_{n-1})_{i j}=d_{q;ij}\cdot { }_{n} \varphi_{n-2}\left(\begin{array}{c}
q^{1+\boldsymbol{\mu}-\nu_{j}-\delta_{ij} } \\
q^{\boldsymbol{\delta^{(n-1)}_{i \hat{j}}}+1+\boldsymbol{\nu_{\hat{j}}}-\nu_{j}-\delta_{ij} }
\end{array} ; q, \frac{q^{\nu_{j}+\delta_{ij}-\delta_{ni}   }}{(1-q) z}\right),\quad 1\leq i \leq n, 1\leq j \leq n-1,
\end{equation*}
and $\hat{H}_q(z;E_{nn},A_{n-1})_{i n}$  is  convergent  with   coefficients determined recursively,  for  $1\leq i\leq n$.      Here \begin{equation*}
    d_{q;ij}:=\left\{
          \begin{array}{lr}
             \frac{q(1-q) a_{i}^{(n-1)} b_{j}^{(n-1)}}{\left(q^{\nu_{j}}-q^{1+\nu_{i}}\right)z},   & \text{if} \ \ 1\le i\neq   j\le n-1,  \\
           1, & \text{if} \ \ 1\le i=   j\le n-1,\\
           -\frac{b_{j}^{(n-1)}}{z},   & \text{if} \ \ i=n,\  1\le    j\le n-1.
             \end{array}
\right.
\end{equation*} 
We denote  $\boldsymbol{\mu}:=\left(\mu_1,...,\mu_n   \right)$,  $\boldsymbol{\nu_{\hat{j}}}:=\left( \nu_1,...,\widehat{\nu_j},...,\nu_{n-1}   \right)$,  and   $\boldsymbol{ \delta^{(n-1)}_{i\hat{j}}   }:=\Big(\delta_{i1},...,\widehat{\delta_{ij}},...,\delta_{i,n-1}\Big)$. The symbol $\delta_{ij}$ is the Kronecker delta   symbol.  So the index notations are
\begin{align*}
   & 1+\boldsymbol{\mu}-\nu_{j}-\delta_{ij}=\left(1+\mu_1-\nu_j-\delta_{ij},...,1+\mu_n-\nu_j-\delta_{ij}   \right),\\
   &\boldsymbol{\delta^{(n-1)}_{i \hat{j}}}+1+\boldsymbol{\nu_{\hat{j}}}-\nu_{j}-\delta_{ij}\\
   =&\Big(\delta_{i1}+ 1+\nu_1-\nu_j-\delta_{ij},...,\reallywidehatt{ \delta_{ij}+ 1+\nu_j-\nu_j-\delta_{ij}} ,...,\delta_{i,n-1}+1+\nu_{n-1}-\nu_j-\delta_{ij}  \Big).
\end{align*}
\end{prop}
\begin{proof}
    Plugging $\hat{F}_q(z;E_{nn},A_{n-1})$ into the equation \eqref{diagqdes} and comparing the coefficients $H_{m}$   for $z^{-m}$  of  $\hat{H}_q(z;E_{nn},A_{n-1})$, for $m\geq 0$ (with $H_0=\mathrm{Id}_n$), we see that ${H}_{m}$ satisfies
    \begin{equation*}
        E_{nn} {H}_{m+1}-{H}_{m+1}\frac{1}{q^{m+1}}E_{nn}=\frac{1}{1-q}\left({H}_{m}-{H}_{m} \frac{1}{q^m}\mathrm{diag}\left( q^{\nu_1},...,q^{\nu_{n-1}},  q^{\nu_n} \right)\right)-A_{n-1} {H}_{m},
    \end{equation*}
    and can be computed explicitly.
\end{proof}
From Proposition \ref{prop: qsummability},   the entries of (the first $n-1$ columns of)  $\hat{H}_q(z;E_{n n},A)$  are $[\lambda;q]$-summable for any $\lambda\in \mathbb{C}^{*}\backslash \cup_{1\leq j \leq n-1}[-{(1-q)}/{q^{\nu_j}};q]$, and  from Proposition  \ref{qsum:convergent},   Proposition \ref{fczclambda}  and  Proposition \ref{ffs qsym},      the  fundamental solution of the equation \eqref{diagqdes} is  give by  
\begin{prop}\label{FinflamAn-1}
   Under the conditions  \eqref{eq:vivj}  (with the  condition \eqref{ineq1} included)  and       \eqref{ineq2},  for       $\lambda\in \mathbb{C}^{*}\backslash\cup_{1\leq l \leq n-1} [-{(1-q)}/{q^{\nu_l}};q]$,  the    fundamental solution  ${F}_q^{(\infty)}(z,\lambda;E_{nn},A_{n-1})$ of  \eqref{diagqdes} is given by 
   \begin{equation}\label{eq:Fqinf(z)}
    {F}_q^{(\infty)}(z,\lambda;E_{nn},A_{n-1})=H^{(\infty)}_q(z,\lambda;E_{nn},A_{n-1})\cdot   z^{\delta^{(n-1)}_q(A_{n-1})}   e_q\left( E_{n n} q^{ -\delta^{(n-1)}_q(A_{n-1}) } z \right),
\end{equation}
with  ${H}^{(\infty)}_q\left(z,\lambda;E_{nn},A_{n-1}\right)_{ij}=\mathcal{L}_{q ; 1}^{[\lambda;q]} \circ \hat{\mathcal{B}}_{q ; 1} (\hat{H}_q\left(z;E_{nn},A_{n-1}\right)_{ij})$. The entries  are for  $1\leq i\leq n$, $1\leq j\leq n-1$:  
   \begin{align}
&H^{(\infty)}_q(z,\lambda;E_{nn},A_{n-1})_{  i  j  }=  d_{q;ij}\cdot  {}_nf_{n-2}\left(\begin{array}{c}
			 q^{1+\boldsymbol{\mu}-\nu_j-\delta_{ij}} \\
			q^{1+\boldsymbol{\nu_{\hat{j}}}-\nu_j+\boldsymbol{\delta^{(n-1)}_{i\hat{j}}}-\delta_{ij} }
		\end{array}; \frac{\lambda   q^{\nu_j+\delta_{ij}-\delta_{ni}}  }{1-q}; q,   {\frac{ {q^{\nu_j+\delta_{ij}-\delta_{ni}   }   } }{ (1-q)z      }}\right),   \nonumber    \\
  &H^{(\infty)}_q(z,\lambda;E_{nn},A_{n-1})_{  i  n  }=\hat{H}_q(z;E_{nn},A_{n-1})_{i n},\label{Hqinfin}
   \end{align}
where     \eqref{Hqinfin}  holds  on the disc of convergence  of  $\hat{H}_q(z;E_{nn},A_{n-1})_{i n}$  minus the points of the $q$-spiral $[-\lambda^{-1};q]$.
\end{prop}

By direct computation, the   fundamental solution around the origin of  \eqref{diagqdes} is given by
\begin{prop}\label{qsolu zero}
Under   the nonresonant condition \eqref{qnreso}  and  the conditions \eqref{ineq1},  \eqref{ineq2},  the fundamental solution  ${F}^{(0)}_q(z;E_{nn},A_{n-1})$ of  \eqref{diagqdes} is given by (recalling   $\log_q(1-(1-q) A_n )=\operatorname{diag}(\mu_1,...,\mu_{n})$)
    \begin{equation}\label{eq:F0q(z)}
    F_q^{(0)}(z;E_{nn},A_{n-1})=H_q^{(0)}(z;E_{nn},A_{n-1}) \cdot z^{\log_q(1-(1-q) A_n )   },
\end{equation}
  with $(i,j)$-th entry $H_q^{(0)}(z;E_{nn},A_{n-1})_{i j}$ given {by}
    \begin{align*}
&    H_q^{(0)}(z;E_{nn},A_{n-1})_{i j}=\frac{a_i^{(n-1)}}{\lambda_{j}^{(n)}-\lambda^{(n-1)}_{i}} {}_n \varphi_{n-1}\left(\begin{array}{c}
q^{1+\mu_{j}-\boldsymbol{{\nu}}-\boldsymbol{{\delta}^{(n-1)}_{ i}}}, 0 \\
q^{1+\mu_{j}-\boldsymbol{\mu_{\hat{j}}}}
\end{array} ; q, \frac{(1-q) z}{q^{\nu_n}}\right),\quad 1\leq i \leq n-1,\, 1\leq j \leq n,\\
&H_q^{(0)}(z;E_{nn},A_{n-1})_{n j}={ }_{n} \varphi_{n-1}\left(\begin{array}{c}
q^{1+\mu_{j}-\boldsymbol{\nu}}, 0 \\
q^{1+\mu_{j}-\boldsymbol{\mu_{\hat{j}}}   }
\end{array} ;q, \frac{(1-q) z}{q^{\nu_n}}\right), \quad 1\leq j \leq n.
\end{align*}
Here we denote $\boldsymbol{\nu}:=\left( \nu_1,...,\nu_{n-1}  \right)$,  $\boldsymbol{\mu_{\hat{j}}}:=\big( \mu_1,...,\widehat{\mu_{j}},...,\mu_n  \big)$  and     $\boldsymbol{\delta^{(n-1)}_i}:=\left(\delta_{i1},...,\delta_{i,n-1}\right)$.  So
\begin{align*}
    &1+\mu_{j}-\boldsymbol{{\nu}}-\boldsymbol{{\delta}^{(n-1)}_{ i}}
    =\left(1+\mu_j-\nu_1-\delta_{i1},..., 1+\mu_j-\nu_{n-1}-\delta_{i,n-1}  \right),\\
    &1+\mu_j-\boldsymbol{\mu_{\hat{j}}}=\left(  1+\mu_j-\mu_1,...,1+\mu_j-\mu_{j-1},\widehat{1+\mu_j-\mu_j},1+\mu_j-\mu_{j+1},...,1+\mu_j-\mu_n \right),\\
    &1+\mu_{j}-\boldsymbol{\nu}=\left(1+\mu_{j}-{\nu}_1,...,1+\mu_{j}-{\nu}_{n-1}  \right).
\end{align*}
\end{prop}

\subsection{Relationship between the System \eqref{diagqdes} and the Equation \eqref{intro:conflu hyper eq}}\label{sec:relation}
The computation of the connection matrix and $q$-Stokes matrix of \eqref{diagqdes}  can be transformed  into the connection problem of the equation  \eqref{intro:conflu hyper eq}. In this subsection, we clarify  the relationship between  the system \eqref{diagqdes} and the equation \eqref{intro:conflu hyper eq}.

 From  explicit expressions of fundamental solutions  in Proposition \ref{FinflamAn-1}  and   Proposition   \ref{qsolu zero},  one gets that:  
\begin{itemize}
    \item For  $1\leq i,k \leq n-1$, $1\leq  j \leq n$,  $F_q^{(0)}(z;E_{nn},A_{n-1})_{i j}$  and   $F_q^{(\infty)}(z,\lambda;E_{nn},A_{n-1})_{i k}$  satisfy the equation 
    \begin{equation}\label{Fij eq}
        \prod_{l=1}^{n}(z\cdot D_q-\lambda^{(n)}_l)y(z)=\prod_{l=1}^{n-1}\left( \frac{q^{-\delta_{l i} }  \sigma_q-1}{q-1}-\lambda^{(n-1)}_l\right)(z\cdot y(z)),
    \end{equation}
where   $\delta_{l i}$  is the Kronecker delta   symbol.    It follows that  $F_q^{(\infty)}(z,\lambda;E_{nn},A_{n-1})_{i n}$ also satisfies \eqref{Fij eq}, since  
\begin{equation*}
    F_q^{(\infty)}(z,\lambda;E_{nn},A_{n-1})_{i n}=\frac{F_q^{(0)}(z;E_{nn},A_{n-1})_{ ij }}{ U_q(z,\lambda;E_{nn},A_{n-1})_{n  j} } -\sum_{k=1}^{n-1}\frac{  U_q(z,\lambda;E_{nn},A_{n-1})_{kj}  }{U_q(z,\lambda;E_{nn},A_{n-1})_{n  j}}F_q^{(\infty)}(z,\lambda;E_{nn},A_{n-1})_{i k},
\end{equation*}
which is a sum  with  pseudo-constant   coefficients.   In fact,    the  function  $F_q^{(\infty)}(z,\lambda;E_{nn},A_{n-1})_{i n}z^{ -\mu_j }$  satisfies the  equation (see $\nu_n$ in \eqref{eq:vn}):
\begin{equation*}
      (1-q)z\prod_{l=1}^{n-1}(1-q^{1+\mu_j-{\nu}_l-{\delta}_{l i}} \sigma_q)y(z)= q^{ \nu_n   }  (1-\sigma_q)\prod_{l=1,l\neq  j}^{n}\left(1-q^{\mu_{j}-{\mu}_{l   } }\sigma_q\right)y(z),
 \end{equation*}
 which  admits  the unique convergent solution       $\frac{ \psi_{i j}(z)  z^{ \nu_n -\mu_j -1   }   }{ \left(  q^{  -\nu_n  }       (1-q)z; q\right)_\infty    }$,    
  such that  $\psi_{i j}(z)$  is  a power
series in  the    variable   $1/z$,   with the  constant term  $1$.   Then  we  have  
\begin{equation}\label{expli Fin}
F_q^{(\infty)}(z,\lambda;E_{nn},A_{n-1})_{i n}=     \frac{   qa_i^{(n-1)}   \psi_{i j}(z)  z^{ \nu_n -1  }    }{\left(  q^{  -\nu_n  }  (1-q)z; q\right)_\infty},
\end{equation}   
     by the  direct computations  involving  recurrence relations  of  power series coefficients  of  $\psi_{i j}(z)$    and     $F_q^{(\infty)}(z,\lambda;E_{nn},A_{n-1})_{i n}$.  
    \item   For   $i=n$, $1\leq  k \leq   n-1$,  $1\leq  j  \leq  n$,  $F_q^{(0)}(z;E_{nn},A_{n-1})_{n  j}$  and  $F_q^{(\infty)}(z,\lambda;E_{nn},A_{n-1})_{n  k}$   satisfy the  equation 
     \begin{equation}\label{Fnj eq}
        \prod_{l=1}^{n}(z\cdot D_q-\lambda^{(n)}_l)y(z)=\prod_{l=1}^{n-1}( z\cdot   D_q-\lambda^{(n-1)}_l)(z\cdot y(z)).
    \end{equation}
    It also follows that  $F_q^{(\infty)}(z,\lambda;E_{nn},A_{n-1})_{n n }  $   satisfies \eqref{Fnj eq} and   
    \begin{equation}\label{expli Fnn}
        F_q(z,\lambda;E_{nn},A_{n-1})_{n n }=\frac{ \psi_{n j}(z) z^{ \nu_n  } }{  \left(  q^{  -\nu_n  }  (1-q)z; q\right)_\infty  },
    \end{equation}
    given  the  power series  $\psi_{n j}(z)$  for the variable  $1/z$,  with the constant term $1$,  such  that  $\frac{ \psi_{n j}(z) z^{ \nu_n-\mu_j  } }{  \left(  q^{  -\nu_n  }  (1-q)z; q\right)_\infty  }$  satisfies  the  equation
    \begin{equation*}
         (1-q)z\prod_{l=1}^{n-1}(1-q^{1+\mu_j-{\nu}_l   } \sigma_q)y(z)= q^{ \nu_n   }  (1-\sigma_q)\prod_{l=1,l\neq  j}^{n}\left(1-q^{\mu_{j}-{\mu}_{l   } }\sigma_q\right)y(z).
    \end{equation*}
\end{itemize}
So  the  coefficients of the connection formula of the equations \eqref{Fij eq}, \eqref{Fnj eq}, which can be derived from Theorem \ref{intro:our connection formula}   (or Theorem \ref{our connection formula} equivalently),  provide  the  entries of $U_q(z,\lambda;E_{nn},A_{n-1})$.

\subsection{Explicit Evaluation of Connection Matrix   and         $q$-Stokes  Matrix  of \eqref{diagqdes}}\label{sec:diagqdes}
In this subsection,  we    compute the connection matrix  $U_q(z,\lambda;E_{nn},A_{n-1})$   (see Theorem \ref{qstokes matrix})   and  the  $q$-Stokes matrix  $S_q\left(z,\lambda,\mu;E_{nn}, A_{n-1}\right)$  (see Theorem \ref{thm:diagqstokes}) of  the system \eqref{diagqdes}.  The computation depends on the connection formula  of the equation \eqref{intro:conflu hyper eq} (see Theorem \ref{intro:our connection formula} or Theorem \ref{our connection formula} equivalently).

We  introduce the following conditions:
\begin{equation}\label{cond:uiujvivj}
\begin{aligned}
    &   \mu_i-\mu_j\notin \mathbb{Z}, \  &&\text{for all}  \  1\leq i,j \leq  n, \ (\text{with  the  nonresonant   condition \eqref{qnreso} included}),  \\
    & \nu_i-\nu_j\notin  \mathbb{Z},  \  &&\text{for all}  \  1\leq i,j \leq  n-1,  \ (\text{with  the    condition \eqref{ineq1} included}), \\
    &\lambda^{\left(n-1\right)}_j   \neq   \lambda^{\left(n-2\right)}_l,  \  &&\text{for all}  \  1\leq j \leq  n-1, 1\leq l \leq n-2.
\end{aligned}
\end{equation}
\begin{thm}\label{qstokes matrix}
Under the condition \eqref{cond:uiujvivj},  and  for  $z,\lambda$  such that
\begin{equation*}
     \lambda\in \mathbb{C}^*,\ \lambda\notin -\frac{1-q}{q^{\nu_l}}q^{\mathbb{Z}}, \  \text{for} \   1\leq l \leq n-1, \  \text{and} \  z\in \mathbb{C}^*, \  z\notin \frac{q^{\nu_n} }{1-q} q^{\mathbb{Z}}, \  z\notin  - \frac{1}{\lambda}q^{\mathbb{Z}},
\end{equation*}
     the entries of the connection matrix $U_q(z,\lambda;E_{nn},A_{n-1})$ are (see \eqref{eq:vn} for $\nu_n$)
    \begin{align}
U_q(z,\lambda;E_{nn},A_{n-1})_{i j}=&
\frac{a_i^{(n-1)}}{\lambda_{j}^{(n)}-\lambda^{(n-1)}_{i}}  \frac{ \prod_{l=1,l\neq  i}^{n-1} ( q^{1+\mu_{j} -{{\nu_l   }}   }  ;q)_{\infty}   \prod_{l=1,l\neq   j}^{n}(   q^{1+\nu_i-{\mu}_{{l}}};q)_\infty}{  \prod_{l=1,l\neq   j}^n  (q^{1+\mu_{j}-{\mu}_{{l}}};q)_{\infty}\prod_{l=1,l\neq   i}^{n-1}(q^{1+\nu_{{i}}-\nu_{{l}}};q)_\infty}  \nonumber \\
& \times    \frac{\theta_q\left(-q^{ {\mu_j -\nu_i   -\nu_n }} (1-q)z\right)}{ \theta_q\left(-q^{-\nu_n }  (1-q)z\right) }  z^{  \mu_j-\nu_i  },  \  \text{for}  \   1\leq i \leq n-1,\  1\leq  j \leq n,   \label{Uij}    \\
 % -\frac{1}{{b^{(n-1)}_i}}  \frac{(q^{1+\alpha_j-\boldsymbol{\beta}_{\hat{i}}},q^{\beta_i-   \boldsymbol{\alpha}_{\hat{j}}}  ;q)_\infty}{(q^{1+\alpha_{j}-\boldsymbol{\alpha}_{\hat{j}}},q^{{\beta}_{{i}}-\boldsymbol{\beta}_{\hat{i}}};q)_\infty}
 % 		\frac{\theta_q\left(-q^{1+\alpha_j-\beta_i}\frac{\prod_{l=1}^{n-1}q^{\beta_l}}{\prod_{l=1}^{n}q^{\alpha_l}}(1-q)z\right)}{\theta_q\left(-\frac{\prod_{l=1}^{n-1}q^{\beta_l}}{\prod_{l=1}^{n}q^{\alpha_l}}(1-q)z\right)}z^{1+\alpha_j-\beta_i},
 %  \\
U_q(z,\lambda;E_{nn},A_{n-1})_{n j}=&
\frac{  \prod_{l=1}^{n-1} ( q^{1+\mu_{j}-\nu_{l}}  ;q)_\infty}{\prod_{l=1,l\neq   j}^{n}   (q^{1+\mu_{j}-\mu_{l}}   ;q)_\infty}  \nonumber  \\
&\times \frac{\prod_{l=1,l\neq j}^{n}\theta_q((1-q)q^{1-\mu_l}/\lambda)}{\prod_{l=1}^{n-1}\theta_q((1-q)q^{1-\nu_l}/\lambda)}      	\frac{\theta_q\left(  q^{ \mu_j-\nu_n  }  \lambda z\right)}{\theta_q(\lambda  z )}z^{ \mu_j-\nu_n     },  \  \text{for}  \  1\leq j \leq  n  .\label{Unj}
% \frac{(q^{1+\alpha_j-\boldsymbol{\beta}};q)_\infty}{(q^{1+\alpha_{j}-\boldsymbol{\alpha}_{\hat{j}}};q)_\infty}
% 		\frac{\prod_{l=1,l\neq j}^{n}\theta_q((1-q)q^{1-\alpha_l}/\lambda)}{\prod_{l=1}^{n-1}\theta_q((1-q)q^{1-\beta_l}/\lambda)}
% 		\frac{\theta_q\left(\lambda \frac{\prod_{l=1}^{n-1}q^{\beta_l}}{\prod_{l=1,l\neq j}^{n}q^{\alpha_l}} z\right)}{\theta_q(\lambda  z )}z^{\sum_{l=1}^{n-1}\beta_l-\sum_{l=1,l\neq j}^n\alpha_l}, 
\end{align}
\end{thm}
\begin{proof}
To prove this theorem, we only need to verify that when \eqref{Uij} and \eqref{Unj} hold, one has 
\begin{equation}\label{UF}
 F_q^{(0)}(z;E_{nn},A_{n-1})_{i j}      =\sum_{k=1}^{n}   U_q(z,\lambda;E_{nn},A_{n-1})_{k j} F^{(\infty)}_q(z,\lambda;E_{nn},A_{n-1})_{i k}, \  \text{for all} \  1\leq i,j \leq n.
\end{equation}
   For  every $i,j \in \{ 1,...,n   \}$,       we take the parameters    $\boldsymbol{a}=\left(a_1,...,a_{n-1}\right),\boldsymbol{b}=\left(b_1,...,b_{n-1}\right)$ in Theorem \ref{our connection formula} as   
\begin{align*}
    &{a}_l=q^{1+\mu_j-{\nu}_l-{\delta}_{l i}}, \  \text{for all}  \    1\leq l \leq   n-1,     \\
     &{b}_l=\left\{
          \begin{array}{lr}
          q^{1+\mu_{j}-{\mu}_{l   } }  , \   \text{for all} \    1\leq l \leq j-1,  \\
              q^{1+\mu_{j}-{\mu}_{l+1   } }  , \   \text{for all} \  j\leq l \leq  n-1,
             \end{array}
\right.
\end{align*}
   and replace $\lambda,z$ by  $  \frac{{\lambda q^{ 1-   \delta_{n i }+\mu_j}}}{1-q}, \frac{\prod_{l=1}^{n-1}q^{\nu_l}}{\prod_{l=1}^{n}q^{\mu_l}}(1-q)z$.  Therefore,  together with      Proposition \ref{FinflamAn-1},     Proposition \ref{qsolu zero},     \eqref{expli Fin}  and  \eqref{expli Fnn},   we have
   \begin{equation*}
    F_q^{(0)}(z;E_{nn},A_{n-1})_{i j}=\sum_{k=1}^{n} C_{ij,k}  F^{(\infty)}_q(z,\lambda;E_{nn},A_{n-1})_{i k}.
\end{equation*}
Here  $C_{ij,k}$ is given as follows.
\begin{itemize}
\item   When    $1\leq i   \leq  n-1$,   $1\leq j \leq n$,       
% \begin{align*}
%     % F_q(z)_{i j}=&\sum_{k=1}^{n-1} C_{ij,k}  F^{(\infty)}_q(z,\lambda)_{i k}\\
%     &+\frac{a_i^{(n)}}{\lambda_{j}^{(n)}-\lambda^{(n-1)}_{i}}
% 		\frac{  \prod_{l=1}^{n-1} ( q^{1+\mu_{j}-\nu_{l}-{{\delta}}_{l i}}  ;q)_\infty}{\prod_{l=1,l\neq   j}^{n}   (q^{1+\mu_{j}-\mu_{l}}   ;q)_\infty}
% 		\frac{\prod_{l=1,l\neq j}^{n}\theta_q\left(\frac{1-q}{q^{\mu_l}\lambda} \right)}{\prod_{l=1}^{n-1}\theta_q\left(  \frac{1-q}{ q^{\nu_l+\delta_{l i}} \lambda  }     \right)}  
% 		\frac{\theta_q\left( q^{1+\mu_j-d_n  }  {\lambda    }   z\right)}{\theta_q\left( \lambda    z  \right)}
% 		\frac{  \psi_{i  j}(z) z^{\mu_j} }{\left(\frac{(1-q)z}{  q^{d_n } }  ; q\right)_\infty}.
% \end{align*}
  we  denote   that  
\begin{align*}
    C_{ij,k}:=&\frac{1}{  \lambda_{j}^{(n)}-\lambda_{i}^{(n-1)}   }
		\frac{ \prod_{ l=1,l\neq  k  }^{n-1} (  q^{  1+\mu_j-\nu_l-\delta_{l i}  }  ;q)_{\infty}  \prod_{ l=1,l\neq  j  }^{n} ( q^{\nu_k -\mu_l  } ;q )_{\infty}    }{   \prod_{l=1,l\neq    j}^{n}   (q^{1+   \mu_{j}-\mu_l        };q)_{\infty}\prod_{l=1,l\neq   k}^{n-1}(q^{\nu_{{k}}-\nu_{l}-{\delta}_{l i}   };q)_\infty}\\
  &\times  
  \frac{\theta_q\left(  -  q^{ 1+\mu_j-\nu_k-\nu_n   }  (1-q)z  \right)}{ \theta_q\left(-q^{ -\nu_n } (1-q)z\right) }
		\frac{  \left(q^{\nu_{k}}-q^{1+\nu_{i}}\right)z^{ 1+\mu_j-\nu_k   } }{   q(1-q)  b_{k}^{(n-1)} },  
\end{align*}
   for all   $k\neq i,    n$,    and  
\begin{align*}
    &C_{ij,i}:=\frac{a_i^{(n-1)}}{\lambda_{j}^{(n)}-\lambda^{(n-1)}_{i}}  \frac{ \prod_{l=1,l\neq  i}^{n-1} ( q^{1+\mu_{j} -{{\nu_l   }}   }  ;q)_{\infty}   \prod_{l=1,l\neq   j}^{n}(   q^{1+\nu_i-{\mu}_{{l}}};q)_\infty}{  \prod_{l=1,l\neq   j}^n  (q^{1+\mu_{j}-{\mu}_{{l}}};q)_{\infty}\prod_{l=1,l\neq   i}^{n-1}(q^{1+\nu_{{i}}-\nu_{{l}}};q)_\infty}  
  \frac{\theta_q\left(-q^{ {\mu_j -\nu_i   -\nu_n }} (1-q)z\right)}{ \theta_q\left(-q^{-\nu_n }  (1-q)z\right) }  z^{  \mu_j-\nu_i  },\\
  &C_{ij,n}:=\frac{  \prod_{l=1}^{n-1} ( q^{1+\mu_{j}-\nu_{l}}  ;q)_\infty}{\prod_{l=1,l\neq   j}^{n}   (q^{1+\mu_{j}-\mu_{l}}   ;q)_\infty}   \frac{\prod_{l=1,l\neq j}^{n}\theta_q((1-q)q^{1-\mu_l}/\lambda)}{\prod_{l=1}^{n-1}\theta_q((1-q)q^{1-\nu_l}/\lambda)}
		\frac{\theta_q\left(  q^{ \mu_j-\nu_n  }  \lambda z\right)}{\theta_q(\lambda  z )}z^{ \mu_j-\nu_n     }.
\end{align*}
  \item   When  $i=n$,   $1\leq  j  \leq   n$,    
      we  denote that  
  \begin{equation*}
      C_{nj,k}:=- \frac{1}{{b^{(n-1)}_k}}  \frac{\prod_{ l=1,l\neq  k }^{n-1} (q^{1+\mu_j-{\nu}_{l}};q)_{\infty  }\prod_{ l=1,l\neq  j }^{n} (q^{\nu_i-\mu_l}  ;q)_\infty}{  \prod_{ l=1,l\neq  j  }^n  (q^{1+\mu_{j}-{\mu}_{l}};q)_{\infty}\prod_{ l=1,l\neq  k }^{n-1}(q^{{\nu}_{{i}}  -\nu_l   };q)_\infty}
		\frac{\theta_q\left(-q^{1+\mu_j-\nu_k -\nu_n  }(1-q)z\right)}{\theta_q\left(-q^{ -\nu_n  }(1-q)z\right)}z^{1+\mu_j-\nu_k}, 
  \end{equation*}
   for  all  $1\leq  k  \leq  n-1$,   and
  \begin{equation*}
        C_{nj,n}:=\frac{  \prod_{l=1}^{n-1} ( q^{1+\mu_{j}-\nu_{l}}  ;q)_\infty}{\prod_{l=1,l\neq   j}^{n}   (q^{1+\mu_{j}-\mu_{l}}   ;q)_\infty}   \frac{\prod_{l=1,l\neq j}^{n}\theta_q((1-q)q^{1-\mu_l}/\lambda)}{\prod_{l=1}^{n-1}\theta_q((1-q)q^{1-\nu_l}/\lambda)}
		\frac{\theta_q\left(  q^{ \mu_j-\nu_n  }  \lambda z\right)}{\theta_q(\lambda  z )}z^{ \mu_j-\nu_n     }.
  \end{equation*}
\end{itemize}

Comparing the  explicit expressions of   $C_{ij,k}$  and   $U_q\left(z,\lambda;E_{nn},A_{n-1}   \right)_{k j}$  given  in    \eqref{Uij},    \eqref{Unj},   by   \eqref{eq:aibi}   one  can find that  \[
    U_q\left(z,\lambda;E_{nn},A_{n-1}   \right)_{k j}=C_{ ij,k  }, \  \text{for any} \ 1\leq  i,j,k\leq  n.  \]         Thus,   \eqref{UF}   is true     and  this theorem  holds.
\end{proof}
% from which we get the entries of $U_q\left(z,\lambda   \right)$ explicitly.
\begin{thm}
Under the condition  \eqref{cond:uiujvivj},  and  for  $z,\lambda$  such that
\begin{equation*}
     \lambda\in \mathbb{C}^*,\ \lambda\notin -\frac{1-q}{q^{\nu_l}}q^{\mathbb{Z}}, \  \text{for} \   1\leq l \leq n-1, \  \text{and} \  z\in \mathbb{C}^*, \  z\notin \frac{q^{\nu_n} }{1-q} q^{\mathbb{Z}}, \  z\notin  - \frac{1}{\lambda}q^{\mathbb{Z}},
\end{equation*}
     the entries of $U_q(z,\lambda;E_{nn},A_{n-1})^{-1}$   are  for $1\leq i \leq n$, $1\leq j \leq n-1$:
    \begin{align}
    \left(  U_q(z,\lambda;E_{nn},A_{n-1})^{-1}   \right)_{i j}=&   -b^{(n-1)}_j \frac{ \prod_{ l=1,l\neq  i  }^n  ( q^{1+\mu_l-{\nu}_j} ;q )_{\infty} \prod_{l=1,l\neq  j}^{n-1}  ( q^{ \nu_l-\mu_i  }  ;q)_{\infty} }{ \prod_{l=1,l\neq  j}^{n-1} ( q^{ 1+\nu_l-\nu_j  } ;q)_{\infty}\prod_{l=1,l\neq  i}^n ( q^{\mu_{ l }-\mu_i} ;q)_{\infty} }\nonumber   \\
      &\times     \frac{ \theta_q\left( {  \lambda z }/{ q^{  1+\mu_i-\nu_j  }    }  \right)  }{ \theta_q\left(  \lambda z   \right)   }  \frac{ \theta_q \left(  {{\lambda q^{ \mu_i       }}}/{(1-q)} \right)  }{ \theta_q\left(   {{\lambda q^{    \nu_j  -1  }}}/{(1-q)}  \right)  }  z^{  \nu_j-\mu_i -1 },  \label{U-1 ij}              \\
% \left(  U_q(z,\lambda)^{-1}   \right)_{i j}=&N_{ij,i}\frac{ \prod_{ l=1,l\neq  i  }^n  ( q^{1+\mu_l-{\nu}_j-{\delta}_{i j}} ;q )_{\infty} \prod_{l=1,l\neq  j}^{n-1}  ( q^{ \nu_l-\mu_i+\delta_{l i}  }  ;q)_{\infty} }{ \prod_{l=1,l\neq  j}^{n-1} ( q^{ 1+\nu_l+\delta_{l i}-\nu_j-\delta_{i j} } ;q)_{\infty}\prod_{l=1,l\neq  i}^n ( q^{\mu_{ l }-\mu_i} ;q)_{\infty} } \nonumber  \\
% &\cdot 
% \frac{ \theta_q\left( \frac{ q^{  2+\mu_i-\nu_j  }  }{  q^{  \delta_{i  j} } \lambda z  }  \right)  }{ \theta_q\left( \frac{q}{ \lambda  z  }  \right)   }  \frac{ \theta_q \left(  \frac{{\lambda q^{ 1+\mu_i     -\delta_{n i }  }}}{1-q} \right)  }{ \theta_q\left(   \frac{{\lambda q^{    \nu_j+\delta_{i j  }  -\delta_{n i }  }}}{1-q}  \right)  }  z^{  \nu_j-\mu_i-1+\delta_{i j}  }   ,  \label{U-1 ij}  \\
\left(  U_q(z,\lambda;E_{nn},A_{n-1})^{-1}   \right)_{i n}=&       \frac{ \prod_{l=1}^{n-1}\left( q^{ {\nu}_l    -\mu_i };q  \right)_{\infty}   }{  \prod_{l=1,l\neq i}^{n}\left(q^{ \mu_l-\mu_i };q  \right)_{\infty}     }  \frac{\theta_q  \left( -{  q^{1+\mu_i}      } \big / ( {(1-q)z  } )  \right)  }{\theta_q \left(-q^{-\nu_n}(1-q)z\right)  }   z^{\nu_n-\mu_i   }.      \label{U-1 in} 
% \\
% \left(  U_q(z,\lambda)^{-1}   \right)_{i n}=&    N_{in,i}  \frac{ \prod_{l=1}^{n-1}\left( q^{ {\nu}_l+{\delta}_{l  i}    -\mu_i };q  \right)_{\infty}   }{  \prod_{l=1,l\neq i}^{n}\left(q^{ \mu_l-\mu_i };q  \right)_{\infty}     }  \frac{\theta_q  \left( -\frac{  q^{2-\delta_{in}+\mu_i}      }{(1-q)z  }   \right)  }{\theta_q \left(-q^{-d_n}(1-q)z\right)  }   z^{d_n-\mu_i-1+\delta_{i n}},    \label{U-1 in}
\end{align}
% where  $N_{i j,i}$ is the special case when  $k=i$ of   the  constant $N_{i j,k}$   given as follows:
\end{thm}
\begin{proof}
To prove this theorem, we only need to verify that when \eqref{U-1 ij} and \eqref{U-1 in} hold, one has 
\begin{equation}\label{U-1 F}
    F^{(\infty)}_q(z,\lambda;E_{nn},A_{n-1})_{i j}   =\sum_{k=1}^{n} \left(  U_q(z,\lambda;E_{nn},A_{n-1})^{-1}   \right)_{k j} F_q^{(0)}(z;E_{nn},A_{n-1})_{i k}, \  \text{for all} \  1\leq i,j \leq n.
\end{equation}
Set the  constant $N_{i j,k}$    as follows:
\begin{equation*}
    N_{ij,k}:=\left\{
          \begin{array}{lr}
     \frac{q(1-q)  b_{j}^{(n-1)}  \left( \lambda^{(n)}_{k}-\lambda^{(n-1)}_{i}  \right)  }{q^{\nu_{j}}-q^{1+\nu_{i}   }   }       ,& \   \text{for all} \    1\leq i\neq  j \leq n-1,    1\leq  k  \leq  n,  \\
   \frac{{ \lambda^{(n)}_k-\lambda^{(n-1)}_i  }  }{a^{(n-1)}_i}       ,& \   \text{for all} \    1\leq  i =j  \leq  n-1,  1\leq  k  \leq  n,  \\
             -{b_{j}^{(n-1)}}, &  \   \text{for all} \   i=n, 1\leq j \leq  n-1,  1\leq  k\leq  n,  \\
    q     { \left(\lambda^{(n)}_k-\lambda^{(n-1)}_i  \right) }   ,& \   \text{for all} \  1\leq   i  \leq  n-1,  j=n,  1\leq  k  \leq  n, \\
    1   ,& \   \text{for all} \    i =  n,    j=n,  1\leq  k  \leq  n.
             \end{array}
\right.
\end{equation*}
Then we have:
\begin{itemize}
\item      For $F^{(\infty)}_q(z,\lambda;E_{nn},A_{n-1})_{i j}$,    $1\leq  i\leq  n$,    $1\leq   j\leq  n-1$,   we   take  the  parameters  $\boldsymbol{a}=\left(a_1,...,a_{n}\right),\boldsymbol{b}=\left(b_1,...,b_{n-2}\right)$ in Lemma \ref{Adachi} as
    \begin{align*}
    &{a}_l=q^{1+\mu_l-{\nu}_j-{\delta}_{i j}}, \  \text{for all}  \    1\leq l \leq   n,     \\
     &{b}_l=\left\{
          \begin{array}{lr}
          q^{1+\nu_{l}+\delta_{l i  }-{\nu}_{j   }-\delta_{i j} }  , \   \text{for all} \    1\leq l \leq j-1,  \\
              q^{1+\nu_{l+1} +\delta_{l+1,i}  -\nu_{j}-{\delta}_{i j   } }  , \   \text{for all} \  j\leq l \leq  n-2,
             \end{array}
\right.
\end{align*}
   and replace $\lambda,z$ by  $  \frac{{\lambda q^{    \nu_j+\delta_{i j  }  -\delta_{n i }  }}}{1-q}, \frac{{ q^{    \nu_j+\delta_{i j  }  -\delta_{n i }  }}}{(1-q)z}$.    Therefore, together with  Proposition \ref{FinflamAn-1}  and      Proposition \ref{qsolu zero},     it  follows  that  
       \begin{equation*}
 F^{(\infty)}_q(z,\lambda;E_{nn},A_{n-1})_{i j}   =\sum_{k=1}^{n} M_{ij,k} F_q^{(0)}(z;E_{nn},A_{n-1})_{i k},
\end{equation*}
where   
\begin{align*}
    M_{ij,k}:=&N_{ij,k}\frac{ \prod_{ l=1,l\neq  k  }^n  ( q^{1+\mu_l-{\nu}_j-{\delta}_{i j}} ;q )_{\infty} \prod_{l=1,l\neq  j}^{n-1}  ( q^{ \nu_l-\mu_k+\delta_{l i}  }  ;q)_{\infty} }{ \prod_{l=1,l\neq  j}^{n-1} ( q^{ 1+\nu_l+\delta_{l i}-\nu_j-\delta_{i j} } ;q)_{\infty}\prod_{l=1,l\neq  k}^n ( q^{\mu_{ l }-\mu_k} ;q)_{\infty} } \\
    & \times  \frac{ \theta_q\left( { q^{  2+\mu_k-\nu_j-\delta_{i  j}  }  }/{(   \lambda z ) }  \right)  }{ \theta_q\left( {q}/{ (\lambda  z)  }  \right)   }  \frac{ \theta_q \left(  {\lambda q^{ 1+\mu_k     -\delta_{n i }  }}/({1-q}) \right)  }{ \theta_q\left(   {\lambda q^{    \nu_j+\delta_{i j  }  -\delta_{n i }  }}/{(1-q)}  \right)  }  z^{  \nu_j-\mu_k -1+\delta_{i j}},  \   \text{for all} \    1\leq k \leq n. 
\end{align*}
\item   For  $F^{(\infty)}_q(z,\lambda;E_{nn},A_{n-1})_{i n}$,   $1\leq  i  \leq n$,   we take the parameters    $\boldsymbol{a}=\left(a_1,...,a_{n-1}\right),\boldsymbol{b}=\left(b_1,...,b_{n-1}\right)$ in Corollary \ref{corfn} as  
    \begin{align*}
    &{a}_l=q^{1+\mu_n-{\nu}_l-{\delta}_{l  i}}, \  \text{for all}  \    1\leq l \leq   n-1,     \\
     &{b}_l=q^{1+\mu_{n}-\mu_{l} }  , \   \text{for all} \    1\leq l \leq n-1,  \end{align*}
 % where   $j$ can take any number from the set $\{1, ..., n\}$,   
 and  replace $z$  by  $q^{-\nu_n}(1-q)z$.   Then, together with \eqref{expli Fin}, \eqref{expli Fnn}  and      Proposition \ref{qsolu zero}  we  have
 \begin{equation*}
     F^{(\infty)}_q(z,\lambda;E_{nn},A_{n-1})_{i n}=\sum_{k=1}^n M_{i n,k} F_q^{(0)}(z;E_{nn},A_{n-1})_{i k}, 
 \end{equation*}
 where  
 \begin{equation*}
     M_{in,k}:= 
     % \left\{
          % \begin{array}{lr}
  N_{in,k}   \frac{ \prod_{l=1}^{n-1}\left( q^{ {\nu}_l+{\delta}_{l  i}    -\mu_k };q  \right)_{\infty}   }{  \prod_{l=1,l\neq k}^{n}\left(q^{ \mu_l-\mu_k };q  \right)_{\infty}     }  \frac{\theta_q  \left( -{  q^{2-\delta_{in}+\mu_k}      }\big/{((1-q)z)  }   \right)  }{\theta_q \left(-q^{-\nu_n}(1-q)z\right)  }   z^{\nu_n-\mu_k  -1+\delta_{ i n  } } 
  % \frac{ \prod_{l=1}^{n-1}\left( q^{   1+\mu_j  -\mu_k }/q^{1+\mu_j-{\nu}_l-{\delta}_{l  i}};q  \right)_{\infty}   }{  \prod_{l=1,l\neq k}^{n}\left(q^{ 1+\mu_j-\mu_k }/{  q^{ 1+\mu_j-\mu_l } };q  \right)_{\infty}     }  \frac{\theta_q  \left( -\frac{q\prod_{l=1,l\neq k}^{n}q^{1+\mu_{j}-\mu_{l} }}{q^{-d_n}(1-q)z\prod_{l=1}^{n-1}q^{1+\mu_j-{\nu}_l-{\delta}_{l  i}}}   \right)  }{\theta_q \left(-q^{-d_n}(1-q)z\right)  }     
  , \   \text{for all} \    1\leq k \leq n.  
    % \frac{ \prod_{l=1}^{n-1}\left(  q^{1+\mu_{j}-\mu_{k+1}  }   /q^{1+\mu_j-{\nu}_l-{\delta}_{l  i}};q  \right)_{\infty}   }{  \prod_{l=1,l\neq k+1}^{n}\left(  q^{1+\mu_{j}-\mu_{k+1}  }  /q^{1+\mu_{j}-\mu_{l} };q  \right)_{\infty}     }   \frac{\theta_q \left( -\frac{q\prod_{l=1,l\neq k+1}^{n}q^{1+\mu_{j}-\mu_{l} } }{q^{-d_n}(1-q)z\prod_{l=1}^{n-1}q^{1+\mu_j-{\nu}_l-{\delta}_{l  i}}}   \right)    }{  \theta_q\left(-q^{-d_n}(1-q)z\right)  }    
 % N_{in,k}      \frac{ \prod_{l=1}^{n-1}\left(  q^{{\nu}_l+{\delta}_{l  i}-\mu_n}  ;q  \right)_{\infty}   }{ \prod_{l=1}^{n-1} \left( q^{\mu_{l}-\mu_{n} }   ;q  \right)_{\infty}     }  \frac{\theta_q \left(-\frac{q^{2-\delta_{in}+\mu_n }  }{ (1-q)z }   \right)    }{  \theta_q\left(-q^{-d_n}(1-q)z\right)  }  z^{d_n-\mu_n-1+\delta_{i n}  }     ,&  \text{for} \      k =  n,
             % \end{array}
% \right.   
 \end{equation*}
 % with  $N_{in,k}$  given in this theorem.
 \end{itemize}
Using \eqref{eq:aibi},  one can show that if the  explicit expressions of  $U_q(z,\lambda;E_{nn},A_{n-1})^{-1}$ are given by    \eqref{U-1 ij} and \eqref{U-1 in},  then
\[
 \left(  U_q(z,\lambda;E_{nn},A_{n-1})^{-1}   \right)_{k s}   =M_{ is,k  }, \  \text{for any} \    1\leq  i,s,k\leq  n.   \]
          Thus,    \eqref{U-1 F}   is true     and  this theorem  holds.
\end{proof}

\begin{thm}\label{thm:diagqstokes}
Under the condition  \eqref{cond:uiujvivj},  and    for  $z,\lambda,\mu$  such that
\begin{equation*}
     \lambda,\mu \in \mathbb{C}^*,\ \lambda,\mu \notin -\frac{1-q}{q^{\nu_l}}q^{\mathbb{Z}}, \  \text{for} \   1\leq l \leq n-1, \  \mu\notin \lambda q^{\mathbb{Z}},    \  \text{and} \  z\in \mathbb{C}^*, \    z\notin  - \frac{1}{\lambda}q^{\mathbb{Z}},\    z\notin  - \frac{1}{\mu}q^{\mathbb{Z}},
\end{equation*}
  the $q$-Stokes matrix $S_q\left(z,\lambda,\mu;E_{nn}, A_{n-1}\right)$ of the system \eqref{diagqdes} takes the block matrix form
\begin{equation}\label{eq:qstokesdiag}
S_q\left(z,\lambda,\mu;E_{nn}, A_{n-1}\right)=\left(\begin{array}{cc}
    \mathrm{Id}_{n-1} & 0  \\
    b_q^{(n)} & 1
  \end{array}\right),
  \end{equation}
  where $\mathrm{Id}_{n-1}$ is the $(n-1)\times (n-1)$ identity matrix, $b_{q}^{(n)}=\left(\big(b_q^{(n)}\big)_1,...,\big(b_{q}^{(n)}\big)_{n-1}\right)$ with  (see \eqref{eq:vn} for $\nu_n$)
\begin{align*} 
   \big(b_q^{(n)}\big)_j =& -b^{(n-1)}_j \sum_{k=1}^n \frac{  \prod_{l=1}^{n-1} ( q^{1+\mu_{k}-\nu_{l}}  ;q)_\infty}{\prod_{l=1,l\neq   k}^{n}   (q^{1+\mu_{k}-\mu_{l}}   ;q)_\infty}  \frac{ \prod_{ l=1,l\neq  k  }^n  ( q^{1+\mu_l-{\nu}_j} ;q )_{\infty} \prod_{l=1,l\neq  j}^{n-1}  ( q^{ \nu_l-\mu_k  }  ;q)_{\infty} }{ \prod_{l=1,l\neq  j}^{n-1} ( q^{ 1+\nu_l-\nu_j  } ;q)_{\infty}\prod_{l=1,l\neq  k}^n ( q^{\mu_{ l }-\mu_k} ;q)_{\infty} }        \\
  &\times  \frac{\prod_{l=1,l\neq k}^{n}\theta_q((1-q)q^{1-\mu_l}/\lambda)}{\prod_{l=1}^{n-1}\theta_q((1-q)q^{1-\nu_l}/\lambda)}   
 \frac{\theta_q\left(  q^{ \mu_k-\nu_n  }  \lambda z\right)}{\theta_q(\lambda  z )}         \frac{ \theta_q\left( { q^{\nu_j-\mu_k-1} \mu z }  \right)  }{ \theta_q\left(  \mu z   \right)   }  \frac{ \theta_q \left(  {{\mu q^{ \mu_k       }}}/{(1-q)} \right)  }{ \theta_q\left(   {{\mu q^{    \nu_j  -1  }}}/{(1-q)}  \right)  }  z^{  \nu_j-\nu_n -1 }  .
\end{align*}
\end{thm}
\begin{proof}
From  Definition \ref{UqSq},   we get 
\begin{equation*}
U_q(z,\lambda;E_{nn},A_{n-1})=S_q(z,\lambda,\mu;E_{nn},A_{n-1})U_q(z,\mu;E_{nn},A_{n-1}),
\end{equation*}
from  which       the  entries of   $S_q(z,\lambda,\mu;E_{nn},A_{n-1})$   can be derived  explicitly.
\end{proof}

\subsection{Explicit Evaluation of Connection Matrix   and         $q$-Stokes  Matrix  of \eqref{qdes}}\label{sec:qdes}

Since the matrices $E_{n n}$ and $\operatorname{diag}\left({P}_{n-1},1\right)$ commute with each other, we get
\begin{prop}\label{conjugation}
Under the  condition \eqref{eq:vivj} (with   the   condition \eqref{ineq1} included), the nonresonant condition \eqref{qnreso},    and  condition \eqref{ineq2},   for       $\lambda\in \mathbb{C}^{*}\backslash \cup_{1\leq l \leq n-1} [-{(1-q)}/{q^{\nu_l}};q]$,     we  have
\begin{align*}
&     {F}_q^{(\infty)}(z,\lambda;E_{nn},A)=\operatorname{diag}\left({P}_{n-1},1\right) {F}_q\left(z,\lambda;E_{nn},A_{n-1}\right)\operatorname{diag}\left({P}_{n-1}^{-1},1\right),  \\
&    {F}^{(0)}_q(z;E_{nn},A)=\operatorname{diag}\left({P}_{n-1},1\right) {F}^{(0)}_q\left(z;E_{nn},A_{n-1}\right)  .
\end{align*}
   Moreover, the  connection matrix and  $q$-Stokes matrices   of the systems \eqref{qdes} and \eqref{diagqdes} are related by
\begin{align*}
&S_q\left(z,\lambda,\mu;E_{nn}, A\right)=\operatorname{diag}\left({P}_{n-1},1\right)S_{q}\left(z,\lambda,\mu;E_n, A_{n-1}\right)\operatorname{diag}\left({P}_{n-1}^{-1},1\right),\\
&U_q\left(z,\lambda;E_{nn},A   \right)=\operatorname{diag}\left({P}_{n-1},1\right) U_q\left(z,\lambda;E_{nn},A_{n-1}   \right).
\end{align*}
\end{prop}
Following Proposition \ref{conjugation}, the explicit expression of the $q$-Stokes matrix of \eqref{qdes} is given below and the connection matrix of \eqref{qdes}  can be determined in a similar manner.
\begin{thm}\label{main q-stokes}
Under the condition \eqref{cond:uiujvivj},    and  for  $z,\lambda,\mu$  such that
\begin{equation*}
     \lambda,\mu  \in \mathbb{C}^*,\ \lambda,\mu  \notin -\frac{1-q}{q^{\nu_l}}q^{\mathbb{Z}}, \  \text{for} \   1\leq l \leq n-1, \  \mu\notin \lambda q^{\mathbb{Z}},      \  \text{and} \  z\in \mathbb{C}^*, \  z\notin - \frac{1}{\mu}q^{\mathbb{Z}}, \  z\notin  - \frac{1}{\lambda}q^{\mathbb{Z}},
\end{equation*}
 the $q$-Stokes matrix $S_q\left(z,\lambda,\mu;E_{nn}, A\right)$ of the system \eqref{qdes} takes the block matrix form
\begin{equation*}
S_q\left(z,\lambda,\mu;E_{nn}, A\right)=\left(\begin{array}{cc}
    \mathrm{Id}_{n-1} & 0  \\
    b_q & 1
  \end{array}\right),
  \end{equation*}
where $b_{q}=\left(\left(b_q\right)_1,...,\left(b_{q}\right)_{n-1}\right)$ with
\begin{align*}
\left(b_{q}\right)_k=&\sum_{j=1}^{n-1} \text{  \scriptsize $\frac{(-1)^{n+k}\Delta_{1,...,n-1}^{1,..., n-2,n}\left(A-\lambda_j^{\left(n-1\right)}\cdot \mathrm{Id}_n\right) \Delta^{1,...,\hat{k},...,n-1}_{1,...,n-2}   \left(A-\lambda^{\left(n-1\right)}_j\cdot \mathrm{Id}_n\right)}{{\prod_{l=1,l\ne j}^{n-1}\left(\lambda^{\left(n-1\right)}_l-\lambda^{\left(n-1\right)}_j\right)\prod_{l=1}^{n-2}\left(\lambda^{\left(n-2\right)}_l-\lambda^{\left(n-1\right)}_j\right)} }$  } \\
    &  \sum_{s=1}^n
    \frac{  \prod_{l=1}^{n-1} ( q^{1+\mu_{s}-\nu_{l}}  ;q)_\infty}{\prod_{l=1,l\neq   s}^{n}   (q^{1+\mu_{s}-\mu_{l}}   ;q)_\infty}   \frac{\prod_{l=1,l\neq s}^{n}\theta_q((1-q)q^{1-\mu_l}/\lambda)}{\prod_{l=1}^{n-1}\theta_q((1-q)q^{1-\nu_l}/\lambda)}   
 \frac{\theta_q\left(  q^{ \mu_s-\nu_n  }  \lambda z\right)}{\theta_q(\lambda  z )}     \\
  &\times   \frac{ \prod_{ l=1,l\neq  s  }^n  ( q^{1+\mu_l-{\nu}_j} ;q )_{\infty} \prod_{l=1,l\neq  j}^{n-1}  ( q^{ \nu_l-\mu_s  }  ;q)_{\infty} }{ \prod_{l=1,l\neq  j}^{n-1} ( q^{ 1+\nu_l-\nu_j  } ;q)_{\infty}\prod_{l=1,l\neq  s}^n ( q^{\mu_{ l }-\mu_s} ;q)_{\infty} }        \frac{ \theta_q\left( { q^{  \nu_j-\mu_s-1  } \mu z }       \right)  }{ \theta_q\left(  \mu z   \right)   }  \frac{ \theta_q \left(  {{\mu q^{ \mu_s       }}}/{(1-q)} \right)  }{ \theta_q\left(   {{\mu q^{    \nu_j  -1  }}}/{(1-q)}  \right)  }    \frac{z^{\nu_j}}{z^{\nu_n+1}} .
\end{align*}
\end{thm}
\begin{proof}
The  expression of $\operatorname{diag}\left({P}_{n-1},1\right)S_{q}\left(z,\lambda,\mu;E_n, A_{n-1}\right)\operatorname{diag}\left({P}_{n-1}^{-1},1\right)$ tells us that
\begin{align}
    \left(b_{q}\right)_k=&-\sum_{j=1}^{n-1}   {b_j^{(n-1)}\left({P}^{-1}_{n-1}\right)_{jk}}  \sum_{s=1}^n
    \frac{  \prod_{l=1}^{n-1} ( q^{1+\mu_{s}-\nu_{l}}  ;q)_\infty}{\prod_{l=1,l\neq   s}^{n}   (q^{1+\mu_{s}-\mu_{l}}   ;q)_\infty}   \frac{\prod_{l=1,l\neq s}^{n}\theta_q((1-q)q^{1-\mu_l}/\lambda)}{\prod_{l=1}^{n-1}\theta_q((1-q)q^{1-\nu_l}/\lambda)}   
 \frac{\theta_q\left(  q^{ \mu_s-\nu_n  }  \lambda z\right)}{\theta_q(\lambda  z )}   \nonumber   \\
  &\times   \frac{ \prod_{ l=1,l\neq  s  }^n  ( q^{1+\mu_l-{\nu}_j} ;q )_{\infty} \prod_{l=1,l\neq  j}^{n-1}  ( q^{ \nu_l-\mu_s  }  ;q)_{\infty} }{ \prod_{l=1,l\neq  j}^{n-1} ( q^{ 1+\nu_l-\nu_j  } ;q)_{\infty}\prod_{l=1,l\neq  s}^n ( q^{\mu_{ l }-\mu_s} ;q)_{\infty} }        \frac{ \theta_q\left( { q^{  \nu_j-\mu_s-1  } \mu z }  \right)  }{ \theta_q\left(  \mu z   \right)   }  \frac{ \theta_q \left(  {{\mu q^{ \mu_s       }}}/{(1-q)} \right)  }{ \theta_q\left(   {{\mu q^{    \nu_j  -1  }}}/{(1-q)}  \right)  }    \frac{z^{\nu_j}}{z^{\nu_n+1}}
       . \label{b+k}
\end{align}
Using Proposition \ref{An-1Pn-1}, the identity \eqref{b+k} can be rewritten as the form in this theorem. 
\end{proof}
Like Remark \ref{rmkres:nvaphin}, we have
\begin{rmk}
    We can use the connection formula of ${}_n\varphi_n(a_1,...,a_n;b_1,...,b_n;z)$ to derive the $q$-Stokes matrix of the system
    \begin{equation*}
    D_qF_q(z)=\left(E_{n n}\sigma_q+\frac{A}{z}\right)F_q(z).
\end{equation*}
\end{rmk}

\section{Behavior of the Equation \eqref{intro:conflu hyper eq} and the System \eqref{qdes} as $q\rightarrow 1$}\label{qlimit}
In this section, we study the behavior of the connection formula for the equation \eqref{intro:conflu hyper eq} (see Section \ref{subsec:eqlim}),   as well as the connection matrix and the $q$-Stokes matrix of the system \eqref{qdes} (see Section \ref{sec:syslim}) as $q\rightarrow 1$.   See  \cite{sauloy2000systemes}  for   more  discussions in   the Fuchsian case  and  \cite{dreyfus2015confluence} in the irregular case.

\subsection{Behavior of the Connection Formula  as $q\rightarrow 1$}\label{subsec:eqlim}
We focus on the  behavior of  the connection formula   as $q\rightarrow 1$ of the equation
 \begin{equation}\label{sec:qconf eq}
    \prod_{l=1}^{n-1}\left(zD_q+\alpha_{l} \right)y(z)=D_q\prod_{l=1}^{n-1}\left(zD_q+\beta_{l}-1 \right)y(z),
\end{equation}
which takes the general  form of the equation \eqref{Fnj eq}.  See \cite{morita2013connection} for more discussion of  $n=2$ case.  If
    the   complex numbers    $\alpha_l\neq  \frac{1}{q-1}  $, $\beta_l\neq  \frac{q}{q-1} $,  for $1\leq l \leq  n-1$,  
we take the  parameters   $a_1,...,a_{n-1},b_1,...,b_{n-1}$           as  
\begin{align}
    &{a}_l=\frac{1}{1+(1-q)\alpha_l}, \  \text{for all}  \    1\leq l \leq   n-1,   \label{al change}  \\
     &{b}_l=\frac{q}{1+(1-q)(\beta_l-1)}, \  \text{for all}  \    1\leq l \leq   n-1.  \label{bl change}
\end{align}

Similar to the   case  of  \eqref{Fnj eq},   we have 
\begin{prop}
Given $a_l,b_l$, for $1\leq l \leq n-1$,   in \eqref{al change}, \eqref{bl change},    the solutions and the    connection formula of  \eqref{sec:qconf eq} are:
   \begin{itemize}
    \item   The  equation \eqref{sec:qconf eq}  admits the convergent solution:
    \begin{equation}
        \tilde{f}^{(0)}_q(z):={}_n\varphi_{n-1}\left(  \begin{array}{c}  a_1,...,a_{n-1},0\\
 b_1,...,b_{n-1}  \end{array};q,\frac{(1-q)z\prod_{l=1}^{n-1}{b_l}}{\prod_{l=1}^{n-1}{q}a_l}   \right).\label{tildef0qz}  
    \end{equation}
 \item  The equation \eqref{sec:qconf eq}  admits  the   fundamental system of meromorphic    solutions  around infinity:
 \begin{align}\label{tildefinfeq}
     \tilde{f}^{(\infty)}_q(z,\lambda)&=  \left(  \tilde{f}^{(\infty)}_q(z,\lambda)_1,...,\tilde{f}^{(\infty)}_q(z,\lambda)_{n-1},\tilde{f}^{(\infty)}_q(z)_n   \right) \nonumber  \\
    &:= \left(
 \frac{\tilde{h}^{(\infty)}_q(z,\lambda)_{1}}{z^{ \log_qa_1  }}  ,
...  ,
   \frac{\tilde{h}^{(\infty)}_q(z,\lambda)_{n-1}}{z^{ \log_qa_{n-1}  }},
\tilde{h}^{(\infty)}_q(z)_{n}  \frac{z^{\sum_{l=1}^{n-1}\log_qa_l}}{  z^{ \sum_{l=1}^{n-1}\log_qb_l   }  } e_q\left( \frac{z\prod_{l=1}^{n-1}{b_l}}{\prod_{l=1}^{n-1}{q}a_l}  \right)
\right)       ,
 \end{align}  
 with
 \begin{equation*}
     \tilde{h}^{(\infty)}_q(z,\lambda)_{i}={}_nf_{n-2}\left(\begin{array}{c}
			a_i, a_i q/b_1,...,a_i q/b_{n-1}\\
			a_i q/a_1,...,\widehat{a_i q/a_{{i}} },...,a_i q/a_{n-1}
		\end{array};  q,\frac{q^{n-1}\lambda}{(1-q)a_i}, \frac{q^{n-1}}{a_i(1-q)z}\right), \   \text{for} \  i=1,...,n-1,   
 \end{equation*}
and $\tilde{h}^{(\infty)}_q(z)_{n}$ is  a  convergent  power
series  whose coefficients can be uniquely determined recursively.
 \item  The solutions  $\tilde{f}^{(0)}_q(z)$  and   $\tilde{f}_q^{(\infty)}(z,\lambda)$   of \eqref{sec:qconf eq}  are related by the following connection formula
 \begin{equation}
     \tilde{f}^{(0)}_q(z)= \sum_{i=1}^{n-1} \tilde{f}^{(\infty)}_q(z,\lambda)_{i}  \tilde{C}_i(z)+\tilde{f}^{(\infty)}_q(z)_{n}  \tilde{C}_n(z,\lambda),\label{eq: qconnefomu}
 \end{equation}
 with for  $1\leq  i\leq n-1$,
 \begin{align*}
    & \tilde{C}_i(z):=\frac{(a_1,...,\widehat{a_{i} },...,a_{n-1},b_1/a_i,...,b_{n-1}/a_i;q)_\infty}{(b_1,...,b_{n-1},a_1/a_i,...,\widehat{a_{i}/a_i},...,a_{n-1}/a_i;q)_\infty}\frac{\theta_q\left(-a_i {(1-q)z\prod_{l=1}^{n-1}{b_l}}\big/{\prod_{l=1}^{n-1}{q}a_l}\right)}{ \theta_q\left(-{(1-q)z\prod_{l=1}^{n-1}{b_l}}\big/{\prod_{l=1}^{n-1}{q}a_l}\right) }z^{ \log_q a_i  },\\
    &\tilde{C}_n(z,\lambda)
  % :=  \frac{(a_1,...,a_{n-1};q)_\infty}{(b_1,...,b_{n-1};q)_\infty}  \frac{\prod_{l=1}^{n-1}\theta_q\left( \frac{b_l (1-q)}{ q^{n-1}  \lambda }   \right)}{\prod_{l=1}^{n-1}\theta_q\left( \frac{  a_l(1-q) }{ q^{n-1} \lambda }     \right)} \frac{\theta_q\left( \frac{q^{n-1} \lambda  }{1-q }  \frac{(1-q)z\prod_{l=1}^{n-1}{b_l}}{\prod_{l=1}^{n-1}{q}a_l}     \right)}{\theta_q\left( \frac{q^{n-1} \lambda  }{1-q }  \frac{(1-q)z\prod_{l=1}^{n-1}{b_l}}{\prod_{l=1}^{n-1}{q}a_l}   \frac{\prod_{l=1}^{n-1}a_l}{  \prod_{l=1}^{n-1}b_l }    \right)}  
		% \frac{z^{\sum_{l=1}^{n-1}\log_qb_l  }}{ z^{ \sum_{l=1}^{n-1}\log_qa_l }  } \\
  :=\frac{(a_1,...,a_{n-1};q)_\infty}{(b_1,...,b_{n-1};q)_\infty}  \frac{\prod_{l=1}^{n-1}\theta_q\left( {b_l (1-q)}\big/{ (q^{n-1}  \lambda) }   \right)}{\prod_{l=1}^{n-1}\theta_q\left( {  a_l(1-q) }\big/{ (q^{n-1} \lambda )}     \right)} \frac{\theta_q\left(   {  { \lambda  }  z\prod_{l=1}^{n-1}{b_l}}\big/{\prod_{l=1}^{n-1}a_l}     \right)}{\theta_q\left( \lambda z   \right)}  
		\frac{z^{\sum_{l=1}^{n-1}\log_qb_l  }}{ z^{ \sum_{l=1}^{n-1}\log_qa_l }  }. 
 \end{align*}
   \end{itemize}
   \end{prop}
% Set $p=q^{-1}$.   In  \eqref{conflu hyper eq},  if    we take the  parameters      as  
% \begin{align*}
%     &{a}_l=q^{ [ q\alpha_l ]^p }, \  \text{for all}  \    1\leq l \leq   n-1,     \\
%      &{b}_l=q^{  [ q(\beta_l-1) ]^p +1 }, \  \text{for all}  \    1\leq l \leq   n-1,
% \end{align*}
% and replace $z$  by  $p^{ \sum_{l=1}^{n-1} [q  a_l]^p-\sum_{l=1}^{n-1}  [ q(\beta_l-1)  ]^p   }(1-q)z$,   we  obtain the equation 
 The  equation     \eqref{sec:qconf eq} converges  in the form  to  the  confluent hypergeometric differential equation  
\begin{equation}\label{c h eq}
    \prod_{l=1}^{n-1}\left(z\frac{\mathrm{d}}{\mathrm{d}z}+\alpha_{l} \right)y(z)=\frac{\mathrm{d}}{\mathrm{d}z}\prod_{l=1}^{n-1}\left(z\frac{\mathrm{d}}{\mathrm{d}z}+\beta_{l}-1 \right)y(z).
\end{equation}
It is well known that  there is a fundamental system of formal solutions around infinity  of  \eqref{c h eq}  (see  e.g., \cite{balser2008formal, turrittin1955convergent}):
\begin{equation*}
        \hat{f}(z):=\left(
 \hat{h}_{1}(z) z^{-\alpha_1},
...  ,
 \hat{h}_{n-1}(z) z^{-\alpha_{n-1}  } ,
\hat{h}_{n}(z)  \mathrm{e}^z z^{  \sum_{l=1}^{n-1}\left(\alpha_l-\beta_l\right)  }
\right),
   \end{equation*}
   with
\begin{align*}
  & \hat{h}_i(z):= {}_{n}F_{n-2}\left(\alpha_i,1+\alpha_i-\boldsymbol{\beta};1+\alpha_i-\boldsymbol{\alpha_{\hat{i}}};-z^{-1} \right), \  \text{for} \  i=1,...,n-1,\\
  &  \hat{h}_n(z):= 1+ \sum_{m=1}^{\infty}\kappa_m z^{-m},
\end{align*}
% \begin{align*}
% z^{1-\alpha_i}:&  \prod_{l=1}^{n-1}\left(-\alpha_i+\alpha_{l} \right)=0, \\
% & \prod_{l=1}^{n-1}\left(-m-1-\alpha_i+\alpha_{l} \right)   a_{m+1} z^{-m-\alpha_i}=\left(-m-\alpha_i \right)\prod_{l=1}^{n-1}\left(-m-\alpha_i+\beta_{l}-1 \right)a_{m} z^{-m-\alpha_i},\\
% &\frac{a_{m+1}}{a_m}=\frac{\left(-m-\alpha_i \right)\prod_{l=1}^{n-1}\left(-m-\alpha_i+\beta_{l}-1 \right)}{\prod_{l=1}^{n-1}\left(-m-1-\alpha_i+\alpha_{l} \right)}=-\frac{\left(m+\alpha_i \right)\prod_{l=1}^{n-1}\left(m+\alpha_i-\beta_{l}+1 \right)}{\prod_{l=1}^{n-1}\left(m+1+\alpha_i-\alpha_{l} \right)},\\
% &a_m=(-1)^m\frac{(\alpha_i)_m\prod_{l=1}^{n-1}(\alpha_i-\beta_l+1)_m}{m!\prod_{l=1,l\neq i}^{n-1}(1+\alpha_i-\alpha_l)_m}\\
% &{}_n F_{n-2}(\alpha_i,\alpha_i-\boldsymbol{\beta}+1;1+\alpha_i-\boldsymbol{\alpha_{\hat{i}}};-z^{-1})z^{-\alpha_i}
% \end{align*}
where   $\kappa_m$  can be uniquely determined recursively.  Here we use  index notations 
\begin{align*}
    &1+\alpha_i-\boldsymbol{\beta}:=( 1+\alpha_i-\beta_1,..., 1+\alpha_i-\beta_{n-1} ),\\
    &1+\alpha_i-\boldsymbol{\alpha_{\hat{i}}}:=\Big( 1+\alpha_i-\alpha_1,...,\widehat{1+\alpha_i-\alpha_i},...,1+\alpha_i-\alpha_{n-1} \Big),
\end{align*}
and  $\widehat{1+\alpha_i-\alpha_i}$ means that this term is skipped.

The standard theory of Borel resummation (see e.g., \cite{balser2008formal, loday2016divergent, Wasow})  states that   
\begin{prop}
     On $\mathrm{Sect}_{+}:=S\left(-\pi/2,3\pi/2 \right)$ and $\mathrm{Sect}_-:=S\left(-3\pi/2,\pi/2 \right)$ (recall the notation  $S\left(a,b\right)$ in Section \ref{sec:q-borel sum}),     there exist a  fundamental system of  meromorphic   solutions around infinity of  \eqref{c h eq}
    \begin{equation}\label{eqfinfpm}
    {f}^{(\infty)}_{\pm}(z):=\left(
 {h}^{(\infty)}_{\pm,1} z^{-\alpha_1},
...  ,
 {h}^{(\infty)}_{\pm,n-1} z^{-\alpha_{n-1}  } ,
{h}^{(\infty)}_{\pm,n}  \mathrm{e}^z z^{  \sum_{l=1}^{n-1}\left(\alpha_l-\beta_l\right)  }
\right),
\end{equation}
    such that the  asymptotic expansions hold
\begin{align*}
    &{h}^{(\infty)}_{\pm,i} \sim {}_{n}F_{n-2}\left(\alpha_i,1+\alpha_i-\boldsymbol{\beta};1+\alpha_i-\boldsymbol{\alpha}_{\hat{i}};-z^{-1} \right),\ \text{as} \ z \rightarrow \infty \ \text{on}\ \mathrm{Sect}_{\pm}, \  \text{for} \ i=1,...,n-1,  \\
    % f^{-}_i z^{\alpha_i}&\sim {}_{n}F_{n-2}\left(\alpha_i,1+\alpha_i-\boldsymbol{\beta};1+\alpha_i-\boldsymbol{\alpha}_{\hat{i}};-z^{-1} \right),\ \text{as} \ z \rightarrow \infty \ \text{on}\ S(-3\pi/2,\pi/2),  \  \text{for} \ i=1,...,n-1,
     &{h}^{(\infty)}_{\pm,n} \sim 1+\sum_{m=1}^{\infty}\kappa_m z^{-m},\ \text{as} \ z \rightarrow \infty \ \text{on}\ \mathrm{Sect}_{\pm}.
\end{align*}
\end{prop}
\begin{defn}\label{def:hi}
For  $1\leq i \leq n-1$,        the  analytic continuation of the    function   ${h}^{(\infty)}_{+,i}(z)$  defined on  $\mathrm{Sect}_{+}$  to   $\widetilde{C^*}$  is denoted as ${h}^{(\infty)}_{i}$.
\end{defn}
Moreover, there is a convergent   solution around the origin of \eqref{c h eq} of the form (see e.g., \cite{OLBC})
\begin{equation*}
    f^{(0)}(z):= {}_{n-1}F_{n-1}\left(\alpha_1,...,\alpha_{n-1};\beta_1,...,\beta_{n-1};z\right).
\end{equation*}
From \cite{ichinobe2001borel, OLBC},  one knows  that  
\begin{prop}[{\cite[Theorem 2.1]{ichinobe2001borel}}]\label{prop:hpm}
 For $1\leq i \leq n-1$,   
% \begin{equation*}
%     h_i(z)\sim  {}_n F_{n-2}(\alpha_i,1+\alpha_i-\boldsymbol{\beta};1+\alpha_i-\boldsymbol{\alpha}_{\hat{i}};-z^{-1}   ),  \    \text{as}  \  z\rightarrow \infty   \  \text{on the sector}\  S(-3\pi/2,3\pi/2  ). 
% \end{equation*}
     ${h}^{(\infty)}_{\pm,i}$  is given by  (here for convenience, we denote $\beta_n=1$) 
\begin{align*}
{h}^{(\infty)}_{\pm,i}
% =&\sum_{k=1}^{n}   \frac{\prod_{l=1,l\neq i}^{n-1}\Gamma(1+\alpha_i-\alpha_l)  \prod_{l=1  ,l\neq  k }^{n}\Gamma( 1+\alpha_i-\beta_l-( 1+\alpha_i-\beta_k ) )}{\prod_{l=1,l\neq  k}^{n}\Gamma( 1+\alpha_i-\beta_l  )  \prod_{l=1,l\neq i}^{n-1}\Gamma(1+\alpha_i-\alpha_l-( 1+\alpha_i-\beta_k ) )}    z^{1+\alpha_i-\beta_k} \\
% &\cdot {}_{n-1}F_{n-1} \left(\begin{array}{c}
% 1+\alpha_i-\beta_k,  1+( 1+\alpha_i-\beta_k )-( 1+\alpha_i-\alpha_1 ),...,  \widehat{1+( 1+\alpha_i-\beta_k )-( 1+\alpha_i-\alpha_i ) },...,  1+( 1+\alpha_i-\beta_k  )- ( 1+\alpha_i-\alpha_{n-1}  )  \\
% 1+(1+\alpha_i-\beta_k   )-( 1+\alpha_i-\beta_1  ),...,\widehat{{1+( 1+\alpha_i-\beta_k  )-( 1+\alpha_i-\beta_k  )}},..., {1+(1+\alpha_i-\beta_k   )-( 1+\alpha_i-\beta_n  )}
% \end{array} ;{z}\right)\\
% =&\sum_{k=1}^{n}   \frac{\prod_{l=1,l\neq i}^{n-1}\Gamma(1+\alpha_i-\alpha_l)  \prod_{l=1  ,l\neq  k }^{n}\Gamma(\beta_k -\beta_l )}{\prod_{l=1,l\neq  k}^{n}\Gamma( 1+\alpha_i-\beta_l  )  \prod_{l=1,l\neq i}^{n-1}\Gamma(\beta_k -\alpha_l )}    z^{1+\alpha_i-\beta_k} \\
% &\cdot {}_{n-1}F_{n-1} \left(\begin{array}{c}
% 1+\alpha_i-\beta_k,  1+\alpha_1 -\beta_k ,...,  \widehat{1+\alpha_i -\beta_k  },...,  1+\alpha_{n-1} -\beta_k    \\
% 1+\beta_1-\beta_k   ,...,\widehat{{1+\beta_k -\beta_k }},..., {1+\beta_n-\beta_k   }
% \end{array} ;{z}\right)\\
=&\sum_{k=1}^{n}   \frac{\prod_{l=1,l\neq i}^{n-1}\Gamma(1+\alpha_i-\alpha_l)  \prod_{l=1  ,l\neq  k }^{n}\Gamma(\beta_k -\beta_l )}{\prod_{l=1,l\neq  k}^{n}\Gamma( 1+\alpha_i-\beta_l  )  \prod_{l=1,l\neq i}^{n-1}\Gamma(\beta_k -\alpha_l )}    z^{1+\alpha_i-\beta_k} \\
&\times {}_{n-1}F_{n-1} \left(\begin{array}{c}
  1+\alpha_1 -\beta_k ,...,  1+\alpha_{n-1} -\beta_k    \\
1+\beta_1-\beta_k   ,...,\widehat{{1+\beta_k -\beta_k }},..., {1+\beta_n-\beta_k   }
 \end{array};{z}\right),
\end{align*}
as  $z\in \mathrm{Sect}_{\pm}$.
As a result,  for $1\leq i \leq  n-1$,  we  have
\begin{equation*}
        {h}^{(\infty)}_{+,i}={h}^{(\infty)}_{-,i}, \  \text{as} \  z\in S(-\pi/2,\pi/2  ),
    \end{equation*}
    and  ${h}^{(\infty)}_{i}$ is also the analytic continuation of   ${h}^{(\infty)}_{-,i}$.
\end{prop}
\begin{prop}[see e.g., {\cite[Chapter 16]{OLBC}}]
    For  $z\in \mathrm{Sect}_{\pm}$,   the solutions  ${f}^{(0)}(z)$  and  ${f}^{(\infty)}_{\pm}(z)$    of \eqref{c h eq}  are related by the following connection formula:
 \begin{equation}\label{olbcconne}
f^{(0)}(z)=\sum_{i=1}^{n-1} \frac{\prod_{l=1}^{n-1}\Gamma\left(\beta_l\right)}{\prod_{l=1,l\neq i}^{n-1}\Gamma\left(\alpha_l\right)}  \frac{\prod_{l=1,l\neq i}^{n-1}\Gamma\left(\alpha_l-\alpha_i\right)}{\prod_{l=1}^{n-1}\Gamma\left(\beta_l-\alpha_i\right)}{ \mathrm{e}^{\pm\pi \mathrm{i} \alpha_i}  } {h}^{(\infty)}_{\pm,i} z^{-\alpha_i}+\frac{\prod_{l=1}^{n}\Gamma\left(\beta_l\right)}{\prod_{l=1}^{n}\Gamma\left(\alpha_l\right)}     {h}^{(\infty)}_{\pm,n}  \mathrm{e}^z z^{  \sum_{l=1}^{n-1}\left(\alpha_l-\beta_l\right)  }.
\end{equation}
\end{prop}

The  connection formula \eqref{eq: qconnefomu}    recovers the    connection formula \eqref{olbcconne} in the following sense.  That  is,
\begin{prop}\label{prop:lim eq}
 As  $q\rightarrow 1$,  it follows that  
 \begin{itemize}
\item[(1)]       The  convergent  solution  $\tilde{f}^{(0)}_q(z)$ (see \eqref{tildef0qz})  of  \eqref{sec:qconf eq}  converges to         the convergent  solution $f^{(0)}(z)$  of  \eqref{c h eq}:
\begin{equation}
 \lim_{q\rightarrow 1} \tilde{f}^{(0)}_q(z)= f^{(0)}(z).\label{lim:zero}
\end{equation} 
\item[(2)]  The coefficients  $\tilde{C}_1(z),...,\tilde{C}_{n-1}(z),\tilde{C}_n(z,\lambda)$ have limits:
\begin{align}
     &\lim_{q\rightarrow 1}\tilde{C}_i(z)=\frac{\prod_{l=1}^{n-1}  \Gamma(\beta_l) \prod_{l=1,l\neq  i}^{n-1} \Gamma( \alpha_l -\alpha_i )  }{\prod_{l=1,l\neq  i}^{n-1}  \Gamma( \alpha_l )  \prod_{l=1}^{n-1} \Gamma( \beta_l-\alpha_i  )}  \cdot       \mathrm{e}^{ \pm \mathrm{i}  \pi  \alpha_i },   \  \text{for}  \   1\leq  i\leq  n-1,  \label{tildeCi} \\
     &\lim_{q\rightarrow 1}\tilde{C}_n(z,\lambda)     =\frac{ \prod_{ l=1 }^{n-1}\Gamma(   \beta_l   )  }{ \prod_{ l=1 }^{n-1}\Gamma(   \alpha_l   )  }.  \label{tildeCn}
\end{align}
In \eqref{tildeCi},  \eqref{tildeCn} we take  $ \mathrm{arg}(\lambda)\in(-\pi,\pi)$  and  $\mathrm{arg}(z)\in  (-\mathrm{arg}(\lambda)-\pi,-\mathrm{arg}(\lambda)+\pi)$.  The  positive    or   negative  sign in \eqref{tildeCi}  is   chosen based on whether    $\mathrm{arg}(z)\in (0,-\mathrm{arg}(\lambda)+\pi  )$  or  $\mathrm{arg}(z)\in (-\mathrm{arg}(\lambda)-\pi, 0  )$.  
\item[(3)]     The fundamental system of meromorphic solutions  $ \tilde{f}_q^{(\infty)}(z,\lambda)$  (see  \eqref{tildefinfeq})   of  \eqref{sec:qconf eq}  converges to the fundamental system of meromorphic solutions  ${f}^{(\infty)}_{ \pm }(z)$   (see \eqref{eqfinfpm})  of  \eqref{c h eq}:
\begin{align}
 &\lim_{q\rightarrow 1}   \tilde{f}_q^{(\infty)}(z,\lambda)_i=h^{(\infty)}_i(z)  z^{-\alpha_i },   \  \text{as} \  \mathrm{arg}(z)\in  (-\mathrm{arg}(\lambda)-\pi,-\mathrm{arg}(\lambda)+\pi  ),   \  \text{for}  \  i=1,...,n-1,  \label{tildeflim} \\
 &\lim_{q\rightarrow 1} \tilde{f}^{(\infty)}_q(z)_n =\left\{
          \begin{array}{lr}
          {h}^{(\infty)}_{+,n}  \mathrm{e}^z z^{  \sum_{l=1}^{n-1}\left(\alpha_l-\beta_l\right)  },   & \text{if} \ \  \mathrm{arg}(z)\in (0,-\mathrm{arg}(\lambda)+\pi  ), \\
             {h}^{(\infty)}_{-,n}  \mathrm{e}^z z^{  \sum_{l=1}^{n-1}\left(\alpha_l-\beta_l\right)  }    , & \text{if} \ \ \mathrm{arg}(z)\in (-\mathrm{arg}(\lambda)-\pi, 0  ).
             \end{array}
\right.\label{tildeflim2eq}
\end{align}
\item[(4)]  If $\mathrm{arg}(z)\in (0,-\mathrm{arg}(\lambda)+\pi  )$,  the connection formula \eqref{eq: qconnefomu}  recovers \eqref{olbcconne} with the positive sign chosen in \eqref{olbcconne}, and if $\mathrm{arg}(z)\in (  -\mathrm{arg}(\lambda)-\pi ,0  )$,  the connection formula \eqref{eq: qconnefomu}  recovers \eqref{olbcconne} with the negative sign chosen in \eqref{olbcconne}.   
 \end{itemize}

\end{prop}
  To prove Proposition \ref{prop:lim eq},  we  need the following lemmas.
\begin{lemma}[see e.g., {\cite[Section 5]{askey1978q}, \cite[Section 1.10]{ gasper2004basic}}]\label{basic lim}
As  $q\rightarrow 1$,  we  have
\begin{align*}
  &  \lim_{q\rightarrow 1}  \log_q(1+(1-q)\alpha)=-\alpha, \ \text{for}  \   \alpha\neq 1/(q-1),\\
  &\lim _{q \rightarrow 1} e_q(z)=\mathrm{e}^z,\\
  &  \lim _{  
    q  \rightarrow 1
	  }  \Gamma_q(z)=\Gamma(z),                  \\
  &  \lim_{q\rightarrow 1}  \frac{\theta_q( q^{\beta} z )}{ \theta_q( q^{\alpha} z )  } =z^{\alpha-\beta},  \  \text{for} \  z\in  \mathbb{C}^*\backslash \mathbb{R}_{<0}.  
\end{align*}
Here  we  take  $-\pi<\mathrm{arg}(z)<\pi$  in the the power function  $z^{\alpha-\beta}$.
\end{lemma}

\begin{lemma}\label{lem:limf}
For $i=1,...,n-1$,    as $q\rightarrow 1$, we have the limit:
    \begin{equation}\label{expre:qlimf}
        \lim_{q\rightarrow 1}  {}_n f_{n-2}\left(\begin{array}{c}
          a_i, a_i q/b_1,...,a_i q/b_{n-1}\\
			a_i q/a_1,..., \widehat{ a_i q/a_{{i}} },...,a_i q/a_{n-1}
        \end{array}  ;\frac{q^{n-1}\lambda}{(1-q)a_i}; q, {\frac{ q^{n-1} }{ {a_i}{(1-q)z  }  }} \right)= {h}^{(\infty)}_{i}(z). 
    \end{equation}
  Here  in  the left side of \eqref{expre:qlimf} (a single-valued function), the variable  $z\in \mathbb{C}^*\backslash  - \lambda \mathbb{R}_{>0}$,   and  in the right side,  we   take the branch   of the multi-valued function    ${h}^{(\infty)}_{i}(z)$ (given in Definition \ref{def:hi})  with   $\mathrm{arg}(z)\in  (-\mathrm{arg}(\lambda)-\pi,-\mathrm{arg}(\lambda)+\pi  )$.  Here   we take $\mathrm{arg}(\lambda)\in (-\pi,\pi)$.
\end{lemma}
\begin{proof}
    From  Lemma \ref{Adachi}  and    Lemma \ref{basic lim},  we have  (here for convenience, we denote $b_n=q,\beta_n=1$) 
   %    \begin{align*}
   %        {}_n f_{n-2} \left(\begin{array}{c}
   %           q^{\gamma_1},...,q^{\gamma_n}    \\
   %    q^{\delta_1},...,q^{\delta_{n-2}} 
   %        \end{array}  ;\frac{\lambda}{1-q},q, \frac{1}{(1-q)z} \right)=&\sum_{k=1}^n \frac{\left(q^{\gamma_1},...,\widehat{q^{\gamma_k}},..., q^{\gamma_n},  q^{\delta_1-\gamma_k} , ..., q^{\delta_{n-2}-\gamma_k} ; q\right)_{\infty}}{\left(q^{\delta_1},... , q^{\delta_{n-2}}, q^{\gamma_1-\gamma_k} ,..., \widehat{q^{\gamma_k-\gamma_k}}   , ...,  q^{\gamma_n-\gamma_k} ; q\right)_{\infty}} \\
   %        &\frac{\theta_{q}\left(q q^{\gamma_k} \frac{ 1 }{(1-q)z} / \frac{\lambda}{1-q}\right)}{\theta_{q}(q \frac{ 1 }{(1-q)z} /\frac{ \lambda}{1-q})} \frac{\theta_q\left( q^{\gamma_k}\frac{ \lambda} {1-q} \right)}{\theta_q\left(\frac{ \lambda}   {1-q}\right)} \\
   %    &    \times{ }_n \varphi_{n-1}\left(\begin{array}{c}
   %  q^{\gamma_k}, q^{1+\gamma_k-\delta_1} , ...,   q^{1+\gamma_k-\delta_{n-2}}, {0} \\
   % q^{1+\gamma_k-\gamma_1} ,..., \widehat{q^{1+\gamma_k-\gamma_k} }  , ...,  q^{1+\gamma_k-\gamma_n}
   %  \end{array} ;q, \frac{q (1-q) z\prod_{l=1}^{n-2} q^{\delta_l}  }{ \prod_{l=1}^n q^{\gamma_l}   }\right).
   %    \end{align*}
    \begin{align*}
     &   \lim_{q\rightarrow 1}   {}_n f_{n-2} \left(\begin{array}{c}
           a_i, a_i q/b_1,...,a_i q/b_{n-1}\\
			a_i q/a_1,..., \widehat{ a_i q/a_{{i}} },...,a_i q/a_{n-1} 
        \end{array}  ;  \frac{q^{n-1}\lambda}{(1-q)a_i}; q, {\frac{ q^{n-1} }{ {a_i}{(1-q)z  }  }}  \right)\\
        =&\lim_{q\rightarrow 1} \sum_{k=1}^n \frac{ \prod_{l=1,l\neq i}^{n-1}  \Gamma_q( \log_q( a_iq/a_l ) )  \prod_{l=1,l\neq  k}^n  \Gamma_q(\log_q( a_iq/b_l ) -\log_q( a_iq/b_k ) ) }{  \prod_{l=1,l\neq k}^n  \Gamma_q(    \log_q(a_iq/b_l  ) ) \prod_{l=1,l\neq i}^{n-1}\Gamma_q( \log_q( a_iq/a_l )-\log_q( a_i q/b_k ) )  } 
        \frac{\left( {z\lambda}  \right)^{\log_q( a_iq/b_k  )}}{ (1-q)^{ \log_q( a_iq/b_k  ) }  }      \\
        &\times \left(\frac{q^{n-1}\lambda}{(1-q)a_i}\right)^{-\log_q a_i  }     { }_n\varphi_{n-1}\left(\begin{array}{c}
   a_iq/b_k,a_1q/b_k,...,\widehat{a_iq/b_k},...,a_{n-1}q/b_k , {0} \\
 b_1q/b_k,...,\widehat{b_kq/b_k},...,b_nq/b_k
  \end{array} ;q, \frac{(1-q)z\prod_{l=1}^n b_l}{q^n\prod_{l=1}^{n-1}a_l}\right)   \\
        =&  \sum_{k=1}^n   \frac{ \prod_{l=1,l\neq i}^{n-1}  \Gamma( 1+\alpha_i-\alpha_l )  \prod_{l=1,l\neq  k}^n  \Gamma(   \beta_k-\beta_l   ) }{  \prod_{l=1,l\neq k}^n  \Gamma( 1+\alpha_i-\beta_l ) \prod_{l=1,l\neq  i}^{n-1}\Gamma( \beta_k -\alpha_l     )  }   \\
    &\times   z ^{1+\alpha_i-\beta_k}    {}_{n-1}F_{n-1}\left( \begin{array}{c}
        1+\alpha_1-\beta_k,...,1+\alpha_{n-1}-\beta_k  \\
         1+\beta_1-\beta_k,...,\widehat{1+\beta_k-\beta_k},...,1+\beta_n-\beta_k
        \end{array}       ;{z} \right)  , 
    \end{align*}
    where  we  take    $\mathrm{arg}{(\lambda)}\in  (-\pi,\pi)$  and   $\mathrm{arg}(z)\in  (-\mathrm{arg}(\lambda)-\pi,-\mathrm{arg}(\lambda)+\pi  )$.     Comparing the above formula and  Proposition    \ref{prop:hpm},  we obtain this proposition.
\end{proof}

\begin{proof}[Proof of Proposition \ref{prop:lim eq}]
The proof contains the following parts    ($a_1,...,a_{n-1},b_1,...,b_{n-1}$    given in  \eqref{al change}, \eqref{bl change}):
\begin{itemize}
    \item The formula \eqref{lim:zero}  is  derived from direct computations.
    \item      The  formula \eqref{tildeCi}    is  derived from    Lemma \ref{basic lim}  and   the following two formulas:
    \begin{equation*}
      \frac{\prod_{l=1,l\neq  i}^{n-1}(a_l;q)_{\infty}\prod_{l=1}^{n-1}(b_l/a_i;q)_\infty}{\prod_{l=1}^{n-1}(b_l;q)_{\infty}\prod_{l=1,l\neq  i}^{n-1}(a_l/a_i;q)_\infty} = (1-q)^{\log_q  a_i} \frac{\prod_{l=1}^{n-1}  \Gamma_q(\log_qb_l) \prod_{l=1,l\neq  i}^{n-1} \Gamma_q( \log_qa_l -\log_q a_i )  }{\prod_{l=1,l\neq  i}^{n-1}  \Gamma_q( \log_qa_l )  \prod_{l=1}^{n-1} \Gamma_q( \log_q b_l-\log_q a_i  )} , 
      \end{equation*}
    and
      \begin{align}
        % &  \frac{(a_1,...,a_{\hat{i}},...,a_{n-1},b_1/a_i,...,b_{n-1}/a_i;q)_\infty}{(b_1,...,b_{n-1},a_1/a_i,...,a_{\hat{i}}/a_i,...,a_{n-1}/a_i;q)_\infty} &&= (1-q)^{\log_q  a_i} \frac{\prod_{l=1}^{n-1}  \Gamma_q(\log_qb_l) \prod_{l=1,l\neq  i}^{n-1} \Gamma_q( \log_qa_l -\log_q a_i )  }{\prod_{l=1,l\neq  i}^{n-1}  \Gamma_q( \log_qa_l )  \prod_{l=1}^{n-1} \Gamma_q( \log_q b_l-\log_q a_i  )} , \nonumber  \\
        & \lim_{q\rightarrow 1} (1-q)^{ \log_q a_i }   \frac{\theta_q\left( -  {   (1-q)a_i  z\prod_{l=1}^{n-1}{b_l}}\big/{\prod_{l=1}^{n-1}{q}a_l} \right)}{  \theta_q\left( -{(1-q)z\prod_{l=1}^{n-1}{b_l}}\big/{\prod_{l=1}^{n-1}{q}a_l} \right)    }  z^{ \log_q a_i  }   \nonumber \\
         =& \lim_{q\rightarrow 1}  \frac{(1-q)^{\log_q a_i }  z^{\log_q a_i  } }  {\left(-  {   (1-q)z\prod_{l=1}^{n-1}{b_l}}\big/{\prod_{l=1}^{n-1}{q}a_l} \right)^{ \log_q a_i }}  \nonumber  \\
            =&  \mathrm{e}^{ \pm \mathrm{i}  \pi  \alpha_i }  \label{pf:lim1}.
    \end{align}
    The  positive    or   negative  sign in \eqref{pf:lim1}  is   chosen based on whether    $\mathrm{arg}(z)\in (0, -\mathrm{arg}(\lambda)+\pi  )$  or  $\mathrm{arg}(z)\in ( -\mathrm{arg}(\lambda)-\pi, 0  )$.     And   the  formula \eqref{tildeCn}    is  derived from Lemma \ref{basic lim}  and the following two formulas:
    \begin{align}
     &   \frac{(a_1,...,a_{n-1};q)_\infty}{(b_1,...,b_{n-1};q)_\infty}  = \frac{ (1-q)^{ \sum_{l=1}^{n-1}\log_q b_l } }{  (1-q)^{ \sum_{l=1}^{n-1}\log_q a_l  } }  \frac{ \prod_{ l=1 }^{n-1}\Gamma_q( \log_qb_l   )  }{ \prod_{ l=1 }^{n-1}\Gamma_q( \log_qa_l   )  } ,\nonumber  \\
     % & \lim_{q\rightarrow 1} \frac{ (1-q)^{ \sum_{l=1}^{n-1}\log_q b_l } }{  (1-q)^{ \sum_{l=1}^{n-1}\log_q a_l  } }    \frac{\prod_{l=1}^{n-1}\theta_q\left( \frac{b_l (1-q)  }{ q^{n-1}  \lambda }   \right)}{\prod_{l=1}^{n-1}\theta_q\left( \frac{  a_l(1-q) }{ q^{n-1} \lambda }     \right)}  \frac{\theta_q\left( \frac{q^{n-1} \lambda  }{1-q }  \frac{(1-q)z\prod_{l=1}^{n-1}{b_l}}{\prod_{l=1}^{n-1}{q}a_l}     \right)}{\theta_q\left( \frac{q^{n-1} \lambda  }{1-q }  \frac{(1-q)z\prod_{l=1}^{n-1}{b_l}}{\prod_{l=1}^{n-1}{q}a_l}   \frac{\prod_{l=1}^{n-1}a_l}{  \prod_{l=1}^{n-1}b_l }    \right)}   \frac{{ z}^{\sum_{l=1}^{n-1}\log_qb_l  }}{ {  z }^{ \sum_{l=1}^{n-1}\log_qa_l }  } \nonumber \\
     & \lim_{q\rightarrow 1} \frac{ (1-q)^{ \sum_{l=1}^{n-1}\log_q b_l } }{  (1-q)^{ \sum_{l=1}^{n-1}\log_q a_l  } }\frac{\prod_{l=1}^{n-1}\theta_q\left( {b_l (1-q)  }\big/{ (q^{n-1}  \lambda )}   \right)}{\prod_{l=1}^{n-1}\theta_q\left( {  a_l(1-q) }\big/{ (q^{n-1} \lambda) }     \right)}  \frac{\theta_q\left(   { \lambda  z\prod_{l=1}^{n-1}{b_l}}\big/{\prod_{l=1}^{n-1}a_l}     \right)}{\theta_q\left( { \lambda  }  {z}       \right)}   \frac{{ z}^{\sum_{l=1}^{n-1}\log_qb_l  }}{ {  z }^{ \sum_{l=1}^{n-1}\log_qa_l }  } \nonumber \\
     =&\lim_{q\rightarrow 1}  \frac{ (1-q)^{ \sum_{l=1}^{n-1}\log_q b_l } }{  (1-q)^{ \sum_{l=1}^{n-1}\log_q a_l  } } \frac{\left({ ( 1-q) }/{  \lambda }\right)^{\sum_{l=1}^{n-1}\log_q \left(a_l\right)} }{\left({ ( 1-q) }/{  \lambda }\right)^{\sum_{l=1}^{n-1}\log_q \left(b_l\right)}} \frac{\left( \lambda z \right)^{\sum_{l=1}^{n-1}\log_q a_l}}{\left( \lambda z \right)^{\sum_{l=1}^{n-1}\log_q b_l}}    \frac{{ z}^{\sum_{l=1}^{n-1}\log_qb_l  }}{ {  z }^{ \sum_{l=1}^{n-1}\log_qa_l }  }   \nonumber  \\
     =&1.  \label{pf:lim2}
    \end{align}
      In \eqref{pf:lim2}  we take  $ \mathrm{arg}(\lambda)\in(-\pi,\pi)$  and  $\mathrm{arg}(z)\in  (-\mathrm{arg}(\lambda)-\pi,-\mathrm{arg}(\lambda)+\pi)$.
      \item  The formula \eqref{tildeflim} is derived from Lemma  \ref{lem:limf}, and  the formula \eqref{tildeflim2eq} is derived from \eqref{olbcconne},   \eqref{lim:zero}-\eqref{tildeflim}.
      \item  From \eqref{lim:zero}-\eqref{tildeflim2eq}, it follows that \eqref{eq: qconnefomu} recovers \eqref{olbcconne} as stated in this proposition.
\end{itemize}
\end{proof}

\subsection{Behavior of  Connection Matrix and  $q$-Stokes Matrix of the System \eqref{qdes} as $q\rightarrow 1$}\label{sec:syslim}
 The $q$-difference system \eqref{qdes}   converges    in the form     to the differential system
\begin{equation}\label{des}
    \frac{\mathrm{d}}{\mathrm{d}z}F(z)=\left(E_{n n}+\frac{A}{z}\right)F(z).
\end{equation}     
Similar to the equation \eqref{c h eq} case,   there is a  formal  fundamental solution  of  \eqref{des}   around infinity     (see  e.g., \cite{balser2008formal, turrittin1955convergent}):
    \begin{equation*}
     \hat{F}(z;E_{nn},A)=\hat{H}(z;E_{nn},A)z^{\delta^{(n-1)} (A)}\mathrm{e}^{E_{n n} z}, \ \text{with} \  \hat{H}(z;E_{nn},A)=\mathrm{Id}_n+\sum_{m=1}^{\infty}H_m z^{-m},
\end{equation*}
where 
\begin{equation*}
     \delta^{(n-1)}\left(A\right)_{ij}:=\left\{
          \begin{array}{lr}
             a_{ij},   & \text{if} \ \ 1\le i, j\le n-1, \ \text{or} \ i=j=n;  \\
           0, & \text{otherwise}.
             \end{array}
\right.
\end{equation*}
And under the nonresonant assumption
\begin{equation*}
    \lambda^{(n)}_i-\lambda^{(n)}_j\notin  \mathbb{Z}_{>0}, \ \text{for all} \ 1\leq i,j \leq n,
\end{equation*}
there is a convergent fundamental solution  of  \eqref{des}   around the origin (see  e.g., \cite{balser2008formal}):
\begin{equation*}
    F^{(0)}(z;E_{nn},A)=H^{(0)}(z;E_{nn},A)z^{A_n},
\end{equation*}
with $A_n=\operatorname{diag}\left(\lambda^{\left(n\right)}_1,...,\lambda^{\left(n\right)}_{n}\right)$.

The standard theory of Borel resummation   (see e.g., \cite{balser2008formal, loday2016divergent, Wasow})   states that 
\begin{prop}
    There are  unique meromorphic   solutions $F^{(\infty)}_{\pm}(z;E_{nn},A)=H^{(\infty)}_{\pm}(z;E_{nn},A) z^{{\delta^{(n-1)}(A)}}\mathrm{e}^{ E_{n n} z} $  of  \eqref{des},   with the asymptotic expansion
\begin{equation*}
   H^{(\infty)}_{\pm}(z;E_{nn},A) \sim \mathrm{Id}_n+\sum_{m=1}^{\infty}H_m z^{-m},\ \text{as} \ z \rightarrow \infty \text{ in } \mathrm{Sect}_{\pm}.
\end{equation*}
Here recall that  $\mathrm{Sect}_{+}=S(-\pi/2,3\pi/2)$ and $\mathrm{Sect}_{-}=S(-3\pi/2,\pi/2)$ (see the notation  $S\left(a,b\right)$ in Section \ref{sec:q-borel sum}).
\end{prop}

\begin{defn}
    Let $U_{\pm}( E_{nn},A   )$  be the connection matrices  of  \eqref{des}  defined by
\begin{equation}
F^{(0)}(z;E_{nn},A) =F^{(\infty)}_{\pm}( z;E_{nn},A   ) \cdot  U_{\pm}( E_{nn},A   ), \ \text{for}\ z\in \mathrm{Sect}_{\pm}.  \label{U+}   
% \\
% &F=\tilde{F}_-\cdot \tilde{U}_-,\   \text{for}\ z\in S(\pi/2,5\pi/2),
\end{equation}
The  Stokes matrices   $S_{\pm}(E_{nn},A)$ of \eqref{des}  are defined   by
\begin{align*}
   & F^{(\infty)}_{+}(z;E_{nn},A)=F^{(\infty)}_{-}(z;E_{nn},A) S_+(E_{nn},A),  \  \text{for} \  z\in  S(-\pi/2,\pi/2),\\
   & F^{(\infty)}_{-}(z \mathrm{e}^{-2\pi \mathrm{i}};E_{nn},A)=F^{(\infty)}_{+}(z;E_{nn},A) S_-(E_{nn},A),  \  \text{for} \  z\in  S(\pi/2,3\pi/2).
\end{align*}
\end{defn}

 We can also  diagonalize the upper left part of the equation \eqref{des} and obtain  
 \begin{equation}\label{desAn-1}
     \frac{\mathrm{d}}{\mathrm{d}z}F(z)=\left(E_{n n}+\frac{A_{n-1}}{z}\right)F(z),
 \end{equation}
 with $A_{n-1}$  given in Proposition \ref{An-1Pn-1}.
 
 The  (formal)   fundamental solution  of  \eqref{desAn-1}   around (infinity)     the   origin is given below.
\begin{prop}[see e.g.,  \cite{balser2008formal, turrittin1955convergent}]\label{diagFz}
    The system \eqref{desAn-1} has a     (formal)   fundamental solution     as follows.
        
         Around the origin,  there exists a convergent fundamental solution  of  \eqref{desAn-1}:
        \begin{equation}\label{eq:con sol}
    F^{(0)}\left(z;E_{nn},A_{n-1}\right)=H^{(0)}\left(z;E_{nn},A_{n-1}\right) \cdot z^{A_{n}},
\end{equation}
where  $A_{n}=\operatorname{diag}\left(\lambda^{\left(n\right)}_1,...,\lambda^{\left(n\right)}_{n}\right)$ and   $H^{(0)}\left(z\right)$ is an $n\times n$ matrix,  with $(i,j)$-th entry $H^{(0)}(z)_{i j}$ given by
\begin{align*}
	&H^{(0)}\left(z\right)_{i j}=\frac{a^{(n-1)}_i}{\lambda_{j}^{\left(n\right)}-\lambda_{i}^{\left(n-1\right)}} \cdot { }_{n-1} F_{n-1}\left(\begin{array}{c}
1+\lambda^{(n)}_j-\boldsymbol{\lambda^{(n-1)}}-\boldsymbol{\delta_{i}^{(n-1)}   } \\
1+\lambda^{(n)}_j-\boldsymbol{\lambda^{(n)}_{\hat{j}}}
\end{array} ; z\right), \ \text{for}\  1\leq i \leq n-1,\, 1\leq j \leq n, \\
	&H^{(0)}\left(z\right)_{n j}={ }_{n-1} F_{n-1}\left(\begin{array}{c}
1+\lambda^{(n)}_j-\boldsymbol{\lambda^{(n-1)}} \\
1+\lambda^{(n)}_j-\boldsymbol{\lambda^{(n)}_{\hat{j}}}
\end{array} ;  z\right), \ \text{for} \  1\leq j \leq n.
\end{align*}
Here  $\boldsymbol{\delta^{(n-1)}_i}=(\delta_{i1},...\delta_{i,n-1})$ and   we denote $\boldsymbol{\lambda^{(n-1)}}:=\left(\lambda^{(n-1)}_1,...,  \lambda^{(n-1)}_{n-1}  \right)$,  $\boldsymbol{\lambda^{(n)}_{\hat{j}}}:=\Big( \lambda^{(n)}_1,...,\widehat{\lambda^{(n)}_j},...,\lambda^{(n)}_n   \Big)$.      So 
\begin{align*}
   & 1+\lambda^{(n)}_j-\boldsymbol{\lambda^{(n-1)}}-\boldsymbol{\delta_{i}^{(n-1)}   } 
    =\left(  1+\lambda^{(n)}_j-\lambda^{(n-1)}_1-\delta_{i1},...,1+\lambda^{(n)}_j-\lambda^{(n-1)}_{n-1} -\delta_{i,n-1} \right),\\
    & 1+\lambda^{(n)}_j-\boldsymbol{\lambda^{(n)}_{\hat{j}}}
    =\left(  1+\lambda^{(n)}_j-\lambda^{(n)}_1,...,\widehat{1+\lambda^{(n)}_j-\lambda^{(n)}_{j}},...,1+\lambda^{(n)}_j-\lambda^{(n)}_{n}  \right),\\
    &1+\lambda^{(n)}_j-\boldsymbol{\lambda^{(n-1)}} =\left(  1+\lambda^{(n)}_j-\lambda^{(n-1)}_1,...,1+\lambda^{(n)}_j-\lambda^{(n-1)}_{n-1}  \right).
\end{align*}
And  around infinity,  there exists a formal   fundamental solution  of    \eqref{desAn-1}:
\begin{equation*}
     \hat{F}(z;E_{nn},A_{n-1})=\hat{H}(z;E_{nn},A_{n-1})z^{\delta^{(n-1)} (A_{n-1})}\mathrm{e}^{E_{n n} z}, 
\end{equation*}
and $\delta^{(n-1)} (A_{n-1})=\mathrm{diag}\left(a_{11},a_{22}...,a_{nn}\right)$.   Entries of $\hat{H}(z;E_{nn},A_{n-1})$  is given by
\begin{align*}
     &\hat{H}(z;E_{nn},A_{n-1})_{i j}=d_{ij}\cdot  { }_{n} F_{n-2}\left(\begin{array}{c}
\boldsymbol{\lambda^{(n)}   }-\lambda^{(n-1)}_{j}+1-\delta_{ij} \\
\boldsymbol{\delta^{(n-1)}_{i \hat{j}} }+\boldsymbol{\lambda^{(n-1)}_{\hat{j}}}-\lambda^{(n-1)}_{j}+1-\delta_{ij}
\end{array} ;-\frac{1}{z}\right) ,\  \text{for} \  1\leq i\leq n, 1 \le j \leq n-1,
\end{align*}
with
\begin{equation*}
    d_{ij}:=\left\{
          \begin{array}{lr}
       \frac{a_{i}^{(n-1)} b_{j}^{(n-1)} }{ \left(\lambda^{(n-1)}_{i}-\lambda^{(n-1)}_{j}+1\right) z  }      ,   & \text{if} \ \ 1\le i\neq   j\le n-1,  \\
           1, & \text{if} \ \ 1\le i=   j\le n-1,\\
           -\frac{b_{j}^{(n-1)}}{z},   & \text{if} \ \ i=n,\  1\le    j\le n-1.
             \end{array}
\right.
\end{equation*}
Here  we   denote  the index notations    $\boldsymbol{\lambda^{(n)}}:=\left(\lambda^{(n)}_1,...,\lambda^{(n)}_n   \right)$,  $\boldsymbol{\lambda^{(n-1)}_{\hat{j}}}:=\Big( \lambda^{(n-1)}_1,...,\widehat{\lambda^{(n-1)}_j},...,\lambda^{(n-1)}_{n-1}   \Big)$ , and  recall that  $\boldsymbol{ \delta^{(n-1)}_{i\hat{j}}   }=\Big(\delta_{i1},...,\widehat{\delta_{ij}},...,\delta_{i,n-1}\Big)$.  So  
\begin{align*}
   & \boldsymbol{\lambda^{(n)}   }-\lambda^{(n-1)}_{j}+1-\delta_{ij}=\left( \lambda^{(n)}_1   -\lambda^{(n-1)}_{j}+1-\delta_{ij},..., \lambda^{(n)}_n   -\lambda^{(n-1)}_{j}+1-\delta_{ij}  \right),\\
   &\boldsymbol{\delta^{(n-1)}_{i \hat{j}} }+\boldsymbol{\lambda^{(n-1)}_{\hat{j}}}-\lambda^{(n-1)}_{j}+1-\delta_{ij}\\
   =&\Big(
            \delta_{i1}+ \lambda^{(n-1)}_1-\lambda^{(n-1)}_j+1-\delta_{ij}, ...,\reallywidehatt{ \delta_{ij}+ \lambda^{(n-1)}_j-\lambda^{(n-1)}_j+1-\delta_{ij}},..., \delta_{i,n-1}+ \lambda^{(n-1)}_{n-1}-\lambda^{(n-1)}_j+1-\delta_{ij}
             \Big).
\end{align*}

\end{prop}
 Similar to Proposition \ref{conjugation},  we  have  
 \begin{prop}
     If  ${F}^{(\infty)}_{\pm}\left(z;E_{nn},A_{n-1}\right):=H^{(\infty)}_{\pm}(z;E_{nn},A_{n-1}) z^{{\delta^{(n-1)}(A_{n-1})}}\mathrm{e}^{ E_{n n} z}$  is  the fundamental solutions on $\mathrm{Sect}_{\pm}$  of  \eqref{desAn-1},  
then
\begin{equation}\label{eq:diag sol inf}
{F}^{(\infty)}_{\pm}\left(z;E_{nn},A_{n-1}\right)=\operatorname{diag}\left({P}_{n-1}^{-1},1\right){F}^{(\infty)}_{\pm}(z;E_{nn},A)\operatorname{diag}\left({P}_{n-1},1\right).
\end{equation}
   Moreover,  the connection matrices and  the Stokes matrices of the systems \eqref{des} and \eqref{desAn-1} are related by
\begin{align*}
&S_{\pm}\left(E_{nn}, A\right)=\operatorname{diag}\left({P}_{n-1},1\right)S_{\pm}\left(E_{nn}, A_{n-1}\right)\operatorname{diag}\left({P}_{n-1}^{-1},1\right),\\
&U_{\pm}\left(E_{nn}, A\right)=\mathrm{diag}\left({P}_{n-1},1\right)U_{\pm}\left(E_{nn}, A_{n-1}\right).
\end{align*} 
 \end{prop}
\begin{defn}
    For  $1\leq i \leq n$, $1\leq j \leq n-1$,        the  analytic continuation of the    function   $ {H}^{(\infty)}_{+}(z;E_{nn},A_{n-1})_{ij}$  defined on  $\mathrm{Sect}_{+}$  to   $\widetilde{C^*}$  is denoted as ${H}^{(\infty)}(z;E_{nn},A_{n-1})_{ij}$.
\end{defn}
Similar to  Proposition \ref{prop:hpm},  one knows that 
\begin{prop}\label{prop:Fclopm}
    For $1\leq i \leq n$,    $1\leq j \leq n-1$,  the entry ${H}^{(\infty)}_{\pm}(z;E_{nn},A_{n-1})_{ij}$  is given by
    \begin{align*}
%         &{}_{n} F_{n-2}\left(\begin{array}{c}
% \boldsymbol{\lambda^{(n)}}-\lambda^{(n-1)}_{j}+1-\delta_{ij} \\
% \boldsymbol{\delta^{(n-1)}_{i \hat{j}} }+\boldsymbol{\lambda^{(n-1)}_{\hat{j}}}-\lambda^{(n-1)}_{j}+1-\delta_{ij}
% \end{array} ;-\frac{1}{z}\right)\\
& {H}^{(\infty)}_{\pm}(z;E_{nn},A_{n-1})_{ij}\\
=&d_{ij}\sum_{k=1}^{n}\frac{\prod_{l=1,l\neq j}^{n-1}\Gamma( \delta_{il}+ \lambda^{(n-1)}_l-\lambda^{(n-1)}_j+1-\delta_{ij} )  \prod_{l=1,l\neq k}^n  \Gamma( \lambda^{(n)}_l   -\lambda^{(n)}_k  )  }{\prod_{l=1,l\neq k}^n\Gamma( \lambda^{(n)}_l   -\lambda^{(n-1)}_{j}+1-\delta_{ij}  )  \prod_{l=1,l\neq j}^{n-1} \Gamma( \delta_{il}+\lambda^{(n-1)}_l- \lambda^{(n)}_k      ) }\\
        &\times  {}_{n-1}F_{n-1}\left(\begin{array}{c} 
            1+\lambda^{(n)}_k-\delta_{i1}-\lambda^{(n-1)}_1,...,1+\lambda^{(n)}_k-\delta_{i,n-1}-\lambda^{(n-1)}_{n-1}  \\
              1+\lambda^{(n)}_k-\lambda^{(n)}_1,...,\widehat{1+\lambda^{(n)}_k-\lambda^{(n)}_k},...,1+\lambda^{(n)}_k-\lambda^{(n)}_n
      \end{array};z    \right)\\
        &\times  z^{ \lambda^{(n)}_k   -\lambda^{(n-1)}_{j}+1-\delta_{ij}  },
    \end{align*}
as $z\in \mathrm{Sect}_{\pm}$.    As a result,   {for}   $i=1,...,n$,  $j=1,...,n-1$,   we  have
\begin{equation*}
        {H}^{(\infty)}_{+}(z;E_{nn},A_{n-1})_{ij}={H}^{(\infty)}_{-}(z;E_{nn},A_{n-1})_{ij}, \  \text{as} \  z\in S(-\pi/2,\pi/2  ),
    \end{equation*}
    and  ${H}^{(\infty)}(z;E_{nn},A_{n-1})_{ij}$ is also the analytic continuation of   ${H}^{(\infty)}_-(z;E_{nn},A_{n-1})_{ij}$.
\end{prop}

For  simplicity, we focus on the diagonalized case   \eqref{desAn-1}.  The expressions of $U_{\pm}\left(E_{nn}, A_{n-1}\right),  S_-\left(E_{nn}, A_{n-1}\right)$  are given as follows.  The Stokes matrix $S_{+}\left(E_{nn}, A_{n-1}\right)$ case is similar  and    ignored here.
\begin{prop}\label{propUpmS+-}
    For $1\leq i \leq n-1$, and $1\leq j \leq n$,  entries of $U_{\pm}\left(E_{nn}, A_{n-1}\right)$  are given by
    \begin{align*}
   & (U_{ \pm  }(E_{nn}, A_{n-1}))_{i j}=\frac{a^{(n-1)}_i}{\lambda^{\left(n\right)}_j-\lambda^{\left(n-1\right)}_i}\frac{\prod_{l=1,l\ne j}^{n} \Gamma(1+\lambda_{j}^{\left(n\right)}-\lambda_{l}^{\left(n\right)})}{\prod_{l=1,l\ne i}^{n-1} \Gamma(1+\lambda_{j}^{\left(n\right)}-\lambda_{l}^{\left(n-1\right)})}\frac{\prod_{l=1,l\ne i}^{n-1} \Gamma(1+\lambda^{\left(n-1\right)}_i-\lambda^{(n-1)}_l)}{\prod_{l=1,l\ne j}^{n}\Gamma(1+\lambda^{(n-1)}_i-\lambda^{(n)}_l)}\frac{\mathrm{e}^{\pm\pi \mathrm{i}\lambda^{(n)}_j}}{\mathrm{e}^{\pm\pi \mathrm{i}\lambda^{(n-1)}_i}},\\
   &\left(U_+\left(E_{nn}, A_{n-1}\right)\right)_{nj}=\left(U_-\left(E_{nn}, A_{n-1}\right)\right)_{nj}=\frac{\prod_{l=1,l \ne j}^{n} \Gamma(1+\lambda_{j}^{\left(n\right)}-\lambda_{l}^{\left(n\right)})}{\prod_{l=1}^{n-1}\Gamma(1+\lambda_{j}^{\left(n\right)}-\lambda_{l}^{\left(n-1\right)})}.
\end{align*}
For convenience, here we denote
  $\lambda^{\left(n-1\right)}_n:=a_{n n}
=\sum_{l=1}^{n}\lambda^{\left(n\right)}_l-\sum_{l=1}^{n-1}\lambda^{\left(n-1\right)}_l$.   Moreover, the Stokes matrix $S_{-}\left(E_n,A_{n-1}\right)$ of  \eqref{desAn-1} is lower triangular with the form
\begin{equation}\label{eq:stokes-}
    {S_{-}\left(E_n,A_{n-1}\right)=\left(\begin{array}{cc}
    \mathrm{e}^{{ -2\pi \mathrm{i}A^{\left(n-1\right)}_{n-1}}} &  0 \\
    b_{-}^{\left(n\right)} & \mathrm{e}^{-2\pi \mathrm{i}{a_{nn}}}
  \end{array}\right),}
\end{equation}
where $A^{\left(n-1\right)}_{n-1}=\mathrm{diag}\left(\lambda^{\left(n-1\right)}_1,...,\lambda^{\left(n-1\right)}_{n-1}\right)$,       $b^{(n)}_{-}=\left(\big(b^{(n)}_{-}\big)_1,...,\big(b^{(n)}_{-}\big)_{n-1}\right)$,  with  
\begin{align*}
    \big(b^{(n)}_{-}\big)_j=-\sum_{k=1}^{n}  
      \frac{\prod_{l=1,l \ne k}^{n} \Gamma(1+\lambda_{k}^{\left(n\right)}-\lambda_{l}^{\left(n\right)})\prod_{l=1,l\neq j}^{n-1}\Gamma(  \lambda^{(n-1)}_l-\lambda^{(n-1)}_j+1 )  \prod_{l=1,l\neq k}^n  \Gamma( \lambda^{(n)}_l   -\lambda^{(n)}_k  )  }{\prod_{l=1}^{n-1}\Gamma(1+\lambda_{k}^{\left(n\right)}-\lambda_{l}^{\left(n-1\right)})\prod_{l=1,l\neq k}^n\Gamma( \lambda^{(n)}_l   -\lambda^{(n-1)}_{j}+1  )  \prod_{l=1,l\neq j}^{n-1} \Gamma( \lambda^{(n-1)}_l- \lambda^{(n)}_k      ) }\frac{b^{(n-1)}_j}{\mathrm{e}^{ 2\pi \mathrm{i}\lambda^{(n)}_k  }} .
\end{align*}
\end{prop}
\begin{proof}
   The explicit expressions of $U_{ \pm  }(E_{nn}, A_{n-1})$  can  be  found in \cite{lin2024explicit}.   The fact that $S_{-}\left(E_n,A_{n-1}\right)$ is lower triangular with the $i$-th diagonal entry being   $\mathrm{e}^{-2\pi\mathrm{i}\lambda^{(n-1)}_i}$ can be found in \cite[Chapter 9.1]{balser2008formal}.     Other entries   of $S_{-}\left(E_n,A_{n-1}\right)$  are derived from 
\begin{equation*}
     S_{-}\left(E_n,A_{n-1}\right)  U_{-}\left(E_n,A_{n-1}\right) =  U_{+}\left(E_n,A_{n-1}\right) \mathrm{e}^{-2\pi \mathrm{i}A_{n}},
\end{equation*} 
 which is obtained  
 from   \eqref{U+}   and  the following formula: 
\begin{equation}\label{U-entry}
    \left(U_-\left(E_{nn}, A_{n-1}\right)^{-1}\right)_{kj}=-\frac{b^{(n-1)}_j\prod_{l=1,l\neq j}^{n-1}\Gamma(  \lambda^{(n-1)}_l-\lambda^{(n-1)}_j+1 )  \prod_{l=1,l\neq k}^n  \Gamma( \lambda^{(n)}_l   -\lambda^{(n)}_k  )  }{\prod_{l=1,l\neq k}^n\Gamma( \lambda^{(n)}_l   -\lambda^{(n-1)}_{j}+1  )  \prod_{l=1,l\neq j}^{n-1} \Gamma( \lambda^{(n-1)}_l- \lambda^{(n)}_k      ) },
\end{equation}
for  $1\leq k \leq n$, $1\leq j \leq n-1$. Here the formula \eqref{U-entry}     is obtained from  Proposition \ref{prop:Fclopm}.
\end{proof}

The $q$-Stokes matrix $S_q\left(z,\lambda,\mu;E_{nn}, A_{n-1}\right)$   recovers    the     Stokes matrix  $S_-\left(E_{nn}, A_{n-1}\right)$   in the following sense. That is,
\begin{prop}
    As $q\rightarrow 1$,  it follows that  
    \begin{itemize}
        \item[(1)] The convergent fundamental solution $F^{(0)}_q\left(z;E_{nn},A_{n-1}\right)$   (see \eqref{eq:F0q(z)}) of \eqref{diagqdes}  converges to the convergent fundamental solution $F^{(0)}\left(z\right)$ (see \eqref{eq:con sol}) of \eqref{desAn-1}:
        \begin{equation}\label{eq:limF0floque}
            \lim_{q\rightarrow 1}  F^{(0)}_q\left(z;E_{nn},A_{n-1}\right)=F^{(0)}\left(z;E_{nn},A_{n-1}\right).
        \end{equation}

        \item[(2)]  The fundamental   solution   $F_q^{(\infty)}(z,\lambda;E_{nn},A_{n-1})$  (see  \eqref{eq:Fqinf(z)})   of  \eqref{diagqdes}  converges to the fundamental solution ${F}^{(\infty)}_{\pm}\left(z;E_{nn},A_{n-1}\right)$        (see \eqref{eq:diag sol inf})  of  \eqref{desAn-1}:
        \begin{align}
            &\lim_{q\rightarrow 1} F_q^{(\infty)}(z,\lambda;E_{nn},A_{n-1})_{ij}={F}^{(\infty)}\left(z;E_{nn},A_{n-1}\right)_{ij},  \  \text{if}  \  \  \mathrm{arg}(z)\in  (-\mathrm{arg}(\lambda)-\pi,-\mathrm{arg}(\lambda)+\pi  ),  \label{eqFqijliminf}  \\
            &\lim_{q\rightarrow 1} F_q^{(\infty)}(z,\lambda;E_{nn},A_{n-1})_{in}=\left\{
          \begin{array}{lr}
          {F}^{(\infty)}_{+}\left(z;E_{nn},A_{n-1}\right)_{in},   & \text{if} \ \  \mathrm{arg}(z)\in (0,-\mathrm{arg}(\lambda)+\pi  ), \\
                {F}^{(\infty)}_{-}\left(z;E_{nn},A_{n-1}\right)_{in} , & \text{if} \ \ \mathrm{arg}(z)\in (-\mathrm{arg}(\lambda)-\pi, 0  ),
             \end{array}
\right.  \label{eqFqnjliminf}
        \end{align}
 for $1\leq i\leq n$, $1\leq j \leq n-1$.       Here in \eqref{eqFqijliminf},   \eqref{eqFqnjliminf}  we  take   $\mathrm{arg}(\lambda)\in (-\pi,\pi)$. 
        \item[(3)] The connection matrix $U_q(z,\lambda;E_{nn},A_{n-1})$    has the  limit:
        \begin{align}
         &   \lim_{q\rightarrow 1}U_q(z,\lambda;E_{nn},A_{n-1})_{i j}   =   \left\{
          \begin{array}{lr}
      (U_{ +  }(E_{nn}, A_{n-1}))_{i j}    ,   & \text{if} \ \  \mathrm{arg}(z)\in (0,-\mathrm{arg}(\lambda)+\pi  ), \\
        (U_{ -  }(E_{nn}, A_{n-1}))_{i j}         , & \text{if} \ \ \mathrm{arg}(z)\in (-\mathrm{arg}(\lambda)-\pi, 0  ),
             \end{array}
\right. \label{eqUqlim1}  \\
         &   \lim_{q\rightarrow 1}U_q(z,\lambda;E_{nn},A_{n-1})_{n j}=  \left(U_+\left(E_{nn}, A_{n-1}\right)\right)_{nj}=\left(U_-\left(E_{nn}, A_{n-1}\right)\right)_{nj},\label{eqUqlim2}
        \end{align}
with   $\mathrm{arg}(z)\in  (-\mathrm{arg}(\lambda)-\pi,-\mathrm{arg}(\lambda)+\pi)$  in  \eqref{eqUqlim2}.  Here         
        in \eqref{eqUqlim1},  \eqref{eqUqlim2} we take  $ \mathrm{arg}(\lambda)\in(-\pi,\pi)$.

        \item[(4)] Set nonzero numbers $\lambda, \mu$  with $\lambda/\mu \notin q^\mathbb{Z}$,   and  we  take  $-\pi<\mathrm{arg}(\lambda),\mathrm{arg}(\mu)<\pi$.  Here          
         $z$  is chosen specifically to    satisfy
\begin{align}
& z \in \mathbb{C^*}, \quad z\notin- {\lambda}^{-1}\mathbb{R}_{>0}, \quad z\notin -\mu^{-1}\mathbb{R}_{>0},  \quad  \mathrm{Re}(z)<0,\quad \pi/2<\mathrm{arg}(z)<3\pi/2,  \label{q->1z} \\
   & -\mathrm{arg}(\mu)-\pi<\mathrm{arg}(z)-2\pi<0<\mathrm{arg}(z)<-\mathrm{arg}(\lambda)+\pi.\label{argqstokesq->1}
\end{align}
Then the  $q$-Stokes matrix $S_q\left(z,\lambda,\mu;E_{nn}, A_{n-1}\right)$ (see \eqref{eq:qstokesdiag})
   recovers  the  Stokes matrix  $S_-\left(E_{nn}, A_{n-1}\right)$  (see   \eqref{eq:stokes-}) by
   \begin{equation}
       \lim_{q\rightarrow 1} S_q\left(z,\lambda,\mu;E_{nn}, A_{n-1}\right) \mathrm{e}^{ -2\pi \mathrm{i}\delta^{(n-1)}_q( A_{n-1} ) }=S_-\left(E_{nn}, A_{n-1}\right).\label{recoverqstokes}
   \end{equation}  
    \end{itemize}
\end{prop}
  To prove Proposition \ref{prop:lim eq},  we  need the following lemma, whose proof is similar to Lemma \ref{lem:limf}.
  \begin{lemma}\label{lem:limf2}
      For $1\leq i \leq n$, $1\leq j \leq n-1$,    as $q\rightarrow 1$, we have the limit:
      \begin{equation*}
          \lim_{q\rightarrow 1}H^{(\infty)}_q(z,\lambda;E_{nn},A_{n-1})_{ij}={H}^{(\infty)}(z;E_{nn},A_{n-1})_{ij}.
      \end{equation*}
  \end{lemma}
\begin{proof}
    The proof contains the following parts:
    \begin{itemize}
        \item The formula \eqref{eq:limF0floque}  is  obtained from direct computations.
        \item The  formula \eqref{eqUqlim1}    is  derived from    Lemma \ref{basic lim}  and   the following  formula:
        \begin{align*}
           & \lim_{q\rightarrow 1}  (1-q)^{ \mu_j  - \nu_i   }  \frac{\theta_q\left(-q^{ {\mu_j -\nu_i   -\nu_n }} (1-q)z\right)}{ \theta_q\left(-q^{-\nu_n }  (1-q)z\right) }  z^{  \mu_j-\nu_i  }\\
            =&\lim_{q\rightarrow 1}(1-q)^{ \mu_j  - \nu_i   } \left(-q^{-\nu_n }  (1-q)z\right)^{ \nu_i-\mu_j  } z^{  \mu_j-\nu_i  }\\
            =&\mathrm{e}^{\pm\pi \mathrm{i}(\lambda^{(n-1)}_i-\lambda^{(n)}_j  )},  \  \text{for}  \   1\leq i \leq n-1,\  1\leq  j \leq n,
        \end{align*}
      where      the  positive    or   negative  sign   is   chosen based on whether      $\mathrm{arg}(z)\in ( -\mathrm{arg}(\lambda)-\pi, 0  )$  or  $\mathrm{arg}(z)\in (0, -\mathrm{arg}(\lambda)+\pi  )$.  And the  formula  \eqref{eqUqlim2}   is  derived from    Lemma \ref{basic lim}  and   the following  formula:
      \begin{align*}
     &\lim_{q\rightarrow 1}    (1-q)^{ \sum_{l=1}^{n-1}\nu_l-\sum_{l=1,l\neq j}^{n}\mu_l } \frac{\prod_{l=1,l\neq j}^{n}\theta_q((1-q)q^{1-\mu_l}/\lambda)}{\prod_{l=1}^{n-1}\theta_q((1-q)q^{1-\nu_l}/\lambda)}      	\frac{\theta_q\left(  q^{ \mu_j-\nu_n  }  \lambda z\right)}{\theta_q(\lambda  z )}z^{ \mu_j-\nu_n     }\\
     =&\lim_{q\rightarrow 1} (1-q)^{ \sum_{l=1}^{n-1}\nu_l-\sum_{l=1,l\neq j}^{n}\mu_l } \left( {(1-q)}/{\lambda}  \right)^{\sum_{l=1,l\neq j}^n\mu_l-\sum_{l=1}^{n-1}\nu_l} \left( \lambda z \right)^{\nu_n-\mu_j} z^{ \mu_j-\nu_n     }\\
     =&1,
      \end{align*}
      where we take  $ \mathrm{arg}(\lambda)\in(-\pi,\pi)$  and  $\mathrm{arg}(z)\in  (-\mathrm{arg}(\lambda)-\pi,-\mathrm{arg}(\lambda)+\pi)$.  Here recall that $\nu_n=\sum_{l=1}^n\mu_l-\sum_{l=1}^{n-1}\nu_l$ (see \eqref{eq:vn}).
            \item  The formula \eqref{eqFqijliminf} is derived from Lemma  \ref{lem:limf2}, and  the formula \eqref{eqFqnjliminf} is derived from Proposition  \ref{propUpmS+-},   \eqref{eq:limF0floque}-\eqref{eqFqijliminf} and \eqref{eqUqlim1}-\eqref{eqUqlim2}.
            \item             From  \eqref{eqFqijliminf},   \eqref{eqFqnjliminf}, since  $z$    satisfies \ \eqref{q->1z} and          $0<\mathrm{arg}(z)<-\mathrm{arg}(\lambda)+\pi$,       $F_q^{(\infty)}(z,\lambda;E_{nn},A_{n-1})$     converges to:
\begin{equation}
    \lim_{q\rightarrow 1} F_q^{(\infty)}(z,\lambda;E_{nn},A_{n-1})={F}^{(\infty)}_{+}\left(z;E_{nn},A_{n-1}\right).   \label{eqFqnjlim1}
\end{equation}
  And    since  $z$    satisfies \ \eqref{q->1z}    and  $-\mathrm{arg}(\mu)-\pi<\mathrm{arg}(z)-2\pi<0$,      $F_q^{(\infty)}(z,\mu;E_{nn},A_{n-1})$    converges to:
\begin{equation}
    \lim_{q\rightarrow 1} F_q^{(\infty)}(z,\mu;E_{nn},A_{n-1})={F}^{(\infty)}_{-}\left(z\mathrm{e}^{-2\pi \mathrm{i}};E_{nn},A_{n-1}\right)\mathrm{e}^{2\pi \mathrm{i}\delta^{(n-1)}(A_{n-1})}.   \label{eqFqnjlim2}
\end{equation} 
Since
\[
\lim_{q\rightarrow 1}\delta^{(n-1)}_q( A_{n-1} )=\delta^{(n-1)}( A_{n-1} ),
\]
if $z$ satisfies \eqref{q->1z} and \eqref{argqstokesq->1}, 
        \eqref{recoverqstokes}  is  obtained.  
    \end{itemize}
\end{proof}
\begin{rmk}
    The formula \eqref{recoverqstokes}  can be checked under the conditions \eqref{q->1z} and \eqref{argqstokesq->1}  from the explicit expression of $S_q\left(z,\lambda,\mu;E_{nn}, A_{n-1}\right)$ in  Theorem \ref{thm:diagqstokes}.
\end{rmk}

% ...........................................
% Bibliography Style
\bibliography{main.bib}
\bibliographystyle{abbrv}
\Addresses
\end{document}